%% file: main.tex
\documentclass{amsbook}

\input{preamble}

\begin{document}

\title{Axiomatic Theory of Independence Relations in Model Theory}

\author[C. d'Elb\'{e}e]{Christian d\textquoteright Elb\'ee}
\address{Mathematisches Institut der Universität Bonn\\
Office 4.004, Endenicher Allee 60\\ 53115 Bonn\\ Germany}
\urladdr{\href{http://choum.net/\textasciitilde chris/page\textunderscore perso/}{http://choum.net/\textasciitilde chris/page\textunderscore perso/}}

\date{\today}

% To do:
% \begin{enumerate}
%     \item Add the example of Yvon bossut: Take a model of the granger example: $(E,K)$ with $[.,.]: E\times E\to K$ and take $a,b$ outside such that $[a,b] = 1$ and every other product is $0$ (i.e. $a,b\in E^\bot$). Let $A = (\vect{Ea},K)$ and $B = (\vect{Eb},K)$. As $\indi K$ is given by 
%     \[(V_1,K_1) \indi K _{(V_0,K_0)} (V_2,K_2) \iff V_1\cap V_2 = V_0 \text{ and } K_1\indi\ACF_{K_0}\quad K_2\]
%     we see that $A\indi K_E B$ and as $K(A) = K(B) = K$ we have $A\indi K ^m _E B$. Let $t$ be transcendental over $K$ and consider $B' = (\vect{Eb}, K(t))$. Assume that $a'\equiv_B a$ and $a'\indi K ^m _E B'$. Then $a'\indi K _{E\cup\set{tb}}  t$. Now $[a',tb] = t[a',b] = t$ as the formula $[x,b] = 1$ is in $\tp(a/B)$. Observe also that the field $K(\acl(E\cup \set{tb})) = K$ since $b\in E^\bot$. As $a'\indi K _{E\cup\set{tb}}  t$, in the field sort we have $t\indi\ACF _K t$, a contradiction.
%     \item Maybe give the concrete examples of failure of right closure for dividing? At least the one in circular order?
% \end{enumerate}
\maketitle

\setcounter{tocdepth}{2}
%\hrulefill
\tableofcontents	
%\hrulefill

{\color{blue} 
Changes since last arXiv version are in BLUE.
\begin{enumerate}
    \item Exercise 63.
    \item Added Conant's reference \cite{conant2022separation}
    \item a whole new section: ``A second criterion for NSOP$_4$ theories".
    \item Removed the section ``A few words on Kim-Pillay style results".
    \item Added a section on ``naive monotonisation"
\end{enumerate}

Erratum:
\begin{enumerate}
    \item In Proposition 4.1.12 replaced $C$ by $F$
    \item In Lemma 2.2.8: removed the superfluous condition: ``$a\equiv_C a'$".
\end{enumerate}

}

\include{intro_bonn_version}

\include{justabunchofsillyaxioms}

\include{chapter1}

\include{chapter2}

\include{chapter3}

\include{chapter4}

\bibliographystyle{plain}
\bibliography{biblio}

\end{document}

%% file: preamble.tex
\usepackage{anysize}
\marginsize{3cm}{3cm}{3cm}{3cm}

\usepackage{amssymb,latexsym}
\usepackage{amsmath,amsthm}
\usepackage{amsfonts,mathrsfs}
\usepackage{mathtools}

\usepackage{tikz-cd}

\usepackage{todonotes}

\usepackage[all]{xy}
\usepackage{graphicx}

\usepackage{xcolor,url}
\usepackage{hyperref}
\hypersetup{
    colorlinks,
    citecolor=blue,
    filecolor=blue,
    linkcolor=black,
    urlcolor=blue
}

\usepackage[shortlabels]{enumitem}

\usepackage{halloweenmath}

\theoremstyle{plain}
\newtheorem{theorem}{Theorem}[section]

\newtheorem{corollary}[theorem]{Corollary}
\newtheorem{lemma}[theorem]{Lemma}
\newtheorem{proposition}[theorem]{Proposition}
\newtheorem{fact}[theorem]{Fact}

\newtheorem*{theorem*}{Theorem}
\newtheorem{example}[theorem]{Example}

\theoremstyle{definition}
\newtheorem{definition}[theorem]{Definition}

\theoremstyle{remark}
\newtheorem{remark}[theorem]{Remark}
\newtheorem{exercise}{Exercise}
\newtheorem{claim}{Claim}

\newcommand\Q{\mathbb{Q}}
\newcommand\N{\mathbb{N}}
\newcommand\Z{\mathbb{Z}}
\newcommand\F{\mathbb{F}}

\newcommand\K{\mathbb{K}}

\newcommand\MM{\mathbb{M}}

\def\cP{\mathcal{P}}

\def\cH{\mathcal{H}}

\def\cF{\mathcal{F}}
\def\cU{\mathcal{U}}

\def\Th{\operatorname{Th}}

\newcommand\LL{\mathscr{L}}

\newcommand\LLr{\mathscr{L}_{\mathrm{ring}}}

\newcommand{\ACF}{\mathrm{ACF}}

\newcommand{\acl}{\mathrm{acl}}
\newcommand{\dcl}{\mathrm{dcl}}
\newcommand{\cl}{\mathrm{cl}}

\newcommand{\Id}{\mathrm{Id}}

\newcommand{\alg}{\mathrm{alg}}

\newcommand{\tp}{\mathrm{tp}}

\newcommand{\dv}{\mathrm{div}}

\newcommand{\opp}{\mathrm{opp}}
\newcommand{\Graph}{\mathrm{Graph}}
\newcommand{\RG}{\mathrm{RG}}
\newcommand{\st}{\mathrm{st}}
\newcommand{\ld}{\mathrm{ld}}
\newcommand{\Sk}{\mathrm{Sk}}
\newcommand{\Av}{\mathrm{Av}}
\newcommand{\Circle}{\mathrm{Circle}}

\def\seq{\subseteq}

\newcommand{\set}[1]{\left\{ {#1} \right\}}
\newcommand{\vect}[1]{\langle {#1} \rangle}
\newcommand{\abs}[1]{\lvert {#1} \rvert}

\DeclareMathOperator{\Aut}{Aut}

\DeclareMathOperator{\trdeg}{Trdeg}
\DeclareMathOperator{\EM}{EM}
\DeclareMathOperator{\cf}{cf}
\DeclareMathOperator{\sq}{sq}

\newcommand{\kappabar}{\bar{\kappa}}

\definecolor{airforceblue}{rgb}{0.36, 0.54, 0.66}

\def\IND{\setbox0=\hbox{$x$}\kern\wd0\hbox to 0pt{\hss$\mid$\hss}
\lower.9\ht0\hbox to 0pt{\hss$\smile$\hss}\kern\wd0}
\def\NotIND{\setbox0=\hbox{$x$}\kern\wd0\hbox to 0pt{\mathchardef\nn=12854\hss$\nn$\kern1.4\wd0\hss}\hbox to 0pt{\hss$\mid$\hss}\lower.9\ht0 \hbox to 0pt{\hss$\smile$\hss}\kern\wd0}

\def\ind{\mathop{\mathpalette\IND{}}}
\def\nind{\mathop{\mathpalette\NotIND{}}}

\def\indi#1{\mathop{\mathpalette\IND{}^{\!\!\!\!\rlap{$\scriptstyle{#1}$}\,\,\,\,}}}

\def\nindi#1{\mathop{\mathpalette\NotIND{}^{\!\!\!\rlap{$\scriptstyle{#1}$}\,\,\,}}}

\makeatletter
\newcommand{\setword}[2]{%
  \phantomsection
  #1\def\@currentlabel{\unexpanded{#1}}\label{#2}%
}
\makeatother

\renewcommand{\models}{\vDash}

%% file: intro_bonn_version.tex
\chapter*{Introduction}

These notes originate from a neostability course held during the summer semester of 2023 at the University of Bonn. I am grateful to my students, who graciously endured my ramblings and persevered until the end of the term.

This course is an introduction to the fruitful links between model theory and a combinatoric of sets given by independence relations. An independence relation on a set is a ternary relation between subsets, usually denoted $\ind$. Properties, or axioms, satisfied by such a relation --sometimes related with an ambient closure operator, or an ambient first-order theory-- can be a witness of nice combinatorial behaviour, or tameness.  A motivating example of this phenomenon is the Kim-Pillay theorem, which is a characterisation of simple theories by the existence of an independence relation satisfying a certain set of axioms. We distinguish between three sorts of axioms for independence relations, those that can be stated in a purely set-theoretic framework (monotonicity, base monotonicity, transitivity) those that necessitate an ambient closure operator (closure, anti-reflexivity) and those that necessitate an ambient theory to be stated (extension, the independence theorem). For clarity, those axioms will be stated separately, but a complete list of axioms is given on p. 6.

Chapter \ref{chapter:1} should be considered as an introductory chapter. It does not mention first-order theories nor formulas and introduces independence relations in a naive set theory framework. Its main goal is to get the reader familiar with the basic axioms of independence relations (those that do not need an ambient theory to be stated) as well as introducing closure operators and pregeometries. 

Chapter \ref{chapter:2} introduces the model-theoretic context. The two main examples (algebraically closed fields and the random graph) are described as well as independence relations in those examples. 

Chapter \ref{chapter:3} gives the axioms of independence relations in a model-theoretic context and introduces what I believe is nowadays called ``forking calculus" (although forking will be defined in Chapter $4$). It introduces the general toolbox of the model-theorists (indiscernible sequences, Ramsey/Erd\"os-Rado and compactness) as well as the independence relations $\indi h$ and $\indi u$ of heirs/coheirs with two main applications: Adler's theorem of symmetry (how symmetry emerges from a weaker set of axioms, which is rooted in the work of Kim and Pillay \cite{KP97}) and a criterion for NSOP$_4$ using stationary independence relations in the style of Conant \cite{gabefreeamalgamation}. Independence relations satisfying Adler's theorem of symmetry are here called \textit{Adler independence relations} (or AIR).

Chapter \ref{chapter:4} treats forking and dividing as abstract independence relations $\indi f$, $\indi d$. It is proved that $\indi d$ is always stronger than any AIR (even though it is not an AIR in general) as well as an abstract connection between the independence theorem and forking independence, which holds in all generality and is based on Kim-Pillay's approach.
Then, simplicity is defined and the interesting direction of the Kim-Pillay theorem (namely that the existence of an Adler independence relation satisfying the independence theorem yields simplicity) is deduced from earlier results.

% It should be understood that this course is taking the point of view of independence relations in model theory. As such, many theorems are not optimal in ths e

The main inspirations for this course are as follows: \cite{TZ12} and \cite{C11} for basics and simplicity theory, Adler's thesis \cite{adlerthesis} for axiomatic theory theory of independence relations and \cite{CR16,KR17} (and subsequent works on NSOP$_1$ theories) for the modern exposition of independence theories. I believe the complete list of prerequisites in model theory for this course is the following: languages, sentences, theories, formulas, types, structures, definable sets, substructures, elementary substructures, models, elementary maps, elementary bijections and automorphism of models. As for set theory, I only assume the existence and basic properties of cardinals and ordinals.\\

% \noindent \textbf{Ackowledgements.} I am very grateful to my students who relentlessly came to listen to me blabbering about independence relations, when axioms are to be used on the left rather than the right. 
% In the end, it is all about compactness and using an automorphism. Special thanks to Frederick for Example \ref{exo:frederickgebert}, Thomas for many good comments, Remark \ref{rk:campenhausen} and Exercise \ref{exo:vancampenhausenDLO} and Ylong for various discussions, notably on Remark \ref{rk:NSOPnNSOPn+1}.

%% file: chapter1.tex
\chapter{Independence relations in a set-theoretic setting}\label{chapter:1}

\section{Independence relations and notations}
In this course, the symbol $\ind$ (sometimes indexed $\indi 0, \indi i$, etc...) will always denote a ternary relation on the powerset of an ambient set. We will sometimes call $\ind$ an \textit{independence relation} \footnote{Unlike other authors \cite{A09}, we do not reserve the denomination ``independence relation" for certain ternary relation satisfying a fixed set of axioms, we rather use it freely for ternary relations that will satisfy some axioms.}. We start with an easy set of axioms, where $A,B,C, \ldots$ are subsets of the ambiant set. We often denote by juxtaposition $AB$ the union $A\cup B$.

\begin{definition}[Axioms of independence relations, part 1]\label{def:axioms_part1}~
\begin{enumerate}[$(1)$]
\item (\setword{finite character}{FIN}) If $a\ind_C B$ for all finite $a\seq A$, then $A\ind_C B$.
  \item (\setword{existence}{EX}) $A\ind_C C$ for any $A$ and $C$.
  \item (\setword{symmetry}{SYM}) If $A\ind_C B$ then $B\ind_C A$.
  \item (\setword{local character}{LOC}) For all $A$ there is a cardinal $\kappa = \kappa(A)$ such that for all $B$ there is $B_0\seq B$ with $\abs{B_0} <\kappa$ with $A\ind_{B_0} B$.
  \item (right \setword{normality}{NOR}) If $A\ind_C B$ then $A\ind_C BC$.
\item (right \setword{monotonicity}{MON}) If $A\ind_C BD$ then $A\ind_C B$.
\item (right \setword{base monotonicity}{BMON}) Given $C\seq B\seq D$ if $A\ind_C D$ then $A\ind_{B} D$.
  \item (right \setword{transitivity}{TRA}) Given $C\seq B\seq D$, if $A\ind_{C} B$ and $A\ind_B D$ then $A\ind_C D$.
\end{enumerate}
\end{definition}

Every property with a ``right ---" prefix has a symmetric counterpart ``left ---". For instance, left \ref{NOR} is $A\ind_C B\implies AC\ind _C B$. We will often omit the left/right prefix when the context is clear, for instance if the relation is symmetric or if it is clear which side we are refering to.\\

\noindent\textbf{Some terminology}. If $A\ind_C B$, we say that \textit{$A$ is independent from $B$ over $C$}, or \textit{$A$ and $B$ are independent over $C$}. We call $C$ the \textit{base set} of the instance $A\ind_C B$. Given two independence relations $\ind,\indi 0$, we say that $\ind$ is \textit{stronger than $\indi 0$}, or $\indi 0$ is \textit{weaker than $\ind$}, denoted $\ind\to\indi 0$, if $\ind\seq \indi 0$, in other words, for all $A,B,C$ we have $A\ind_C B\implies A\indi 0 _C B$.

\begin{exercise}
Let $S$ be any set and define the relation $\ind$ by $A\ind_C B$ if and only if $A\cap B\seq C$. Prove that $\ind$ satisfies all properties above.
\end{exercise}

\begin{exercise}
    Check that $\ind$ satisfies right \ref{MON} if and only if for all $A,C$ and $B\seq D$, we have $A\ind_C D\implies A\ind_C B$.
\end{exercise}

\begin{exercise}
    Prove that if $\ind$ satisfies (right) \ref{NOR}, (right) \ref{MON} and (right) \ref{TRA}, then $\ind$ satisfies the following stronger version of \ref{TRA}:
    \[A\ind_C B\text{ and } A\ind_{BC} D \implies A\ind_C BD.\]
\end{exercise}

\begin{exercise}
    Prove that if $\ind$ satisfies (right) \ref{NOR}, (right) \ref{MON} and (right) \ref{BMON}, then $\ind$ satisfies the following stronger version of \ref{BMON}:
    \[A\ind_C BD \implies A\ind_{CD} B.\]
\end{exercise}

\begin{exercise}\label{exo:loccharactereasy}
    Assume that $\ind$ satisfies \ref{FIN}, \ref{BMON} and the following weak version of locality: if $a$ is finite, for all $B$ there exists a finite $b\seq B$ with $a\ind_b B$. Prove that $\ind$ satisfies \ref{LOC}.
\end{exercise}

\begin{exercise}
    If $\ind$ satisfies \ref{LOC} and \ref{BMON} then $\ind$ satisfies \ref{EX}.
\end{exercise}

\section{Independence relations in pregeometries} We now define a purely set-theoretic and combinatorial context where two independence relations appear very naturally. 

\begin{definition}
    A \textit{(finitary) closure operator} $\cl$ on a set $S$ is a function:
    \[\cl : \cP(S)\rightarrow \cP(S)\]
    which satisfies the following properties:
    \begin{enumerate}
        \item (Reflexivity) $A\seq \cl(A)$;
        \item (Monotonicity) if $A\seq B$ then $\cl(A)\seq \cl(B)$;
        \item (Transitivity) $\cl(\cl(A)) = \cl(A)$;
        \item (Finite Character) $\cl(A) = \bigcup_{ \text{finite}\ a\seq A} \cl(a)$.
    \end{enumerate}
    If $\cl$ also satisfies:
    \begin{enumerate}\setcounter{enumi}{4}
        \item (Exchange) if $a\in \cl(Ab)\setminus \cl(A)$ then $b\in \cl(Aa)$.
    \end{enumerate}
    then $(S,\cl)$ is called a \textit{pregeometry}.
\end{definition}

\begin{example}
Let $(G,+,0)$ be an abelian group. For $A\seq G$, we denote by $\vect{A}$ the group span by $A$ and $\vect{A}^\dv $ the divisible closure of $\vect{A}$ in $G$, that is
\[\vect{A}^\dv = \set{g\in G\mid ng\in \vect{A} \text{ for some $n\in \N$}}\]
For any abelian group $(G,+,0)$, the map $A\mapsto \vect{A}$ is a finitary closure operator. It does not satisfy exchange in general (think of $1$ and $2$ in $\Z$). If $G$ is divisible, then $A\mapsto \vect{A}^\dv$ is a pregeometry.
\end{example}

\begin{exercise}
    Assume that $\cl$ is a finitary closure operator on a set $S$. Let $B\seq S$. We define:
    \begin{itemize}
        \item $\cl_B(A):= \cl(A\cup B)$ for all $A\seq S$ (\textit{Relativisation});
        \item $\cl_{\upharpoonright B} (A) = \cl(A)\cap B$ for all $A\seq S$ (\textit{Restriction}).
    \end{itemize}
    Prove that $\cl_B$ (resp. $\cl_{\upharpoonright B}$) is a finitary closure operator on $S$ (resp. on $B$). Similarly, if $\cl$ is a pregeometry, so are $\cl_B$ and $\cl_{\upharpoonright
    B}$.
\end{exercise}

\begin{definition}
    Given a closure operator $\cl$ on a set $S$ we define $A\indi a_C B$ if $\cl(AC)\cap \cl(BC) = \cl(C)$.
\end{definition}

\begin{exercise}\label{exo:propofindia}
    Check that if $\cl$ is a finitary closure operator, then $\indi a$ satisfies all properties of Definition \ref{def:axioms_part1} except \ref{BMON}.
\end{exercise}

\begin{exercise}\label{exo:pregeom_BMON_algebraic_dim1}
    Assume that $\cl$ is a pregeometry on $S$. Prove that if $a\notin \cl(B)$, then $a\indi a _C B$ for all $C\seq B$. %{\color{red} This is false if $\cl$ is not a pregeometry, c-ex: in $(\N,\cl)$ as after, $3\notin\cl(2)$ but $2\in \cl(2)\cap \cl(3)$ so we have $3\nindi a _{\cl(1)} 2$.}
\end{exercise}

\begin{definition}
    Let $\cl$ be a closure operator on $S$. Let $A_0, A, B\seq S$.
\begin{itemize}
    \item We say that $A_0$ is \textit{independent} over $B$ if $a\notin \cl(B\cup (A_0\setminus\set{a}))$ for all $a\in A_0$.
    \item We say that $A_0$ \textit{generates} $A$ over $B$ if $A\seq \cl(BA_0)$.
\end{itemize}
\end{definition}

\begin{exercise}
    Prove that $A$ is independent over $B$ for $\cl$ if and only if $A$ is independent over $\emptyset $ for $\cl_B$. Prove that $A_0$ generates $A$ over $B$ if and only if $A_0$ generates $AB$ over $\emptyset$ for $\cl_B$.
\end{exercise}

\begin{proposition}\label{prop:pregeom_maximal_independent}
Let $\cl$ be a closure operator on $S$. Let $A_0\seq A, B\seq S$ such that $A_0$ is independent over $B$. Then, there exists a set $A_1\seq \cl(A)$ containing $A_0$ which is maximal among the subsets of $A$ containing $A_0$ that are independent over $B$.
\end{proposition}
\begin{proof}
First, by working with $\cl_B$, we may assume that $B = \emptyset$. 
Let $\cF$ be the set of independent (over $\emptyset$) subsets of $A$ containing $A_0$. Let $(X_n)_n$ be a chain of elements in $\cF$ and $X = \bigcup_n X_n$. We check that $X$ is independent. If $x\in X$, then $x\in X_n$ for cofinitely many $n$. If $x\in \cl(X\setminus \set{x})$ then by finite character, $x\in \cl(X_n\setminus\set{x})$ for cofinitely many $n$, hence there is $n$ such that $x\in X_n$ and $x\in \cl(X_n\setminus\set{x})$, which contradicts that $X_n$ are independent. By Zorn's lemma, $\cF$ admits a maximal element. 
\end{proof}

\begin{proposition}[and Definition]\label{prop:pregeom_existence_basis}
Let $\cl$ be a pregeometry on $S$. Let $A_0,A, B\seq S$ with $A_0$ independent over $B$. Then, there exists a set $A_1\seq \cl(A)$ containing $A_0$ such that $A_1$ is independent over $B$ and $A_1$ generates $A$ over $B$.
Such $A_1$ is called a \textit{basis of $A$ over $B$}. For all bases $A_1,A_2$ of $A$ over $B$ we have $\abs{A_1} = \abs{A_2}$. We call $\abs{A_1}$ the dimension of $A$ over $B$, denoted $\dim(A/B)$.
\end{proposition}

\begin{proof} Similarly as above, if $(S,\cl)$ is a pregeometry, then so is the closure operator $\cl_B$ on $S$. This implies that we may assume that $B = \emptyset$. We start with a claim.
\begin{claim}\label{claim:pregeom_extension_independentset} If $X$ is independent and $a\notin \cl(X)$. Then $X\cup\set{a}$ is again independent\end{claim}
\begin{proof}[Proof of the claim]
Otherwise, either $a\in\cl(X)$ which is excluded or $x\in \cl(Xa\setminus\set{x})$ for some $x\in X$. By assumption $x\notin \cl(X\setminus \set{x})$ hence \[x\in \cl((X\setminus\set{x})a)\setminus \cl(X\setminus \set{x}).\]
By exchange, $a\in \cl(X\setminus\set{x}\cup \set{x}) = \cl(X)$, a contradiction. 
\end{proof}
Let $A_1$ be a maximal independent subset of $A$ extending $A_0$. If $A_1$ does not generates $A$, there exists $a\in A\setminus \cl(A_1)$. By Claim \ref{claim:pregeom_extension_independentset}, $A_1\cup\set{a}$ is again independent, which contradicts the maximality of $A_1$. 

For the second part, it is enough to show that if $X$ is independent in $A$ and $Y$ is generating in $A$, then $\abs{X}\leq \abs{Y}$. 
\begin{claim}\label{claim:pregeom_exchangebasis}
    If $x\in X\setminus Y$ there exists $y\in Y\setminus X$ such that $(X\setminus\set{x})\cup\set{y}$ is independent.
\end{claim}
\begin{proof}[Proof of the claim]
    As $X$ is independent, $x\notin\cl(X\setminus\set{x})$ hence $Y$ is not included in $\cl(X\setminus\set{x})$, otherwise $\cl(X\setminus\set{x})$ would be all $A$ by transitivity and the fact that $Y$ is generating. Then for any $y\in Y\setminus \cl(X\setminus\set{x})$ we have $(X\setminus\set{x})\cup \set{y}$ is independent by Claim \ref{claim:pregeom_extension_independentset}.
\end{proof}
Assume that $X$ is finite. Then using iteratively Claim \ref{claim:pregeom_exchangebasis}, we can replace elements of $X$ by elements of $Y$ to get an independent family, hence there is a subset of $Y$ with the same cardinality as $X$ which is independent, hence $\abs{X}\leq \abs{Y}$. 

If $X$ is infinite, let $X'$ be a basis of $A$ extending $X$. Then by finite character, for every $y\in Y$ there is a finite tuple $x_y$ from $X'$ such that $y\in \cl(x_y)$. We have $Y\seq \bigcup_{y\in Y} x_y$ hence $\bigcup_{y\in Y} x_y$ is a generating subset of $X'$. Assume that $\bigcup_{y\in Y} x_y\subsetneq X'$, then there exists $x\in X'\setminus \bigcup_{y\in Y} x_y$. Hence $\bigcup_{y\in Y} x_y\seq X'\setminus\set{x}$. As $X'$ is independent, $X'\setminus\set{x}$ is not a generating set (since $x\notin\cl(X'\setminus\set{x})$). This contradicts that $\bigcup_{y\in Y} x_y$ is generating. It follows that $X' = \bigcup_{y\in Y} x_y$. As each $x_y$ is finite, this implies that $Y$ is infinite and hence $\abs{\bigcup_{y\in Y} x_y }\leq \abs{Y}$. We conclude that $\abs{X}\leq \abs{X'} =\abs{\bigcup_{y\in Y} x_y}\leq  \abs{Y}$.
\end{proof}

\begin{exercise}
    Let $\cl$ be a closure operator on $S$. Let $A$ be a finite subset of $S$, prove that there exists a finite set $A_0\seq \cl(A)$ which is a basis of $A$. 
\end{exercise}

The following example shows that the previous exercise does not generalise to arbitrary sets $A$.
\begin{example}\label{exo:frederickgebert}
    If $(S,\cl)$ is a closure operator which is not a pregeometry, then infinite subsets may not admit a basis. Here is an example cooked up by Frederick Gebert. Take $S = \N$ and $\cl(A) = [0,\max A]$. Then $\cl$ is a closure operator on $\N$. It is easy to check that $(\N,\cl)$ is not a pregeometry and that $A\seq \N$ is an independent set if an only if $A = \set{a}$. Therefore $\N$ do not admit a basis. In this example, every finite set $A_0$ admit a basis: $\max A_0$.
\end{example}

We denote $\dim(A) = \dim(A/\emptyset)$. 
\begin{lemma}\label{lm:pregeom_dimension}
    Let $(S,\cl)$ be a pregeometry and $A,B\seq S$. Then
    \begin{enumerate}
        \item $\dim(AB) = \dim(A/B)+\dim(B)$;
        \item If $C\seq B$ then $\dim(A/B)\leq \dim(A/C)$;
        \item $\dim(AB)+\dim(A\cap B)\leq \dim A + \dim B$, for $A,B$ closed.
    \end{enumerate}
\end{lemma}

\begin{proof}
\textit{(1)} Let $B_0$ be a basis of $B$ over $\emptyset$ using Proposition \ref{prop:pregeom_existence_basis}. Using again Proposition \ref{prop:pregeom_existence_basis}, let $A_0$ be a basis of $AB$ over $B$. We check that $A_0B_0$ is a basis of $AB$ over $\emptyset$. As $A_0$ is independent over $B_0$, in particular, $A_0\cap \cl(B) = \emptyset$ hence $A_0\cap B_0 = \emptyset$. Let $x\in A_0\cup B_0$. If $x\in A_0$, then $x\notin \cl(B\cup (A_0\setminus\set{x})) = \cl(A_0B_0\setminus\set{x})$ as $A_0\cap B_0 = \emptyset$ hence we are done. If $x\in B_0$ and $x\in \cl(A_0\cup (B_0\setminus \set{x}))$. Let $(a_1,\ldots,a_n)$ be a tuple of elements of $A_0$ of minimal size such that $x\in \cl(\set{a_1,\ldots,a_n}\cup (B_0\setminus \set{x}))$. Thus $x\notin \cl(\set{a_2,\ldots,a_n}\cup (B_0\setminus \set{x}))$. By exchange, $a_1\in \cl(\set{a_2,\ldots,a_n}\cup B_0)$, which contradicts the fact that $a_i$ are elements of $A_0$. It follows that $\dim(AB) = \abs{A_0}+\abs{B_0}$. By definition, $\abs{B_0} = \dim B$ and $\dim(A/B) = \abs{A_0}$ so we are done. $(2)$ Easily follows from (1), but it is also clear by definition. $(3)$ Let $C = A\cap B$ and $C_0$ a basis of $C$. Let $A_0$ a basis of $A$ over $C$ and $B_0$ a basis of $B$ over $C$. Then certainly $A_0B_0C_0$ is a generating set for $AB$. $\dim(AB)\leq \abs{A_0}+\abs{B_0}+\abs{C_0}$. As $\abs{C_0} = \dim C$, $\abs{A_0}+\abs{C_0} = \dim A$, $\abs{B_0}+\abs{C_0} = \dim B$, we get $\dim(AB)\leq \dim A+\dim B - \dim C$, as requested. 
\end{proof}

\begin{definition}
    In a pregeometry, we denote $A\indi\cl_C B$ if for all finite subsets $A_0\seq A$ we have $\dim(A_0/BC) = \dim(A_0/C)$.
\end{definition}

Equivalently, $A\indi\cl_C B$ if and only if every finite $A_0\seq A$ which is independent over $C$ stays independent over $BC$.

\begin{exercise}\label{exo:relativised_independence_relation}
    Prove that for all $A,B,C $ $A\indi\cl_C B$ if and only if $A\indi{\cl_C}_\emptyset B$.
\end{exercise}

\begin{definition}[Axioms of independence relations, part 2]\label{def:axioms_part2} With an ambient closure operator $\cl$.
\begin{enumerate}\setcounter{enumi}{8}
    \item (\setword{anti-reflexivity}{AREF}) If $a\ind_C a$ then $a\in \cl(C)$;
    \item (right \setword{closure}{CLO}) $A\ind_C B \implies  A\ind _C \cl(B)$.
    \item (\setword{strong closure}{SCLO}) $A\ind_C B \iff  \cl(AC)\ind _{\cl(C)} \cl(BC)$.
\end{enumerate}
\end{definition}

\begin{remark}
    Note that if $\ind$ satisfies two-sided \ref{MON}, \ref{AREF} and \ref{SCLO} then $\ind\rightarrow \indi a$. 
\end{remark} 

\begin{exercise}
    Prove that if $\ind$ satisfies \ref{SYM}, \ref{NOR}, \ref{BMON} and \ref{CLO}, then $\ind$ satisfies the ``$\Rightarrow$" direction of \ref{SCLO}. What about the ``$\Leftarrow$" direction?
\end{exercise}

\begin{exercise}
    Check that if $\cl$ is a finitary closure operator, then $\indi a$ satisfies further \ref{AREF} and \ref{SCLO}.
\end{exercise}

\begin{theorem}\label{thm:pregeom_axiomsfordimindependence}
    Let $(S,\cl)$ be a pregeometry. Then, $\indi\cl$ satisfies \ref{SYM}, \ref{FIN}, \ref{EX},  \ref{NOR}, \ref{MON}, \ref{BMON}, \ref{TRA}, \ref{AREF}, \ref{CLO}, \ref{SCLO}, \ref{LOC}.
\end{theorem}

\begin{proof}
(\ref{FIN} left and right). Left \ref{FIN} is clear by definition. For right \ref{FIN}, assume that $A\nindi\cl_C B$ then there is a finite subset of $A$ which is independent over $C$ and which is not independent over $BC$. In particular, it is not independent over $B_0C$ for a finite subset $B_0\seq B$, by finite character, hence $A\nindi\cl _C B_0$. 
(\ref{MON} left and right) Left \ref{MON} is clear by definition. For right \ref{MON}, observe that for any finite $A$ we have $\dim(A/C) \geq \dim(A/BC)\geq \dim(A/BDC)$ hence if $\dim(A/C) =\dim(A/BDC)$ then $\dim(A/C) = \dim(A/BC)$. It follows that $A\indi\cl_C BD\implies A\indi\cl_C B$. 
    (\ref{SYM}) By Exercise \ref{exo:relativised_independence_relation}, we may assume that $C = \emptyset$ and by two-sided \ref{FIN} and \ref{MON} that $A,B$ are finite. By Lemma \ref{lm:pregeom_dimension}, we have 
    \begin{align*}
        \dim(AB) &= \dim(A/B)+\dim(B)\\
        &=\dim(B/A)+\dim(A)
    \end{align*}
    Hence, if $A\indi\cl B$, then $\dim(A/B) = \dim(A)$, so $\dim(B/A) = \dim(B)$ and $B\indi\cl A$.
 (\ref{EX}, \ref{NOR}) are trivial. Note that by left finite character we may assume that $A$ is finite in what follows. (\ref{BMON}) Similarly as for \ref{MON}, $\dim(A/C) \geq \dim(A/CD)\geq \dim(A/BDC)$ hence if $\dim(A/C)= \dim(A/BDC)$ then $\dim(A/CD) = \dim(A/BCD)$. It follows that $A\indi\cl_C BD\implies A\indi\cl_{CD} B$. (\ref{TRA}) Assume that $A\indi\cl_C B$ and $A\indi\cl _B D$, with $C\seq B\seq D$. Then $\dim(A/C) = \dim(A/B)$ and $\dim(A/B) = \dim(A/D)$, hence $\dim(A/C) = \dim(A/D)$ so $A\indi\cl_C D$. (\ref{AREF}) If $\dim(a/Ca) = \dim(a/C)$, then $\dim(a/C) = 0$, so $\cl(C) = \cl(Ca)$ i.e. $a\in \cl(C)$. (\ref{CLO}, \ref{SCLO}) We prove only \ref{SCLO} as \ref{CLO} is proved similarly. By definition, for any $A,D$, $\dim(A/D) = \dim(\cl(AD)/\cl(D))$ hence $\dim(A/C) = \dim(A/BC)$ if and only if $\dim(\cl(AC)/\cl(C)) = \dim(\cl(ABC)/\cl(BC)) = \dim(\cl(AC)/\cl(BC))$, so \[A\indi\cl_C B\iff \cl(AC)\indi\cl_{\cl(C)} \cl(BC).\]
 
    (\ref{LOC}) Assume that $A$ is finite and $B$ is any set. Let $X$ be a basis of $A$ over $B$. For each singleton $a\in A$, there exists a finite set $b_a\seq B$ such that $a\in \cl(Xb_a)$. Let $B_0 = \bigcup_{a\in A}b_a$. Then $X$ is a basis of $A$ over $B_0$, hence $\dim(A/B) = \dim(A/B_0)$, so $A\indi\cl_{B_0} B$. We proved that: for all finite $A$ and any $B$, there exists a finite $B_0\seq B$ such that $A\indi\cl_{B_0} B$. We conclude by Exercise \ref{exo:loccharactereasy}.
\end{proof}

\begin{exercise}
    Prove that $A\indi \cl_C B$ if and only if all subsets
$A_0 \seq A$ and $B_0 \seq B$ which are both independent over $C$ are disjoint and their
union is again independent over $C$. 
\end{exercise}

Of course, we have $\indi \cl\rightarrow \indi a$. In terms of properties they satisfy, the two relations $\indi\cl$ and $\indi a$ only differ by the property \ref{BMON}. Recent results in model theory make apparent a tension associated with the presence or the absence of the property \ref{BMON} for different notions of forking. As we will see later it is the turning point between simple and NSOP$_1$ theories when considering the independence relation associated to Kim-forking. In a pregeometry, $\indi a$ and $\indi \cl$ are ``close" to each other, as we will see now, and forcing the \ref{BMON} axiom on $\indi a$ preserves most of the axioms mentioned above. This is not true in general, as we will see later. 

\subsection{Forcing the right base monotonicity axiom}

\begin{definition}\label{def:monotonisation}
    Let $\ind$ be a ternary relation. We associate the \textit{monotonisation $\indi M$ of $\ind$} which is defined as the following:
    \[A\indi M_C B\iff A\ind_X B \text{ for all $X$ with $C\seq X\seq \cl(BC)$}.\]
\end{definition}

\begin{proposition}\label{prop:monotonisation}
  The relation $\indi M $ satisfies right \ref{BMON}. \begin{itemize}
      \item If $\ind$ satisfies left or right \ref{MON}, left or right \ref{CLO}, left \ref{NOR}, so does $\indi M$. If $\ind$ further satisfies right \ref{NOR} or left \ref{TRA}, then so does $\indi M$. 
      \item If $\ind$ satisfies \ref{AREF} or (left or right) \ref{CLO} then so does $\indi M$.
      \item If $\indi 0\to \ind$ and $\indi 0$ satisfies the right-sided instance of: \ref{NOR}, \ref{MON}, \ref{CLO} and \ref{BMON}, then $\indi 0\rightarrow \indi M$.
  \end{itemize}
\end{proposition}
% finite character and existence are trivially preserved
\begin{proof}
 Let $A,C\seq B\seq D$ be such that $A\indi M_C D$. Then for all $X$ such that $C\seq X\seq \cl(D)$ we have $A\ind_X D$. In particular, for all $Y$ such that $B\seq Y\seq \cl(BD) = \cl(D)$ we have $A\ind_Y D$ so $A\indi M_B D$. 

We turn to the first item. Left \ref{MON}, \ref{CLO} and \ref{NOR} are clearly preserved. For right \ref{MON}: if $A\indi M _C BD$ then in particular for all $X$ with $C\seq X\seq \cl(CB)$ we have $A\ind_X BD$ hence $A\ind_X B$ by right \ref{MON}, hence $A\indi M _C B$. For right \ref{NOR}, assume that $A\indi M_C B$. Then for any $X$ with $C\seq X\seq \cl(BC)$ we have $A\ind_X B$ hence $A\ind_X BX$ by right \ref{NOR} of $\ind$ and hence $A\ind_X BC$ by right \ref{MON} of $\ind$. We conclude that $A\indi M _C BC$. We assume left \ref{TRA} and assume that $B\indi M _C A$ and $D\indi M _B A$ for $C\seq B\seq A$. To get $D\indi M _C A$, let $X$ be such that $C\seq X\seq \cl(AC)$. Using $B\indi M _C A$ we get $B\ind_X A$ and by left \ref{NOR}, we get $BX\ind_X A$ $(*)$. Also, $B\seq XB\seq \cl(AB)$ hence using $D\indi M _B A$ we get $D\ind_{XB} A$ and by left \ref{NOR} we have $DXB \ind_{XB} A$ $(**)$. By putting $(*)$ and $(**)$ together with left \ref{TRA} for $\ind$ we get $DXB\ind_X A$. By left \ref{MON}, we get $D\ind_X A$.

We prove the second item. Preservation of \ref{AREF} is clear since $\indi M$ is stronger than $\ind$. Preservation of left \ref{CLO} is clear. Assume that $\ind$ satisfies right \ref{CLO}, and $A\indi M _C B$. Then $\cl(CB) = \cl(C\cl(B))$ hence for any $C\seq X\seq \cl(C\cl(B))$ we have $A\ind_X B$ hence $A\indi M _C \cl(B)$.

We prove the last item. If $A\indi 0 _C B$, then by \ref{NOR} and \ref{CLO}, we have $A\indi 0 _C \cl(BC)$. Then, by \ref{BMON}, for all $D$ with $C\seq D\seq \cl(BC)$ we have $A\indi 0 _D \cl(BC)$ so $A\indi 0 _D B$ by \ref{MON}. As $\indi 0\to \ind$, we get $A\ind_D B$ . We conclude that $A\indi M_C B$. 
\end{proof}

\begin{exercise}
    If $\ind$ satisfies left \ref{BMON} does $\indi M$ satisfies left \ref{BMON}? 
\end{exercise}

\begin{theorem}\label{thm:pregeom_forcingbmononalg}
Let $(S,\cl)$ be a pregeometry then $\indi{a}^M = \indi\cl$. In particular, $\indi\cl$ is minimal (for $\to$) among the relations stronger than $\indi a$ which satisfy the right version of: \ref{NOR}, \ref{MON}, \ref{CLO} and \ref{BMON}.
\end{theorem}
\begin{proof}
    The direction $\indi\cl\rightarrow \indi a ^M$ always holds since $\indi\cl\rightarrow \indi a$ and $\indi\cl$ satisfies \ref{BMON} by Theorem \ref{thm:pregeom_axiomsfordimindependence}. We prove the contrapositive of the converse. Assume that $A\nindi\cl_C B$ for some $A,B,C$. Using \ref{SCLO} we may assume that $C\seq A\cap B$ and $A,B,C$ are $\cl$-closed. By \ref{SYM}, we have $B\nindi\cl _C A$, hence there exists a finite set $B_0\seq B$ independent over $C$ such that $B_0$ is not independent over $A$. Hence there exists $b\in B_0$ such that $b\in \cl(A\cup (B_0\setminus\set{b}))$. Let $D = B_0\setminus \set{b}$. As $B_0$ is independent over $C$, we have $b\notin\cl(CD)$. However $b\in \cl(AD)\cap \cl(B)$, hence $A\nindi a_{CD} B$. In particular $A\nindi a ^M_C B$. 
\end{proof}

This gives a general definition of $\indi \cl$ even if there is no notion of dimension (i.e. if $\cl$ does not satisfy exchange).

\begin{example}
    Assume that $\cl$ is a closure operator on a set $S$. If $\cl$ satisfies exchange, then $\indi{a}^M $ is symmetric. Is the converse true? The answer is no, we can see this in Example \ref{exo:frederickgebert}. One checks that $A\indi a _C B$ if and only if $\cl(A)\seq \cl(C)$ or $\cl(B)\seq \cl(C)$ if and only if $\sup A\leq \sup C$ or $\sup B\leq \sup C $. It is easy to check that this relation satisfies \ref{BMON}.
\end{example}

\begin{proposition}
    The following are equivalent:
    \begin{enumerate}
        \item If $x\in \cl(AB)$, then there exist singletons $a\in \cl(A)$ and $b\in \cl(B)$ such that $x\in \cl(ab)$;
        \item $\indi a$ satisfies \ref{BMON};
        \item $\indi a = \indi\cl$;
        \item For all $A,B$, $A\indi\cl _{\cl(A)\cap \cl(B)} B$;
        \item (Modular law) $\dim(AB)+\dim(A\cap B) = \dim A +\dim B$, for all $\cl$-closed $A,B$.
    \end{enumerate}
    We say that $\cl$ is \textit{modular} if one of those statements hold.
\end{proposition}
\begin{proof}
    $(1\implies 2)$ Assume that $A\indi a _C BD$, we show $A\indi a_{CD} B$. Using (1), if $x\in \cl(ACD)\cap \cl(BCD)$, there exists $a\in\cl(AC)$ and $d\in\cl(D)$ such that $x\in \cl(ad)$. If $x\in \cl(D)$ we are done hence we may assume otherwise. By exchange, this implies $a\in \cl(xd)\seq \cl(BCD)$. Also, $a\in \cl(AC)$ hence by $A\indi a_C BD$ we get $a\in \cl(C)$ so $x\in \cl(ad)\seq \cl(CD)$. We have proved that $\cl(ACD)\cap\cl(BCD) = \cl(CD)$ hence $A\indi a_{CD} B$.

    $(2\implies 3)$ By Theorem \ref{thm:pregeom_forcingbmononalg} and Proposition \ref{prop:monotonisation}.

    $(3\implies 4)$ By definition, $\cl(A)\indi a_{\cl(A)\cap\cl(B)} \cl(B)$ for all $A,B$. By two-sided \ref{MON}, we get $A\indi a_{\cl(A)\cap \cl(B)} B$ hence $A\indi\cl_{\cl(A)\cap \cl(B)} B$ by hypothesis.

    $(4\implies 5)$ Let $A,B$ be closed sets. By Lemma \ref{lm:pregeom_dimension} \textit{(1)}, we have $\dim(AB) = \dim(A/B)+\dim(B)$. As $A\indi\cl_{A\cap B} B$, we have $\dim(A/B) = \dim(A/A\cap B)$. Using again Lemma \ref{lm:pregeom_dimension} \textit{(1)}, we get $\dim(A) = \dim(A/A\cap B) + \dim A\cap B$. We conclude by putting together the two equations.

    $(5\implies 1)$ We start with a claim.
    \begin{claim}
        Let $A$ be closed and $b$ be a singleton. If $x\in \cl(Ab)$ then there exists $a\in A$ such that $x\in \cl(ab)$.
    \end{claim}
    \begin{proof}[Proof of the claim]
        We may assume that $x,b\notin A$. As $x\in \cl(Ab)$ (hence $b\in \cl(Ax)$) we have $\dim(xb/A) = 1$. Also we may assume $\dim(xb) = 2$. By the modular law, we have
        \[\dim(Abx) +\dim(A\cap \cl(bx)) = \dim(A)+\dim(bx). \quad (*)\]
        By Lemma \ref{lm:pregeom_dimension} \textit{(1)}, we also have 
        \[\dim(Abx) = \dim(bx/A)+\dim A\quad (**)\]
        By putting together $(*)$ and $(**)$ with the above we conclude $\dim(A\cap \cl(bx)) = 1$. Let $a$ be a basis of $A\cap \cl(bx)$. We have $a,b\notin\cl(\emptyset)$ hence as $b\notin\cl(a)$ we have $a\notin\cl(b)$ by exchange. As $a\in \cl(xb)$ we conclude again by exchange that $x\in \cl(ab)$.
        \end{proof}
    We may assume that $A = \cl(A_0)$ and $B = \cl(B_0)$ for finite bases $A_0,B_0$ of $A,B$ (respectively). We prove it by induction on $\dim(A_0B_0)$. Let $b\in B_0$ be such that $x\notin\cl(A_0B_0\setminus\set{b})$ and $b\notin\cl(A_0B_0\setminus\set{b})$. Then by the claim, there exists $y\in \cl(A_0B_0\setminus\set{b})$ such that $x\in \cl(yb)$. As $\dim(A_0(B_0\setminus\set{b}))<\dim(A_0B_0)$, we get by induction hypothesis that there exists $a\in A$ and $b'\in \cl(B_0\setminus\set{b})$ such that $y\in \cl(ab')$, hence $x\in \cl(abb')$. Using again the claim (as $a$ is a singleton), there exist $b''\in \cl(bb')\seq B$ such that $x\in \cl(ab'')$.
\end{proof}

\begin{exercise}
    Let $(V,+,0)$ be a vector space over some field $K$. For $A\seq V$ define $\vect{A}$ to be the vector span of $A$.
    \begin{enumerate}
        \item Prove that $(V,\vect{\cdot})$ is a pregeometry.
        \item Prove that $\indi a = \indi \cl$, hence the pregeometry is modular.
    \end{enumerate}
\end{exercise}

\begin{example}
    Let $(K,+,\cdot,0,1)$ be an algebraically closed field of infinite transcendence degree. Define for any subset $A\seq K$ the closure operator $A^\alg$ to be the relative algebraic closure in $K$ of the field generated by $A$. By Steinitz exchange principle, $A\mapsto A^\alg$ defines a pregeometry on subsets of $K$. This pregeometry is not modular, as $\indi a$ does not satisfy \ref{BMON}. To see this, let $F$ be an algebraically closed subfield of $K$, such that $\trdeg(K/F)$ is large enough. Let $x,y,z$ be algebraically independent over $F$ and $t = xz+y$. Then consider $A = F(x,y)^\alg$ and $B = F(z,t)^\alg$. We have $A\cap B = F$ hence $xy\indi a _F zt$ (this necessitates some checking, left as an exercise). However we do not have $xy\indi a_{Fz} t$ since $t\notin F(z)^\alg$ (again left as an exercise).
\end{example}

\color{blue}

\subsection{Naive monotonisation and forcing the right closure axiom}

\begin{definition}\label{def:naivemonotonisation}
    Let $\ind$ be a ternary relation. We associate the \textit{naive monotonisation $\indi{m}$ of $\ind$} which is defined as the monotonisation with respect to the trivial closure operator:
    \[A\indi m_C B\iff A\ind_D B \text{ for all $D$ with $C\seq D\seq BC$}.\]
\end{definition}

\begin{proposition}\label{prop:naivemonotonisation}
  The relation $\indi{m} $ satisfies right \ref{BMON}. 
  \begin{itemize}
      \item If $\ind$ satisfies left or right \ref{MON}, left \ref{NOR}, so does $\indi m$. If $\ind$ further satisfies right \ref{NOR} or left \ref{TRA}, then so does $\indi m$. 
      \item If $\indi 0\to \ind$ and $\indi 0$ satisfies the right-sided instance of: \ref{NOR}, \ref{MON} and \ref{BMON}, then $\indi 0\rightarrow \indi m$.
  \end{itemize}
\end{proposition}

\begin{proof}
The beginning and the first item are obtained by applying Proposition \ref{prop:monotonisation} with the trivial closure operator. We prove the last item. If $A\indi 0 _C B$, then by \ref{NOR}, we have $A\indi 0 _C BC$. Then, by \ref{BMON}, for all $D$ with $C\seq D\seq BC$ we have $A\indi 0 _D BC$ so $A\indi 0 _D B$ by \ref{MON}. As $\indi 0\to \ind$, we get $A\ind_D B$ . We conclude that $A\indi m_C B$. 
\end{proof}

\begin{definition}\label{def:rightclosure}
    Let $\ind$ be a ternary relation. We associate the \textit{(right) closure extension $\indi{c}$ of $\ind$} which is defined as :
    \[A\indi{c}_C B\iff A\ind_C \cl(BC).\]
\end{definition}

We naturally have the following:
\begin{proposition}\label{prop:forcingclosure}
    Let $\ind$ be a ternary relation. Then $\indi c$ satisfies right \ref{CLO} and right \ref{NOR}. If $\ind$ satisfies right \ref{MON}, then $\indi c\to \ind$. \begin{itemize}
        \item If $\ind$ satisfies left or right \ref{MON}, left \ref{NOR}, so does $\indi c$.
        \item If $\indi 0\to \ind$ and $\indi 0$ satisfies right \ref{NOR} and right \ref{CLO}, then $\indi 0 \to \indi c$.
        \item If $\ind$ satisfies right \ref{BMON} then so does $\indi c$.
    \end{itemize}
\end{proposition}
\begin{proof}
    For right \ref{CLO}: if $A\indi c_C B$ we have $A\ind_C \cl(BC)$. As $\cl(\cl(BC)) = \cl(BC)$ we have $A\ind_C \cl(\cl(BC))$ i.e. $A\indi c _C \cl(BC)$. For right \ref{NOR}: if $A\indi c_C B$ then $A\ind_C \cl(BC)$ so $A\indi c_C BC$. If $\ind$ satisfies right \ref{MON}, then $A\indi c_C B$ implies $A\ind_C \cl(BC)$ which implies $A\ind_C B$.

    For the first item: the left-sided properties are trivial. If $\ind$ satisfies right \ref{MON}, then $A\indi c_C BD$ implies $A\ind_C \cl(CBD)$ which implies $A\ind_C \cl(BC)$ so $A\indi c _C B$. 

    For the second item, if $A\indi 0_C B$ then by right \ref{NOR} and right \ref{CLO} we have $A\indi 0_C \cl(BC)$ which implies $A\ind_C \cl(BC)$ so $A\indi c_C B$.

    For the third item, assume that $A\indi c _C B$ and $C\seq D\seq B$. Then $A\ind_C \cl(B)$ hence by \ref{BMON} $A\ind_D \cl(B)$ so $A\indi c _D B$.
\end{proof}

\begin{proposition}\label{prop:naivetomonviac}
    Assume that $\ind$ satisfies right \ref{NOR} and right \ref{MON} then 
    \[(\indi m)^c\to \indi M \]
    If $\ind$ further satisfies right \ref{CLO} then $(\indi m)^c = \indi M$.
\end{proposition}
\begin{proof}
    First, $\indi m$ satisfies the right-sided version of \ref{NOR}, \ref{MON}, \ref{BMON}, by Proposition \ref{prop:naivemonotonisation}. By Proposition \ref{prop:forcingclosure} we have $\indi m ^c$ satisfies the right-sided version of \ref{NOR}, \ref{MON}, \ref{BMON} and \ref{CLO}. Again by Proposition \ref{prop:forcingclosure} $\indi m ^c\to \ind$ hence by Proposition \ref{prop:monotonisation} we have $\indi m ^c \to \indi M$.

    Assume that $\ind$ satisfies right \ref{CLO}, then by Proposition \ref{prop:monotonisation}, $\indi M$ satisfies the right-sided instances of \ref{NOR} \ref{MON}, \ref{BMON} and \ref{CLO}. As $\indi M\to \ind$, from Proposition \ref{prop:naivemonotonisation} we have $\indi M\to \indi m$. From $\indi M\to \indi m$ and Proposition \ref{prop:forcingclosure} we get $\indi M\to \indi m ^c$.
\end{proof}

The essential difference between $\indi M $ and $\indi m$ is that the former preserves right \ref{CLO} and the other one does not a priori. However this distinction does not appear in a pregeometric theory (\cite[Remark 2.19]{conant2023surprising}):

\begin{proposition}\label{fact:monpregeom}
    Let $(S,\cl)$ be a pregeometry, then  $\indi a ^M = \indi a ^m$.
\end{proposition}
\begin{proof}
    By Proposition \ref{prop:naivetomonviac}, it is enough to prove that $(\indi a ^m) ^c = \indi a ^m$. By Proposition \ref{prop:naivemonotonisation}, $\indi a ^m$ satisfies right \ref{MON} (as $\indi a$ does by Exercise \ref{exo:propofindia}) hence $(\indi a ^m)^c\to \indi a ^m$. It remains to prove the converse. Assume $A\nindi a ^m _C \cl(BC)$ hence there exists $C\seq D\seq \cl(BC)$ and a singleton $u\in (\cl(AD)\cap \cl(BC))\setminus \cl(D)$. Let $A_0\cup\set{a}$ be a finite minimal subset of $A$ witnessing $u\in \cl(aA_0D)$. Then $u\notin \cl(A_0D)$ hence by exchange $a\in\cl(uA_0D)\setminus\cl(A_0D)$. As $u\in \cl(BC)$ we have $a\in \cl(A_0BC)$ hence there is a minimal subset $B_0\cup \set{b}$ of $B$ such that $a\in \cl(A_0B_0bC)\setminus \cl(A_0B_0C)$. Note that $B_0\cup\set{b}\neq \emptyset$ since $a\notin \cl(A_0D)$. Using exchange, we have that $b\in \cl(A_0aB_0C)\setminus \cl(A_0B_0C)$ hence in particular $b\in \cl(AB_0C)\cap \cl(BC)\setminus \cl(B_0C)$ so $A\nindi a _{CB_0} B$. As $B_0\seq B$ we conclude $A\nindi a ^m_C B$.
\end{proof}

\color{black}

%% file: chapter2.tex
\chapter{Independence relations in theories}\label{chapter:2}

\section{Model-theoretic setting.}

\subsection{Special models.} We assume known the notions of languages, sentences, theories, formulas, types, structures, definable sets, substructures, elementary substructures, models, elementary maps, elementary bijections and automorphism of models. 
We are given a complete theory $T$ in a language $\LL$ which is at most countable. By convenience, we assume that $T$ has only infinite models.

\begin{definition}
We say that a model $M$ of $T$ is:
\begin{itemize}
    \item ($\kappa$-universal): every model $N$ of $T$ with $\abs{N}<\kappa$ elementary embeds in $M$;
    \item ($\kappa$-saturated): for all $A\seq M$ with $\abs{A}<\kappa$, every type over $A$ is realized in $M$;
    \item (strongly $\kappa$-homogeneous): for all (partial) elementary bijection $f:A\to B$, for $A,B\seq M$ with $\abs{A} = \abs{B}<\kappa$, there is an automorphism $\sigma$ of $M$ extending $f$.
\end{itemize}
\end{definition}

The following is classical, we will not cover it in the course. Proofs of such results can be found in e.g. \cite[Chapter 10]{H08}. A more set-theoretic approach is given in \cite[Chapter 6]{TZ12}.
\begin{fact}
For any uncountable cardinal $\kappa$ there exists a model of $T$ which is $\kappa^+$-universal, $\kappa$-saturated and strongly $\kappa$-homogeneous.
\end{fact}

A \textit{saturated} model is a model which is saturated in its own cardinality. Under some set-theoretic assumption, such models always exist, but we will not worry about those for now. If a theory is stable, then saturated models exist. The interested reader might consult \cite{halevi2021saturated} for an interesting note on those matters.

\subsection{Monster models.} 
 We work in a so-called \textit{monster model}\footnote{The concept of a monster model is a rhetorical subterfuge invented by model theorists to avoid statements of theorems that starts by ``Let $A\seq M$ where $M$ is $\beth^{2^{\abs{A}^+}}$-saturated...". It allows to avoid taking successive extensions of models to realize types or have automorphisms. This strategy is very similar to the strategy taken in classical algebraic geometry (not the scheme language, before that!) where one would work in a big algebraically closed field of infinite transcendence degree, in order to have the existence of generics of a variety, and automorphisms by infinite Galois theory.} $\MM$ of $T$, that is a model $\MM\models T$ which is $\kappabar^+$-universal, $\kappabar$-saturated and strongly $\kappabar$-homogeneous \textbf{for some big enough $\kappabar$}. We fix once and for all a cardinal $\kappabar $ and a monster model $\MM$ of $T$. Subsets $A,B,C,\ldots$ and tuples $a,b,c,\ldots$ of $\MM$ are \textit{small} if their cardinality is strictly smaller than $\kappabar$. For instance, a small model $M$ of $T$ has an isomorphic copy in $\MM$ by $\kappabar$-universality, hence we will just consider that $M\prec \MM$. We will often omit the mention of $\kappabar$, or $\MM$ by just considering small models, sets and tuples.

\subsection{Types.} For a small set $C\seq \MM$, we denote by $\LL(C)$ the set of $\LL$-formulas with parameters in $C$. If $\phi\in \LL(C)$ this implicitely means that there is a tuple $x = (x_1,\ldots,x_n)$ of variables which are the free variables of $\phi$. We write $S_n(C)$ for the set of complete $n$-types over $C$, i.e. maximal consistent sets of $\LL(C)$-formulas in variables $x_1,\ldots,x_n$. For us, a type is always complete (i.e. maximal), and we will call \textit{partial type} a consistent set of formula that may not be complete. Note that a single formula may be considered as a partial type. By $\kappabar$-saturation, being a consistent type over a small set is equivalent to having a realisation in the monster. Types may be in an infinite number of variables: fix an enumeration $(x_i)_{i<\alpha}$ of variables along some ordinal $\alpha<\kappabar$. A type $p(x_i)_{i<\alpha}$ is just a maximal consistent set of formulas $\phi((x_i)_{i\in S})$ where $S$ is a finite subset of $\alpha$. In this case, we write $p\in S_\alpha(C)$. Often the number of variables in the type will not matter, hence we will just write $S(C)$ for the set of types over $C$ in a small number of variables. For any small tuple $a$ we write $\tp(a/C)$ for the complete type $\set{\phi(x)\in \LL(C)\mid \MM\models \phi(a)}$. This defines a map $\MM^\alpha \rightarrow S_\alpha(C)$ via $a\mapsto \tp(a/C)$. This map is surjective by saturation. This implies that we may always consider a type (over a small set) via its realisations, and that proving that a type is consistent is equivalent to find a realisation of it.

\subsection{Topology on the space of types.} The ``$S$" in $S(C)$ stands for Stone, as $S(C)$ is the Stone space of the boolean algebra $\LL(C)$ of formulas, identified with definable subsets of $\MM$. It is the set of ultrafilters on $\LL(C)$ and is a topological space. The topology on $S(C)$ is given by basic clopen sets (=closed and open) of the form 
\[ [\phi] = \set{p\in S(C)\mid p\models \phi}\] for $\phi\in \LL(C)$. 
\begin{exercise}
    Check that the set $\set{[\phi]\mid \phi\in \LL(C)}$ is a basis of topology. Check that $[\phi]$ are clopen for this topology and that the topology is Hausdorff.
\end{exercise}

The cornerstone of model theory is the so-called compactness theorem:
\begin{fact}[Compactness]
    $S(C)$ is a compact topological space.
\end{fact}

A more practical statement of the compactness theorem is the following.
\begin{center}
    \textit{Let $\pi(x)$ be a partial type over $C$, then $\pi(x)$ is consistent if and only if it is finitely satisfiable ($\pi_0$ admits a realisation for all finite subsets $\pi_0\seq \pi$).}
\end{center}

In our setting, this is equivalent to: 
\begin{center}
    \textit{$\pi(x)$ admits a realisation $a\models \pi$ if and only if for all finite $\pi_0\seq \pi$, there exist $a_0\models \pi_0$.}
\end{center} 

To see that this statement is equivalent to the compactness of $S(C)$: observe that finite satisfiability of $\pi$ is equivalent the family $\set{[\phi]\mid \phi\in \pi}$ having the finite intersection property: $a\models \phi_1\wedge\ldots\wedge\phi_n$ if and only if $\tp(a/C)\in [\phi_1]\cap\ldots\cap [\phi_n]$. Compactness of $S(C)$ is then equivalent to saying that there exist $p\in\cap_{\phi\in \pi} [\phi]$ hence any realisation $a$ of $p$ will satisfy $\pi$.

\subsection{Automorphisms and the Galois approach.} The set of automorphisms of $\MM$ is denoted $\Aut(\MM)$. Those are the bijections of $\MM$ which satisfy $\MM\models \phi(a_1,\ldots,a_n)$ if and only if $\MM \models \phi(\sigma(a_1),\ldots,\sigma(a_n))$ for all $\LL$-formula $\phi(x_1,\ldots,x_n)$, $a_1,\ldots,a_n\in \MM$. Of course, $\Aut(\MM)$ is a group for the composition of maps. For a small set $C\seq \MM$, an automorphism \textit{over $C$} is an automorphism $\sigma$ that fixes the set $C$ pointwise ($\sigma(c) = c \ \forall c\in C$). We write $\Aut(\MM/C)$ for the subgroup of $\Aut(\MM)$ consisting of all automorphisms over $C$. The group $\Aut(\MM/C)$ acts on $\MM$ and fixes the set $p(\MM) = \set{a\in \MM\mid a\models p}$ setwise, for any $p\in S(C)$. Conversely, by strong $\kappabar$-homogeneity, if $a,b\models p$ and $p\in S(C)$, then there is an automorphism $\sigma\in \Aut(\MM/C)$ such that $\sigma(a) = b$.
\begin{exercise}
    Check that $\tp(a/C) = \tp(b/C)$ if and only if there is $\sigma\in \Aut(\MM/C)$ such that $\sigma(a) = b$.
\end{exercise}
\noindent We denote $\tp(a/C) = \tp(b/C)$ by $a\equiv_C b$. We then have
 \[a\equiv_C b\iff \sigma(a) = b\text{ for some $\sigma\in \Aut(\MM/C)$}.\]
Thus we may and will identify a types over a small sets $C\seq \MM$ with an orbit under the automorphism group $\Aut(\MM/C)$. The group $\Aut(\MM)$ also acts on the set $S(C)$ via the mapping $p\mapsto p^\sigma$ where $p^\sigma = \set{\phi(x,\sigma(c))\mid \phi(x,c)\in p}$. For this action, $\Aut(\MM/C)$ fixes $S(C)$ pointwise.\\

Before we continue, we recall notations and definitions that we just introduces and that we will use until the end of this text.

\begin{itemize}
    \item $A,B,C,\ldots$ are small subsets of $\MM$, $a,b,c,\ldots$ are small tuples from $\MM$ maybe of infinite length, for instance indexed by some ordinal $\alpha$;
    \item For a formula $\phi(x)$ and $a$ a tuple from $\MM$ such that $\abs{a} = \abs{x}$ ($a$ and $x$ are indexed by the same ordinal), then $a\models \phi(x)$ stands for $\MM\models \phi(a)$, sometimes simply denoted $\models\phi(a)$. $\phi(\MM)$ is the set of realisations of $\phi$ in $\MM$. Those notations extends to partial types: $a\models \pi(x)$, $\models \pi(a)$, $\pi(\MM)$.
    \item For two partial types in the same variable $\pi(x), \Sigma(x)$, we denote $T\models \forall x (\Sigma(x)\to \pi(x))$ or equivalently $\Sigma(\MM)\seq \pi(\MM)$) by $\Sigma\models \pi$. If $\Sigma$ is closed under finite conjunctions, this is equivalent to $\pi(x)\seq \Sigma(x)$ (as sets of formulas). Also if $\phi\in p$ we denote $p\models \phi$;
    \item $\LL$ is the language of $T$, $\LL(C)$ denotes formulas with parameters in a set $C$, sometimes $\LL_n(C)$ denotes formulas in variables $x_1,\ldots,x_n$;
    \item $S_\alpha(C)$ the space of (complete) types over $C$ in variables indexed by an ordinal $\alpha$, $S(C)$ stands for $S_\alpha(C)$ for some $\alpha$; 
    \item $\Aut(\MM/C)$ is the group of automorphism of $\MM$ fixing $C$ pointwise;
    \item $a\equiv_C b$ iff $\tp(a/C) = \tp(b/C)$, where $\tp(a/C) = \set{\phi(x)\mid \models\phi(a)\text{ for } \phi(x)\in \LL(C)}$;
    \item For small sets, $A\equiv_C B$ stands for ``there exists an elementary bijection $f:A\to B$ fixing $C$ pointwise". By strong $\kappabar$-homogeneity, this is equivalent to ``there exists an automorphism $\sigma\in \Aut(\MM/C)$ such that $\sigma(A) = B$ (as sets)". This is equivalent to saying that for any enumeration $a = (a_i)_{i<\alpha}$ of $A$, there exists and enumeration $b = (b_i)_{i<\alpha}$ of $B$ such that $a\equiv_C b$.
\end{itemize}

\subsection{The algebraic closure and the definable closure.}
\begin{lemma}\label{lm:topological_separation}
    Let $X$ be a topological space, $Y_1,Y_2\seq X$ be two compact subsets. Let $\cH$ be a set of clopen in $X$ such which is closed by finite union and intersections. The following are equivalent:
    \begin{enumerate}
        \item There exists $B\in \cH$ such that $Y_1\seq B$ and $Y_2\cap B = \emptyset$.
        \item for all $y_1\in Y_1$, $y_2\in Y_2$, there exists $B\in \cH$ such that $y_1\in Y_1$ and $y_2\notin Y_2$.
    \end{enumerate}
\end{lemma}
\begin{proof}
    $(1\implies 2)$ is trivial. We assume that $(2)$ holds.
    \begin{claim}
        For each $y_1\in Y_1$ there exist $B(y_1)\in \cH$ such that $y_1\in B(y_1)$ and $Y_2\cap B(y_1) = \emptyset$.
    \end{claim} 
    \begin{proof}[Proof of the claim] Let $y_1\in Y_1$ and consider the set $\cH(y_1)$ consisting of those $B\in \cH$ such that $y_1\in B$. By $(2)$, for all $y_2\in Y_2$ there exists $B\in \cH(y_1)$ such that $y_2\in X\setminus B$. It follows that \[Y_2\seq \bigcup_{B\in \cH(y_1)} X\setminus B \]
    As $\cH$ is a set of clopen, $X\setminus B$ are open and hence by compactness of $Y_2$ there exists $B_1,\ldots,B_n\in \cH(y_1)$ such that $Y_2\seq (X\setminus B_1)\cup\ldots\cup (X\setminus B_n) = X\setminus (B_1\cap\ldots\cap B_n)$. As $\cH$ is closed under positive boolean combination we have $B(y_1) := B_1\cap\ldots \cap B_n\in \cH$. Also clearly $y_1\in B(y_1)$. \end{proof}
    By the claim, we have
    \[Y_1\seq \bigcup_{y_1\in Y_1} B(y_1)\]
    As $B(y_1)$ is open, again by compactness, there exists $B^1,\ldots,B^n\in \set{B(y_1)\mid y_1\in Y_1}$ such that $Y_1\seq B^1\cup\ldots \cup B^n$. Let $B = B^1\cup\ldots \cup B^n$. As $\cH$ is closed under unions, $B\in \cH$. As $Y_2\cap B^i = \emptyset$ for each $i$ we have $Y_2\cap B = \emptyset$.
\end{proof}

\begin{proposition}\label{prop:clopen_definable}
    Let $X$ be a clopen of $S(C)$, then $X = [\phi]$ for some $\phi(x)\in \LL(C)$.
\end{proposition}
\begin{proof}
    Let $\cH$ be the set of all $[\phi]$ for $\phi\in \LL(C)$. Clearly, $\cH$ is closed by positive boolean combinations (even by complement). If $X$ is clopen, then $X$ and $(S(C)\setminus X)$ are closed hence compact. For each $p\in X$ and $q\in S(C)\setminus X$ we have $p\neq q$ hence there exists $\phi\in p$ such that $\phi\notin q$. Then $p\in [\phi]$ and $q\notin [\phi]$. By Lemma \ref{lm:topological_separation} there exists $\phi\in \LL(C)$ such that $X\seq [\phi]$ and $(S(C)\setminus X)\seq S(C)\setminus [\phi]$ hence $[\phi]\seq X$, so $X = [\phi]$. 
\end{proof}

\begin{exercise}\label{exo:topo_restriction_continuous}
    Let $C\seq B$ and consider the restriction map $\pi: S(B)\rightarrow S(C)$ defined by \[\pi(p) = p\upharpoonright C = \set{\theta\in \LL(C)\mid \theta\in p}.\]
    Prove that $\pi$ is continuous and onto.
\end{exercise}

\begin{lemma}\label{lm:invariant-definable-equivalent}
    Let $X$ be a definable set and $C$ a small set. The following are equivalent:
    \begin{enumerate}
        \item $X$ is definable over $C$;
        \item $\sigma(X) = X$ for all $\sigma\in \Aut(\MM/C)$.
    \end{enumerate}
\end{lemma}
\begin{proof}
    $(1\implies 2)$ is clear: if $a\in X = \phi(\MM,c)$ for $\phi(x,c)\in\LL(C)$. Then $\models \phi(\sigma(a),c)$ since $\sigma$ fixes $C$ pointwise, hence $\sigma(a)\in X$. 

    $(2\implies 1)$ Assume that $X$ is defined by a formula $\phi$ over $B$ and that $C\seq B$ (you can change $B$ to $B\cup C$). Consider the restriction map $\pi: S(B)\rightarrow S(C)$ such that $\pi(p) = p\upharpoonright C$.
    By Exercise \ref{exo:topo_restriction_continuous} $\pi$ is continuous and onto. Let $Y\seq S(C)$ be the image of $[\phi]$ under $\pi$. Then $Y$ is closed as the direct image of a compact by a continuous function. 
    \begin{claim}
        $[\phi] = \pi^{-1} (Y)$.
    \end{claim}
    \begin{proof}
        As $Y = \pi([\phi])$, we get $[\phi]\seq \pi^{-1}(Y)$. Assume that $p\in \pi^{-1}(Y)$, so that $p\upharpoonright C =: q\in Y$. As $q\in Y$ there exist $p'\in [\phi]$ such that $p'\upharpoonright C = q = p\upharpoonright C$. Consider $a\models p,a'\models p'$. As $X = \phi(\MM)$, $a'\in X$. Also, $a\equiv_C a'$ hence by invariance, $a\in X$. It follows that $\phi\in p$ hence $p\in [\phi]$.
    \end{proof}
    As $\pi $ is onto, we get from the claim that $S(C)\setminus Y = \pi(S(B)\setminus [\phi]) = \pi([\neg \phi])$ hence $Y$ is open. By Proposition \ref{prop:clopen_definable} we conclude that $Y = [\psi]$ for some $\psi\in \LL(A)$. We conclude: $a\in X \iff \tp(a/B)\in [\phi]\iff \pi(\tp(a/B)) = \tp(a/C)\in [\psi]\iff a\models \psi$ hence $X = \psi(\MM)$.
\end{proof}

\begin{definition} Let $A$ be a small set and $b$ a singleton.
\begin{itemize}
    \item (Algebraic closure) We say that $b$ is \textit{algebraic over $A$} if there exists an $\LL(A)$-formula $\phi(x)$ such that $b\models \phi$ and $\phi(\MM)$ is finite. We denote by $\acl(A)$ the set of all $b\in \MM$ which are algebraic over $A$. 
    \item (Definable closure) We say that $b$ is \textit{definable over $A$} if there exists an $\LL(A)$-formula $\phi(x)$ such that $\phi(\MM) = \set{b}$. We denote by $\dcl(A)$ the set of all $b\in \MM$ which are definable over $A$. 
\end{itemize}
\end{definition}

    A partial type $\pi(x)$ with finitely many realisations ($\pi(\MM)$ is finite) is called \textit{algebraic}. If $\pi = \set{\phi}$, the formula $\phi(x)$ is called \textit{algebraic}.
    
\begin{exercise}
Prove that for elementary extensions $M\prec N$, $\phi(M)$ is finite if and only if $\phi(N)$ is finite. 
\end{exercise}

\begin{exercise}[*]
    Prove that $\acl$ and $\dcl$ are closure operators on $\MM$.
\end{exercise}

\begin{proposition}\label{prop:algeclodefclo_automorphism} Let $C$ be a small set and $a$ a singleton.
    \begin{enumerate}
        \item $a\in \acl(C)$ if and only if the orbit of $a$ under $\Aut(\MM/C)$ is finite.
        \item $a\in \dcl(C)$ if and only if $\sigma(a) = a$ for all $\sigma\in \Aut(\MM/C)$.
    \end{enumerate}
\end{proposition}
\begin{proof}
    \textit{(1)} If $a\in \acl(C)$, let $\phi(x)$ be a formula over $C$ such that $\phi(x)$ has only finitely many realisations: $a_1,a_2,\ldots,a_n$ with $a = a_1$. If $\sigma\in \Aut(\MM/C)$, then for all $b\models\phi$ we have $\sigma(b)\models \phi$ hence the orbit of $a$ under $\Aut(\MM/C)$ is contained in $\set{a_1,\ldots,a_n}$. Conversely, let $a_1,\ldots,a_n$ be the set of all conjugates of $a$ over $C$, say with $a_1 = a$. Then the set $\set{a_1,\ldots,a_n}$ is invariant under $\Aut(\MM/C)$ hence this set is $C$-definable using Lemma \ref{lm:invariant-definable-equivalent}. 
    \textit{(2)} Similarly as above, $a\in \dcl(C)$ if and only if the set $\set{a}$ is definable over $C$ which is equivalent to $a$ being invariant under $\Aut(\MM/C)$ by Lemma \ref{lm:invariant-definable-equivalent}.
\end{proof}

\begin{proposition}\label{proposition:elementarymapextendstoalgebraicclosure}
    Let $f:A\to B$ be an elementary bijection, for some small subsets $A,B$. Then $f$ extends to an elementary bijection $f':\acl(A)\to \acl(B)$.
\end{proposition}
\begin{proof}
    By strong $\kappabar$-homogeneity, there exists an automorphism $\sigma$ extending $f$. For all $a\in \acl(A)$ we have $\sigma(a)\in \acl(B)$  hence $f' = \sigma\upharpoonright\acl(A)$ extends $f$ and is an elementary bijection.
\end{proof}

\begin{exercise}
    Let $a\in \MM$ be algebraic over $C$. Prove that there exists a formula $\phi(x)\in \LL(C)$ such that $\phi(\MM)$ is the orbit of $a$ over $\Aut(\MM/C)$. Prove that $\phi(\MM) = p(\MM)$ where $p(x) = \tp(a/C)$. Such a formula $\phi(x)$ is said to \textit{isolates} the type of $a$ over $C$.
\end{exercise}

\begin{exercise}[*]
    Prove that if $a\equiv_C b$ then there exists $a'\equiv_C a$ such that $ab\equiv_C ba'$.
\end{exercise}

\section{Theories: working examples}

\subsection{A first example: algebraically closed fields}
    Let $\LLr$ be the language of rings, i.e. $\LLr = \set{+,-,\cdot,0,1}$ and let $T$ be the (incomplete) theory of fields, that is $(K,+,-,\cdot,0,1)\models T$ if an only if $T$ is a field. Note that, somehow traditionally, we do not include the function $x\mapsto x^{-1}$, hence an $\LL$-substructure of a model of $T$ is an integral domain. We consider the expansion ACF of $T$ defined by adding to $T$ the following axiom-scheme:
    \[\forall x_0\ldots \forall x_n \exists y (x_n\neq 0 \to x_0+x_1y+\ldots+x_ny^n = 0)\]
    for all $n\in \N$, so that $K\models \ACF$ if and only if $K$ is an algebraically closed field.

We recall the following classical criterion for proving quantifier elimination:

\begin{fact}
An $\LL$-theory $T$ has quantifier elimination if and only if for all $M,N\models T$ with a common substructure $A$ and for all primitive existential formula $\exists x\phi(x)$ where $\phi\in \LL(A)$ is a conjunction of atomic and negatomic, \[M\models \exists x\phi(x)\implies N\models \phi(x).\]  
We may assume that $N$ is $\abs{A}$-saturated.
\end{fact}
    
\begin{theorem}[Tarski]
ACF has quantifier elimination.
\end{theorem}
\begin{proof}
    Let $K_1,K_2$ be two algebraically closed fields, with $K_2$ $\abs{A}$-saturated. In particular the transcendence degree of $K_2$ over $A$ is infinite. Let $\phi(x)$ be a conjunction of atomic and negatomic formulas with parameters in $A$. Then $\phi(x)$ is a conjunction of polynomial equations and inequation with parameters in $A$. Assume that $b\in K_1$ is such that $K_1\models \phi(b)$. It is enough to find $b'\in K_2$ such that $K_2\models \phi(b')$. The fraction field $A_1$ of $A$ in $K_1$ and $A_2$ of $A$ in $K_2$ are isomorphic via a function $f:A_1\to A_2$, hence if $b\in A_1$ then we can choose $b' = f(b)$. We assume $b\notin A_1$. Let $F_i$ be the algebraic closure of $A_i$ in $K_i$. By Steinitz Theorem (unicity of algebraic closure) there exists a field isomorphism $g:F_1\to F_2$ extending $f$, hence similarly, we may assume that $b\notin F_1$. It follows that $\phi(x)$ is a conjunction of polynomial inequations, hence we conclude by taking $b'$ transcendental over $A$.
\end{proof}

Recall that a theory $T$ is \textit{model-complete} if for all $M,N\models T$ we have \[M\seq N\implies M\prec N.\]
Every theory with quantifier elimination is model-complete.
\begin{corollary}
    ACF is model-complete.
\end{corollary}

\begin{exercise}
    For $p$ a prime number or $0$, we denote ACF$_p$ the expansion of ACF expressing that the characteristic is $p$. Check that $\Q^\alg$ and $\F_p^\alg$ are prime models and deduce that any completion of ACF is of the form  ACF$_p$ for some $p$ prime or $0$.
\end{exercise}

Let $\K$ be a monster model for $\ACF_p$, for $p\in \mathbb P \cup\set{0}$.
\begin{corollary}Let $A$ be a small subset of $\K$.
\begin{enumerate}
    \item Every definable subset of $\K$ is either finite or cofinite;
    \item $\acl(A)$ is the field theoretic algebraic closure of the field generated by $A$, denoted $A^\alg$;
    \item (p=0) $\dcl(A)$ is the field generated by $A$, $\Q(A)$;
    \item (p>0) $\dcl(A)$ is the perfect hull of $\F_p(A)$, i.e. $\dcl(A) = \set{b\in \K\mid b^{p^n}\in \F_p(A),\ n\in \N^{>0}}$.
\end{enumerate}
\end{corollary}

\begin{proof} \textit{(1)} By quantifier elimination, every definable set is a boolean combination of polynomial equations. Every polynomial equation defines a finite set and a boolean combination of finite sets is finite or cofinite.\\
\textit{(2)} Clear by quantifier elimination. \\
\textit{(3) and (4)} Let $F$ be $\Q(A)$ or $\F_p(A)$ and $b\in \dcl(F)$. If $b$ is transcendental over $F$, then it has many conjugates over $F$ by Proposition \ref{prop:algeclodefclo_automorphism} so we may assume that $b\in F^\alg$. Let $K$ be the normal closure of $F(b)$. Any automorphism of $F^\alg$ over $F$ extends to a global automorphism of $\K$, which, by Proposition \ref{prop:algeclodefclo_automorphism} fixes $b$. In particular, by Galois theory, $b$ is the only root of its minimal polynomial over $F$ hence $K = F(b)$ and $b$ is purely inseparable over $F$.
\end{proof}

\begin{exercise}
    Let $a$ be an element algebraic over a field $C$. Prove that the equation $m(x) = 0$ isolates $\tp(a/C)$, where $m(X)$ is the minimal polynomial of $a$ over $C$.
\end{exercise}

\begin{exercise}[*]\label{exo:typesinfields}
    Prove that for small subfields $A,B\seq \K$, we have that any field isomorphism $f:A\to B$ is a partial elementary map in $\K$. In particular, $A\equiv B$ if and only if $A$ and $B$ are isomorphic as fields. %By quantifier elimination, for $u\in A$ we have $\M\models \phi(u)$ iff $A\models \phi(u)$ (by quantifier elimination, phi can be choosen to be QF, and A is a substructure) iff $B\models \phi(f(u))$ iff $\M\models \phi(f(u))$, so $f$ is elementary.
\end{exercise}

\subsection{Digression: strongly minimal theories}

\begin{definition}
    A theory $T$ is \textit{strongly minimal} if $\phi(\MM)$ is finite or cofinite,  for all $\phi(x)\in \LL_1(C)$, for all (small) $C$.
\end{definition}

\begin{remark}
    In the definition above, it is important to look at realisations of formulas in a saturated model: consider $(\N,<)$ then every definable set is finite or cofinite however in elementary extensions of $(\N,<)$ there exists infinite co-infinite definable sets, so $\Th(\N,<)$ is not strongly minimal.
\end{remark}

\begin{lemma}\label{lm:stronglyminimal_unique_nonalgextension}
    Let $T$ be strongly minimal and $p\in S_1(C)$ be a non-algebraic type. Then for all $C\seq B$ there exists a unique non-algebraic type $q\in S_1(B)$ extending $p$.
\end{lemma}
\begin{proof}
    If $p\in S_1(C)$ is non algebraic then for every formula $\phi(x)\in \LL_1(C)$ we have $\phi\in p$ if and only if $\phi(\MM)$ is cofinite. This gives an intuition of what $q$ should be: we define 
    \[q := \set{\psi\in \LL(B)\mid \psi(\MM) \text{ is cofinite}}.\]
    $q$ is clearly consistent by compactness: any finite number of cofinite sets always intersect. $q$ is unique by definition.
\end{proof}

\begin{proposition}\label{prop:stronglyminimal:morleysequence}
    Let $T$ be strongly minimal. 
    Then we have the following: 
    \begin{itemize}
        \item (full existence) for all $a, C\seq B$ where $\abs{a} = 1$ there exists $a'\equiv_C a$ with $a'\notin \acl(B)$;
        \item (stationarity) for all $a,a',C\seq B$ with $C = \acl(C)$ and $\abs{a} = \abs{a'} = 1$, if $a\notin \acl(B)$ and $a'\notin\acl(B)$ then $a\equiv_B a'$.
    \end{itemize} Further, for every $a\notin\acl(C)$ and ordinal $\alpha$, there exists an infinite sequence $(a_i)_{i<\alpha}$ such that $a_i\equiv_C a$, $a_0 = a$ and $a_i\notin\acl( C a_{<i})$ for all $i<\alpha$. For any other such sequence $(b_i)_{i<\beta}$ starting with some $b\notin \acl(C)$ we have    \[a_{i_1}\ldots a_{i_n}\equiv_C b_{j_1}\ldots b_{j_n}\]
    for all $(i_1<\ldots<i_n),(j_1<\ldots<j_n)$ tuples of distinct elements in $\alpha, \beta$ respectively.
\end{proposition}
\begin{proof}
    The two items are a rewriting of Lemma \ref{lm:stronglyminimal_unique_nonalgextension}. If $a\in \acl(C)$ then take $a' = a$. Otherwise, take $a'$ to satisfy the (unique) extension $q$ to $B$ of $\tp(a/C)$. We have $a'\equiv_C a$ and $a'\notin \cl(B)$. Assume the hypotheses of the second item. If $a\in \acl(C) = C$ then $a' = a$ hence $a\equiv_B a'$. Otherwise, $\tp(a/C)$ is non-algebraic and has a unique extension $q$ over $B$. By hypotheses, we also have $a\notin \acl(B)$ and $a'\notin \acl(B)$ hence $a$ and $a'$ both satisfy $q$, so $a\equiv_B a'$. For the reminder, take $a_0 = a\notin \acl(C)$. By full existence, take $a_1\equiv_C a_0$ with $a_1\notin\acl(Ca_0)$. Then $a_2\equiv_C a_0$ with $a_2\notin\acl(Ca_0 a_1)$ and by induction $a_i\equiv_C a$ with $a_i\notin\acl(C a_{<i})$. To finish the proof it is enough to prove that $a_0\ldots a_n\equiv_C b_{i_0}\ldots b_{i_n}$ for all $i_0<\ldots<i_n<\beta$. For $n = 0$ this is by Lemma \ref{lm:stronglyminimal_unique_nonalgextension} (for $B = C$). Assume by induction that $a_0\ldots a_{n-1}\equiv_C b_{i_0}\ldots b_{i_{n-1}}$ with $i_n>i_{n-1}>\ldots >i_0$. Let $\sigma$ be an automorphism over $C$ which maps $ b_{i_0}\ldots b_{i_{n-1}}$ to $a_0\ldots a_{n-1}$ and let $b' = \sigma(b_{i_n})$. As  $b_{i_n}\notin\acl(C b_{<i_n})$, we have $b_{i_n}\notin\acl(C b_{i_0}\ldots b_{i_{n-1}})$. Applying $\sigma$, we have $b'\notin\acl(C a_{0}\ldots a_{n-1})$. Also, $a_n\notin\acl(C a_0\ldots a_{n-1})$. Using the property (stationarity) we have $b'\equiv_{Ca_0\ldots a_{n-1}} a_n$ hence $ a_0\ldots a_{n-1} b' \equiv_C a_0 \ldots a_n$ so by applying $\sigma^{-1}$ we get $b_{i_0}\ldots b_{i_n} \equiv_C a_0\ldots a_n$.
    \end{proof}

\begin{remark}
    The second item is false in general if $C$ is not algebraically closed. For a counterexample, take $a,a'$ two distinct conjugates over $C$ and $B = \acl(C)$.
\end{remark}

\begin{corollary}
    Let $T$ be strongly minimal, then $A\mapsto \acl(A)$ defines a pregeometry.
\end{corollary}
\begin{proof}
    By considering the localised closure operator, it is enough to prove that if $a,b\notin \acl(\emptyset)$ and $b\in \acl(a)$ then $a\in \acl(b)$. Equivalently, we assume that $a\notin \acl(b)$ and we prove that $b\notin \acl(a)$. Let $a_0 = b$, $a_1 = a$. Proposition \ref{prop:stronglyminimal:morleysequence} let $a_2\notin\acl(a_0a_1)$ with $a_2\equiv a_0$, and iteratively $a_n\notin\acl( a_{<n})$ with $a_n\equiv a_0$, for all $n\leq \omega$. As $a\notin \acl(b)$, the sequence satisfies $a_i\notin\acl(a_{<i})$ for all $i<\omega$. By Proposition \ref{prop:stronglyminimal:morleysequence} we have $a_0a_1\equiv a_ia_j$ for all $i<j\leq \omega$. On one side we have $a_0 a_\omega\equiv a_ia_\omega$ hence $\tp(a_0/a_\omega) = \tp(a_i/a_\omega)$ for all $i<\omega$. As the sequence is infinite, this implies that $\tp(a_0/a_\omega)$ is non-algebraic. On the other side, as $a_0a_1\equiv a_0a_\omega$,  we have $a_0\notin \acl(a_1)$ i.e. $b\notin \acl(a)$.
\end{proof}

\begin{remark}
    Consider $\indi a$ in $T$ for $\cl = \acl$ i.e. 
    \[A\indi a_C B \iff \acl(AC)\cap \acl(BC) = \acl(C).\]
    As $\acl$ is a pregeometry, $a\notin \acl(B)$ is equivalent to $a\indi a ^M B$ hence $a\indi a _C B$ for all $C\seq B$. By Proposition \ref{prop:stronglyminimal:morleysequence}, $\indi a$ satisfies the following: 
    \begin{itemize}
        \item (full existence) for all $a, C\seq B$ where $\abs{a} = 1$ there exists $a'\equiv_C a$ with $a'\indi a _C B$;
        \item (stationarity) for all $a,a',C\seq B$ with $C = \acl(C)$ and $\abs{a} = \abs{a'} = 1$, if $a\equiv_C a'$, $a\indi a_C B$ and $a'\indi a _C B$ then $a\equiv_B a'$.
    \end{itemize} 
\end{remark}

\begin{remark}[Zilber's trichotomy conjecture]
    We have seen essentially 3 types of pregeometries: trivial ones (where $\cl(X) = X$), modular ones (as in vector spaces or divisible groups) and the one in ACF which is not modular. Zilber's trichotomy, which dates back to the 70's, essentially says that those are the only ones in a strongly minimal theory: 
    \begin{center} If $T$ is strongly minimal, then the geometry induced by the pregeometry $\acl$ is either the trivial geometry ($\acl(X) = X$ for all $X$), either an affine or projective geometry over a division ring, or $T$ interprets an algebraically closed field.
    \end{center}
    The first two cases are often translated into a single notion which corresponds to the case where the pregeometry is called \textit{locally modular}. So the conjecture is often restated as: \textit{every non locally modular strongly minimal theory interprets an algebraically closed field.} It was shown later that the conjecture fails, in 1993 Hrushovski introduced special forms of Fraïssé limits --now called \textit{Hrushovski constructions}-- to construct a counterexample to Zilber's conjecture. Nonetheless, the idea of classifying what sorts of pregeometry appear in reducts or in interpretable structures persisted and have been the source of many applications of model theory to other branches of mathematics. Those are called \textit{restricted trichotomy conjectures} and are shown to hold in several theories. For instance, a form of the trichotomy conjecture proved for generic difference fields was used by Hrushovski to prove the Manin-Mumford conjecture in algebraic geometry in 2001. This field of research is still active today, in 2022 Castle proved the following restricted trichotomy, which was long overdue: if $M$ is a non locally modular strongly minimal structure interpretable in an algebraically closed field $K$ of characteristic $0$ then $M$ interprets $K$.
\end{remark}

\begin{exercise}\label{exo:stronglyminimaliteratingalgeindependence}
    Assume that $a_1,\ldots,a_n$ are such that $a_i\notin\acl(Ca_{<i})$ for all $1\leq i\leq n$, then $a_1,\ldots,a_n$ are independent over $C$.
\end{exercise}

\begin{exercise}[Tarski's Test]\label{exo:domainofelementarysubstructure}
Let $S\seq \MM$. The following are equivalent:\begin{itemize}
\item $\phi(\MM)\cap S\neq\emptyset$, for all $\phi(x)\in \LL_1(S)$ with $\phi(\MM)\neq \emptyset$;
 \item $S$ is the domain of an elementary substructure of $\MM$.
 \end{itemize}
\end{exercise}

In a strongly minimal theory, we denote by $\dim(S)$ the dimension of any subset of $\MM$, relatively to the pregeometry $\acl$.

\begin{corollary}
Let $T$ be strongly minimal.
\begin{enumerate}
\item Every infinite algebraically closed set $S = \acl(S)$ is the domain of an elementary substructure of $\MM$.
\item Two models $M$ and $N$ are isomorphic if and only if $\dim(M) = \dim(N)$. 
\item For all uncountable cardinal $\kappa$, $T$ is \textit{$\kappa$-categorical}: if $M,N\models T$ and $\abs{M} = \abs{N} = \kappa$ then $M\cong N$.
\end{enumerate} 
\end{corollary}
\begin{proof}
\textit{(1)} Let $\phi(x)\in \LL_1(S)$. If $\phi(\MM)$ is finite it consists of elements algebraic over $S$ hence $\phi(\MM)\seq \acl(S)=S$. If $\phi(\MM)$ is infinite, it is cofinite in $\MM$ hence intersects any infinite set, in particular $S$. We conclude by Exercise \ref{exo:domainofelementarysubstructure}. 

\textit{(2)} Let $M$ and $N$ be two small models of $T$ of the same dimension $\kappa$. Let $a = (a_i)_{i<\kappa}$ be a basis of $M$ and $b = (b_i)_{i<\kappa}$ a basis of $N$. Then $a_i\indi a a_{<i}$ and $b_i\indi a b_{<i}$ for all $i<\kappa$ hence $a \equiv b$ by Proposition \ref{prop:stronglyminimal:morleysequence}. Let $\sigma$ be an automorphism of $\MM$ sending $a$ to $b$, it restricts to an elementary bijection between $\acl(a)$ and $\acl(b)$ (Proposition \ref{proposition:elementarymapextendstoalgebraicclosure}). We have $\acl(a) = M$ and $\acl(b) = N$. Any elementary bijection between two structures is an isomorphism, hence we get $M\cong N$. The converse implication is trivial.

\textit{(3)} Observe that $\abs{\acl(S)} \leq \max\set{\aleph_0,\abs{S}}$ for all $S$. In particular, if $\abs{M}\geq \aleph_1$ then $\dim(M) = \abs{M}$ so we conclude by \textit{(2)}.
\end{proof}

\begin{definition}
We say that $T$ has \textit{uniform finiteness} or \textit{eliminates $\exists^\infty$} if for all $\phi(x,y)\in \LL_{1+k}(\emptyset)$ with $\abs{x} = 1$ and $\abs{y} = k$ there is $n\in \N$ such that for all $c\in \MM^k$ if $\phi(\MM,c)$ is finite, then $\abs{\phi(\MM,c)}\leq n$.
\end{definition}

\begin{exercise}
Prove that every strongly minimal theory has uniform finiteness. %Otherwise there are parameters $(a_n)$ such that $\phi(\MM,a_n)$ is finite and of cardinality $\geq n$ for all $n$. Then  by compactness there is an infinite coinfinite definable set.
\end{exercise}

\begin{exercise}[*]
    Let $M$ be a structure such that every definable subsets of $M$ is finite or cofinite ($M$ is called \textit{minimal}). Prove that $\Th(M)$ is strongly minimal if and only if $\Th(M)$ has uniform finiteness.
\end{exercise}

\begin{exercise}
Let $M$ be an $\omega$-saturated model. Prove that $T$ has uniform finiteness if and only if for all $\phi(x)\in \LL_1(M)$ there is $n\in \N$ such that $\phi(M)$ is either infinite or $\abs{\phi(M)}\leq n$.
\end{exercise}

\subsection{Back to algebraically closed fields}

Let $T$ be any completion of ACF and $\MM = \K$. Let $\F$ be the prime field ($\Q$ or $\F_p$ for some prime $p$).

\begin{definition}[Algebraic disjointness]
    A tuple $a = (a_1,\ldots,a_n)$ is \textit{(algebraically) dependent} over some set $C$ if there exists a nontrivial polynomial $P(X_1,\ldots,X_n)$ over $\F(C)$ (the field generated by $C$) such that $P(a_1,\ldots,a_n) = 0$. For small subsets $A,B,C$ of $\K$ we define the following relation \[A\indi \alg _C \ B \iff \text{every finite tuple from $\F(AC)$ which is dependent over $BC$ is dependent over $C$}\]
\end{definition}
\begin{exercise}
    Check that $\indi\alg\  = \indi \cl = \indi a ^M$ for $\cl = \acl = (\cdot)^\alg$.
\end{exercise}  
The dimension in this context is usually called the \textit{transcendence degree}. By Theorem \ref{thm:pregeom_axiomsfordimindependence}, we get 

\begin{proposition}\label{prop:ACFalgebraicsatisfiesbasics}
    The relation $\indi\alg\ $ satisfies \ref{SYM}, \ref{FIN}, \ref{EX},  \ref{NOR}, \ref{MON}, \ref{BMON}, \ref{TRA}, \ref{AREF}, \ref{CLO}, \ref{SCLO}, \ref{LOC}.
\end{proposition}

\begin{definition}[Linear disjointness]
  A tuple $ (a_1,\ldots,a_n)$ is \textit{linearly dependent} over some set $C$ if there exists a nontrivial tuple $(c_1,\ldots,c_n)$ from $\F(C)$ such that $\sum_i a_ic_i = 0$. For small subsets $A,B,C$ we define 
  \[A\indi \ld _C \ B \iff \text{every finite tuple from $\F(AC)$ which is linearly dependent over $BC$ is linearly dependent over $C$}.\]
\end{definition}

\begin{exercise}\label{exo:relationbetweenalgebraicdisjointandlinearlydisjoint}
    Check that $\indi\ld\rightarrow \indi\alg\ $. Also $A\indi\ld_C B$ implies $\F(AC)\cap \F(BC) = \F(C)$ werease $A\indi\alg_C\  B$ implies $(AC)^\alg\cap (BC)^\alg = C^\alg$, where $\F(C)$ is the field generated by $C$, for $\F$ the prime field.
\end{exercise}

\begin{exercise}
    The relation $\indi \ld$ satisfies \ref{MON}, \ref{BMON}, \ref{TRA}, \ref{SYM}, \ref{LOC} (\textit{hard}).
\end{exercise}

Recall that for a field $C$ and two (unital) commutative $C$-algebras $A$ and $B$ we may form the tensor product $A\otimes_C B$ which is again a $C$-algebra which satisfy the universal property: 
\begin{center} \begin{tikzcd}[column sep=tiny]
& A \ar[dr] \ar[drr, "f", bend left=20]
&
&[1.5em] \\
C \ar[ur] \ar[dr]
&
& A\otimes_C B  \ar[r, "h",dashed]
& D \\
& B \ar[ur ]\ar[urr, "g"', bend right=20]
&
&
\end{tikzcd}\\
{\footnotesize \textit{To read as:} for all $C$-algebra homomorphisms $f$ and $g$ from $A$ (resp. $B$) to another $C$-algebra $D$ there is a $C$-algebra homomorphism $h$ from $A\otimes_C B$ to $D$ which commutes with the canonical embeddings $A\to A\otimes_C B$ (mapping $a$ to $a\otimes 1$) and $B\to A\otimes_C B$ (mapping $b$ to $1\otimes b$).}\end{center}
The main facts we will use concerning linear disjointness are the following:
\begin{fact}\label{fact:lineardisjointnesstensor}
  Let $A$, $B$ and $C$ be fields with $C\subseteq A\cap B$.
  Let $A\otimes_C B\to A[B]$ be the canonical mapping $a\otimes b\mapsto ab$.
  This map is an isomorphism iff $A\indi\ld _C B$.
\end{fact}

\begin{fact}\label{fact:lineardisjointnessandalgebraicdisjointness}
If $C$ is algebraically closed, then for all $A,B$ we have $A\indi\ld _C B$ if and only if $A\indi\alg_C\ B$.
\end{fact}

The proof of Fact \ref{fact:lineardisjointnesstensor} is simply unravelling the definition, it is not hard and you can do it as an exercise. Fact \ref{fact:lineardisjointnessandalgebraicdisjointness} is a particular case of a more general phenomenon: for fields $E\seq F\cap K$, if $F\indi\alg _E K$ and $F$ is a \textit{regular} extension of $E$ (which means that $F\indi\ld_E E^\alg$) then $F\indi\ld_E\ K$. The proof of this is more involved.

\begin{remark}\label{remark:indlddonotsatisfyCLO}
        Note that $\indi\alg\ $ is not stronger than $\indi\ld\ $: take any $\alpha\in \Q^\alg\setminus \Q$ then $\alpha \indi\alg_\Q \ \alpha$ but $\alpha\nindi\ld_\Q \alpha$. In particular, in ACF, $\indi\ld\ $ do not satisfy right \ref{CLO} for $\cl = \acl$. It does satisfy \ref{CLO} for the closure operator given by the field generated by.
\end{remark}

\begin{proposition}\label{prop:stationaritylineardisjointness}
    Let $A,A',B,C$ be small fields with $C\seq A\cap A'\cap B$. If $A\equiv_C A'$, $A\indi\ld_C B$ and $A'\indi\ld_C B$, then $A\equiv_B A'$.
\end{proposition}
\begin{proof}
    We use abundantly Exercise \ref{exo:typesinfields}. Let $f:A\to A'$ be a field isomorphism over $C$. Using Fact \ref{fact:lineardisjointnesstensor}, we have $A[B]\simeq A\otimes_C B$ and $A'[B]\simeq A'\otimes_C B$. We extend $f$ to an isomorphism of rings $A\otimes_C B\to A'\otimes_C B$ by preserving $B$ pointwise, i.e. mapping $a\otimes b\mapsto f(a)\otimes b$. This yields an isomorphism of integral domaines $A[B]\to A'[B]$ preserving $B$ pointwise, which extends to the field generated by $AB$ and $A'B'$.
\end{proof}

\subsection{A second example: the random graph} Let $\LL$ be the language of graphs, that is a single binary relation symbole $R$. Let $\Graph$ be the theory of graphs, i.e. expressing that $R$ is antireflexive ($\forall x\neg R(x,x)$) and symmetryc ($\forall xy R(x,y)\leftrightarrow R(y,x)$). The theory RG is defined by adding to Graph the following axiom-scheme:
\[\forall x_0\ldots x_{m-1} \forall y_1\ldots y_{n-1} \left( \bigwedge_{i\neq j} x_i\neq y_j \to \exists z \bigwedge_{i<m} R(z,x_i)\wedge \bigwedge_{j<n} \neg R(z,y_j) \wedge z\neq y_j \right)\]
for all $n,m<\omega$.

A \textit{random graph} (sometimes \textit{Rado graph}, sometimes \textit{Erdős–Rényi}) is a countable model of RG. Note that RG has no finite model. 

\begin{theorem}
    The theory RG has quantifier elimination and is complete. Further, it is model-complete.
\end{theorem}
\begin{proof}
    This is an easy exercise.
\end{proof}

\begin{example}
    Define the following graph with vertex $\N$ and put an edge $\Gamma$ between $n$ and $m>n$ if and only if $2^n$ appears non-trivially in the binary development of $m$. Then $(\N, \Gamma)$ is a random graph.
\end{example}

\begin{exercise}
    Prove that for all small $A,B$, any graph isomorphism $f:A\to B$ is an elementary map. It follows that $A\equiv B$ if and only if the (restricted) graphs $(A,R)$ and $(B,R)$ are isomorphic.
\end{exercise}

\begin{exercise}\label{exo:RGmodelcompanion}
    Prove that any model of $\Graph$ admits and extension which is a model of $\RG$. We say that $\RG$ is the \textit{model-companion} of $\Graph$.
\end{exercise}

\begin{corollary}
    In RG, $\acl(A) = \dcl(A) = A$, for all $A$.
\end{corollary}

It follows that in a monster model $\MM$ of RG, the independence relation $\indi a$ is given by $A\indi a_C B$ if and only if $AC\cap BC = C$, or equivalently $A\cap B\seq C$. 

\begin{proposition}\label{prop:algebraicindepRGbasicproperties}
    In RG, the relation $\indi a$ satisfies  \ref{SYM}, \ref{FIN}, \ref{EX},  \ref{NOR}, \ref{MON}, \ref{BMON}, \ref{TRA}, \ref{AREF}, \ref{CLO}, \ref{SCLO}, \ref{LOC}.
\end{proposition}
\begin{proof}
    In RG, the algebraic closure defines a trivial pregeometry hence in particular modular hence the relation $\indi a = \indi{\acl}\ $ satisfies all those properties by Theorem \ref{thm:pregeom_axiomsfordimindependence}.
\end{proof}

We define a strengthening of the relation $\indi a$: 
\[A\indi{\st}_C\  B \text{ if $A\indi a _C B$ and for all $u,v\in ABC$, if $R(u,v)$ then $u,v\in AC$ or $u,v\in BC$}.\]
We call the relation $\indi{\st}$ the \textit{free amalgamation in RG}.
\begin{proposition}\label{prop:propositionstrongindependenceinRG}
    The relations $\indi{\st}$ satisfies \ref{SYM}, \ref{FIN}, \ref{EX},  \ref{NOR}, \ref{MON}, \ref{BMON}, \ref{TRA}, \ref{AREF}, \ref{CLO}, \ref{SCLO}.
\end{proposition}
\begin{proof}
    The properties \ref{SYM}, \ref{FIN}, \ref{EX}, \ref{NOR}, \ref{MON}, \ref{BMON} and \ref{TRA} are easy to check. The properties \ref{AREF}, \ref{CLO} and \ref{SCLO} are trivial since $\acl(A) = A$.
\end{proof}

\begin{example}\label{example:stindepinRGdonotsatLOC}
    The property \ref{LOC} is not satisfied by $\indi\st$. For a counterexample, take a singleton $a$ and an arbitrary big infinite $B$ such that $R(a,b)$ for all $b\in B$. %if $a\indi\st_{B_0} B$ then as $R(a,b)$ holds, we have $ab\in aB_0$ or $ab\in B_0B$, as $a\notin B$, we must have $ab\seq aB_0$ hence $b\in B_0$, this is for all $b\in B$ hence $B_0 = B$.
\end{example}

%% file: chapter3.tex
\chapter{Axiomatic calculus with independence relations}\label{chapter:3}

\section{Axioms for independence relations in an ambient theory}

We work in the same setting as above, $\MM$ is a monster model for some $\LL$-theory $T$.

\begin{definition}\label{def:axioms_part3}
    A ternary relation $\ind$ on $\MM$ is \textit{invariant under automorphisms} (or simply \textit{invariant}) if for all small $A,B,C$ and $\sigma\in \Aut(\MM)$ we have:
    \[A\ind_C B\iff \sigma(A)\ind_{\sigma(C)} \sigma(B)\]
\end{definition}

\begin{definition}[Axioms of independence relations, part 3]
 Let $\ind$ be an invariant ternary relation on small subsets of $\MM$. We define the following axioms.
\begin{enumerate}[$(1)$]
\setcounter{enumi}{11}
%    \item (\setword{strong finite character}{STRFIN}) If $a\nind_C B$ there exists $\phi(x)\in\tp(a/BC)$ such that $a'\nind _C B$ for all $a'\models \phi(x)$.
  \item (\setword{extension}{EXT}) If $A\ind_C B$ then for any $D\supseteq B$ there is $A'\equiv_{BC} A$ with $A'\ind_C D$.
  \item (\setword{full existence}{FEX}) For all $A,B,C$ there exists $A'\equiv_C A$ such that $A'\ind _C B$.
  \item (\setword{the independence theorem}{INDTHM} over models) Let $M$ be a small model, and assume $A\ind_M B$, $C_1\ind_M A$, $C_2\ind_M B$, and $C_1\equiv_M C_2$. Then there is a set $C$ such that $C\ind_M AB$, $C\equiv_{MA}C_1$, and $C\equiv_{MB}C_2$.
  \item (\setword{stationarity}{STAT} over models) Let $M$ be a small model, and assume $C_1\ind_M A$, $C_2\ind_M A$, and $C_1\equiv_M C_2$. Then $C_1\equiv_{MA} C_2$.
\end{enumerate}
\end{definition}

\begin{proposition}\label{prop:basicpropertieswiththeory}
    Let $\ind$ be an invariant relation.
    \begin{enumerate}[(a)]
        \item If $\ind$ satisfies \ref{EXT} then $\ind$ satifies right \ref{NOR} and right \ref{CLO}.
        \item If $\ind$ satisfies \ref{EX}, \ref{SYM}, \ref{MON}, \ref{BMON}, \ref{TRA} and \ref{EXT}, then $\ind$ satisfies \ref{SCLO}.
        \item If $\ind$ satisfies \ref{EXT} and right \ref{MON} then $\ind$ satisfies \ref{EX} if and only if it satisfies \ref{FEX}.
        \item If $\ind$ satisfies \ref{FEX}, right \ref{NOR}, right \ref{MON} and right \ref{TRA}, then $\ind$ satisfies \ref{EXT}.
    \end{enumerate}
\end{proposition}

\begin{proof}
    \textit{(a)} Assume that $A\ind_C B$. By \ref{EXT} there exists $A'\equiv_{BC} A$ such that $A'\ind_C BC$. As $A'\equiv_{BC} A$ we get $A\ind_C BC$ by invariance. Similarly there exists $A''\equiv_{BC} A$ with $A''\ind_C \acl(B)$. If $\sigma$ is an automorphism over $BC$ which sends $A''$ to $A$, then by invariance $A\ind_C \sigma(\acl(B))$. As $\sigma(\acl(B)) = \acl(B) $ (as set) we conclude $A\ind_C \acl(BC)$.

    \textit{(b)} First we show that $A\ind_C B$ implies $\acl(AC)\ind_{\acl(C) } \acl(BC)$. Assume that $A\ind_C B$ then using \textit{(a)} we have $A\ind_C BC$ and $A\ind_C \acl(BC)$. By \ref{BMON} we have $A\ind_{\acl(C)} \acl(BC)$ and we conclude $\acl(AC)\ind_{\acl(C)}\acl(BC)$ by \ref{SYM}. Conversely if $\acl(AC)\ind_{\acl(C)}\acl(BC)$ then by \ref{MON} and \ref{SYM} we have $A\ind_{\acl(C)} B$. Using \ref{EX} we have $A\ind_C C$ and again by \textit{(a)} we have $A\ind_C \acl(C)$. Using \ref{TRA} and \ref{NOR} with $A\ind_{\acl(C)} B$ we get $A\ind_C B\acl(C)$ hence $A\ind_C B$ by \ref{MON}.

    \textit{(c)} Assume that $\ind$ satisfies \ref{EXT} and right \ref{MON}. Let $A,B,C$ be given then by \ref{EX} $A\ind_C C$. By \ref{EXT} with $B=C$ and $D = BC$, there exists $A'\equiv_{C} A$ with $A'\ind_{C} BC$, so $A'\ind_{C} B$ by \ref{MON}. Conversely if \ref{FEX} holds then in particular by taking $B = C$ we get that there exists $A'\equiv_C A$ with $A'\ind_C C$. By invariance $A\ind _C C$.

    \textit{(d)} Assume that $A\ind_C B$ and $D\supseteq B$ is given. Using \ref{FEX} there exists $A'\equiv_{BC} A$ with $A'\ind_{BC} D$. By right \ref{NOR} we have $A'\ind_{BC} CD$. By invariance $A'\ind_C B$ and by \ref{NOR} $A'\ind_C BC$. Using right \ref{TRA} $A'\ind_C CD$ and by right \ref{MON} we have $A'\ind_C D$.
\end{proof}

\begin{exercise}[Adler]
    Can you prove that in $(d)$ the property right \ref{MON} is necessary? To answer this, you could try to find an invariant relation on some theory satisfying \ref{FEX}, right \ref{NOR}, right \ref{TRA} but neither right \ref{MON} nor \ref{EXT}. The existence of such a theory and relation is an open problem.
\end{exercise}

\begin{lemma}[Baudisch]\label{lm:baudishmonBMON}
    Let $\ind$ be an invariant relation satisfying \ref{FEX}, right \ref{TRA} and \ref{STAT} over some set $C$. Then $\ind$ satisfies \ref{BMON} over $C$: $A\ind_C D\implies A\ind_{B} D$ for all $A,B,D$ with $C\seq B\seq D$.  
\end{lemma}
\begin{proof}
Assume that $A\ind_C D$ and $C\seq B\seq D$. By \ref{FEX}, there exists $A'\equiv_C A$ such that $A'\ind_C B$. Again by \ref{FEX} there exists $A''\equiv_{B} A'$ with $A''\ind_{B} D$. By invariance, we also have $A''\ind_C B$. By right \ref{TRA}, we have $A''\ind_C D$. As $A\ind _C D$ and $A''\equiv_C A'\equiv_C A$ we get $A''\equiv_{D} A$  by \ref{STAT} over $C$. In particular, by invariance, we conclude that $A\ind_{B} D$.
\end{proof}

\begin{exercise}
    Prove that if $\ind$ is invariant, then \ref{EXT} is equivalent to: if $A\ind_C B$ then for all $D\supseteq B$ there exists $D'\equiv_{BC} D$ with $A\ind_C D'$.
\end{exercise}

\begin{exercise}[*]
    Assume that $\ind$ satisfies right \ref{MON} and \ref{FEX}. Prove that for all set $C$, \ref{STAT} over $C$ implies \ref{INDTHM} over $C$. %Sol: assume C_1,C_2,A,B as in indthm, take C_3\equiv_C C_1 with C_3\ind A,B then using rMON C_3\ind A so by stationarity C_3\equiv_{CA} C_1 and using stationarity again C_3 \equiv_{CB} CB
\end{exercise}

\begin{exercise}
    Prove that the following sets of axioms are equivalent for an invariant relation $\ind$:
    \begin{itemize}
        \item \ref{SYM}, \ref{EX},  \ref{NOR}, \ref{MON}, \ref{BMON}, \ref{TRA}, \ref{CLO}, \ref{SCLO}, \ref{EXT}.
        \item \ref{SYM}, \ref{EX},  \ref{MON}, \ref{BMON}, \ref{TRA}, \ref{EXT}.
        \item \ref{SYM}, \ref{MON}, \ref{BMON}, \ref{TRA}, \ref{NOR}, \ref{FEX}.
    \end{itemize}
\end{exercise}

% {\color{red}
% \begin{exercise}
%     If $\ind$ is invariant an satisfies \ref{FEX}. If $\kappa(A) = (\abs{T}+\abs{A})^+$, then for all $A,B$ there exists $C$ such that $\abs{C}\leq \kappa(A)$ such that $A\ind_C B$.\red{?}
% \end{exercise}

% \begin{exercise}
%     Let $\ind$ be an invariant relation satisfying \ref{EXT} and \ref{LOC} then $\ind$ satisfies \ref{FEX} \red{(missing hypothesis?)}.
% \end{exercise}
% }

\textbf{Convention.} For now on, the notation $\indi a$ refers to 
 the independence relation relatively to the closure operator $\acl$, i.e.
 \[A\indi a _C B\iff \acl(AC)\cap \acl(BC) = \acl(C)\]

 \begin{exercise}
     Using Proposition \ref{proposition:elementarymapextendstoalgebraicclosure}, check that $\indi a$ is invariant.
 \end{exercise}

We recall a classical Lemma from P.M. Neumann.%Prove it as in Tent-Ziegler?
\begin{fact}[P.M. Neuman]\label{fact:neumanslemma}
    Let $X$ be a set and suppose $G$ is a group of permutations of $X$. Let $P,Q\seq X$ be finite subsets such that no point in $P$ has a finite orbit. Then there is some $g\in G$ such that $gP\cap Q = \emptyset$.
\end{fact}

\begin{proposition}\label{prop:propertiesofalgebraicindependence}
    The relation $\indi a$ is invariant and satisfies \ref{SYM}, \ref{FIN}, \ref{EX},  \ref{NOR}, \ref{MON}, \ref{TRA}, \ref{AREF} and \ref{FEX}. %(\ref{STRFIN}? dans Kruckman Ramsey)
\end{proposition}
\begin{proof}
    All were checked before except \ref{FEX}. We fix $A,B,C$. Let $B^*:= \acl(BC)\setminus \acl(C)$. We want to find $A'\equiv_C A$ such that $\acl(A'C)\cap B^* = \emptyset$. 
    \begin{claim}
        By compactness it is enough to show that for all finite $A_0\seq \acl(AC)$ and $B_0\seq B^*$, there is $A_0'\equiv_C A_0$ such that $A_0'\cap B_0 = \emptyset$.
    \end{claim}
    \begin{proof}[Proof of the claim]
        I give this argument in details now. This sort of arguments are classical and generally just referred to as ``by compactness". Let $(a_i)_{i<\alpha}$ be an enumeration of $\acl(AC)$ starting with $A$ (i.e. there is $\beta\leq \alpha$ such that $\set{a_i\mid i<\beta} = A$). Let $\Sigma(x)$ be the type of $(a_i)_{i<\alpha}$ over $C$. Now observe that for any $a' = (a_i')_{i<\alpha}$ realising $\Sigma$, the tuple $a'$ is an enumeration of $\acl(A'C)$ for $A' = \set{a_i'\mid i<\beta}$, and $A'\equiv_C A$. To conclude the claim it suffices to prove that $\Delta(x):=\Sigma(x)\cup \set{x_i\neq b\mid b\in B^*, i<\alpha}$ is consistent. By compactness it suffices to prove that it is finitely consistent and each finite fragment of $\Delta(x)$ involves finitely many variables, hence is contained in $\tp(A_0/C)\bigcup \set{x_i\neq b\mid b\in B_0}$ for some finite $A_0\seq \acl(AC)$ and $B_0\seq B^*$.
    \end{proof}

     For such $A_0$ and $B_0$, we have $B_0\cap \acl(C) = \emptyset$ hence the orbit of every element in $B_0$ over $\Aut(\MM/C)$ is infinite. By Fact \ref{fact:neumanslemma} there exists $\sigma\in \Aut(\MM/C)$ such that $\sigma B_0\cap A_0 = \emptyset$. Take $A_0' = \sigma^{-1}(A_0)$.
\end{proof}

\begin{proposition}\label{prop:ACFsatisfiesINVEXTSTAT}
    For $T = \ACF$, the relation $\indi\alg\ $ is invariant and satisfies \ref{EXT} and \ref{STAT} over algebraically closed sets.
\end{proposition}
\begin{proof}
We start with \ref{EXT}. Let $a,B,C$ be given and we prove that there exists $a'\equiv_C a$ with $a'\indi\alg_C\ B$. We may assume that $C\seq B$ and that $a = (a_1,\ldots,a_n)$ is a finite tuple, $B,C$ are fields. By Exercise \ref{exo:typesinfields}, $\tp(a/C)$ is given by the field isomorphism type of $C(a)$ over $C$. Let $X = (X_1,\ldots,X_n)$ be a tuple of algebraically independent elements over $B$. Let $I = \set{P(X)\in C[X] \mid P(a) = 0}$. Observe that $I$ is a prime ideal and let $F$ be the fraction field of $C[X]/I$. For $a_i' := X_i+I$ we have that the map $f:C(a)\to F = C(a')$ which fixes $C$ and $a_i\mapsto a_i'$ is a field isomorphism. To check that $a'\indi\alg _C B$, observe that if $a'$ is algebraically dependent over $B$, there exists $Q(X)\in B[X]$ such that $Q(a') = 0$. Then $Q(X)+I = 0$ hence $Q(X)\in I\seq C[X]$ so $a'$ is algebraically dependent over $C$. To conclude, consider $K = (FB)^\alg$, then $K$ is an elementary extension of $(CB)^\alg$ (by model-completeness) hence the type given by the isomorphism type of $K$ over $CB$ is consistent so there exists such $a'$ in $\K$ and $a'\equiv_C a$ (using $f$).

To conclude, \ref{STAT} over algebraically closed sets follows from Proposition \ref{prop:stationaritylineardisjointness} and Fact \ref{fact:lineardisjointnessandalgebraicdisjointness}.
\end{proof}

\begin{remark}
    Note that in ACF, $\indi \ld$ does not satisfy \ref{EXT}. To see this, use Proposition \ref{prop:basicpropertieswiththeory} (a) and Remark \ref{remark:indlddonotsatisfyCLO} or simply consider $a = \sqrt{2}, C = B = \Q$ and $D = \Q(\sqrt{2})$.
\end{remark}

\begin{proposition}\label{propositionpropertiesindepRGwithtypes}
    If $T = \RG$ then $\indi a$ and $\indi\st$ are invariant relations. $\indi\st$ satisfies \ref{EXT} and \ref{STAT} over every set and $\indi a$ satisfies \ref{EXT} and \ref{INDTHM} over models. 
\end{proposition}

\begin{proof}
    Invariance is trivial. 
    We prove \ref{EXT} for $\indi\st$ which implies it for $\indi a$ (although we already know that $\indi a$ satisfies it by Proposition \ref{prop:propertiesofalgebraicindependence}). Let $a,B,C$ be given and assume that $B$ and $C$ are contained live in a small model $M\prec \MM$. We may assume that $a\cap C = \emptyset$. Let $x$ be a tuple of new elements (i.e. outside of $M$) such that $\abs{x} = \abs{a}$. Let $N = Mx$. The relation $\Gamma = R\upharpoonright M$ defines a graph on $M$. We define an edge relation on $N$ which extends $\Gamma$. Let $f$ be the bijection between $Ca$ and $Cx$ which fixes $C$ pointwise and maps $a_i$ to $x_i$. Define $\Gamma' = f(\Gamma)$ in the sense that $\Gamma(u,v)$ holds if and only if $\Gamma'(f(u),f(v))$ for all $u,v\in Ca$. Observe that $\Gamma'\cap \Gamma = \Gamma(C)$. Let $\Gamma_1 = \Gamma\cup \Gamma'$. Then $(N,\Gamma_1)$ is an extension of $(M,\Gamma)$ (i.e. $\Gamma_1\cap M = \Gamma$) . By Exercise \ref{exo:RGmodelcompanion} there exists a model $(M_1,\Gamma_2)$ of RG extending $(N,\Gamma_1)$. By model-completeness, $(M,\Gamma)\prec (M_1,\Gamma_2)$ hence in particular the type $\tp^{M_1}(x/CB)$ (computed in $M_1$) is finitely satisfiable in $\MM$. Let $a'$ be a realisation of this type in $\MM$. Clearly $a'\equiv_C a$. It is clear that $a'\cap B\seq C$. Further, by construction there is no new edge between $a'$ and $B$ hence $a'\indi\st_C B$. For \ref{STAT}, assume that $a,a',C,B$ are given with $a\equiv_C a'$, $a\indi\st_C B$ and $a'\indi\st_C B$. Let $f:Ca\to Ca'$ be a graph isomorphism over $C$ mapping $a$ to $a'$ and $g:B\to B$ be the identity. We claim that $h = f\cup g$ is elementary. First, as $a\cap B \seq C$, $a'\cap B\seq C$ and $f$ and $g$ agree on $C$, the map is well-defined. The second condition of $\indi\st$ implies that there is no edge between $Ca\setminus C$ (resp. $Ca'\setminus C$) and $B\setminus C$ hence $h$ thus defined preserves the edge relation.

    We turn to the independence theorem for $\indi a$. Let $C_1\equiv_E C_2$, $C_1\indi a_E A$ and $C_2\indi a _E B$. By \ref{SYM}, \ref{MON} and \ref{NOR} we may assume that $E\seq C_1\cap C_2\cap A\cap B$. Let $M$ be a small model containing $A,B,C_1,C_2$. Let $X$ be a set of new points with $X\cap M = \emptyset$ and $\abs{X} = \abs{C_1\setminus E}$. Let $C = X\cup E$. There is a graph isomorphism $h:C_1\cong C_2$ fixing $E$ pointwise. Let $f:C\to C_1$ be a bijection fixing $E$ pointwise. Let $g:C\to C_2$ be $h\circ f$. As $C_1\cap A = E$ the following extension $f_A:CA\to C_1A$ of $f$ is well-defined:
    \[f_A(x) =
    \begin{cases*}
      f(x) &  if $x\in C$\\
      x    &  if $x\in A$
    \end{cases*}
    \]
    Similarly we define the extension $g_B:CB\to C_2B$ of $B$ by $\Id_B$. Let $\Gamma_1 = f_A^{-1}(R)$ and $\Gamma_2 = g_B^{-1}(R)$, and $\Gamma = \Gamma_1\cup \Gamma_2$. Similarly as above, let $N$ be a model of RG extending $(MX, R(M)\cup\Gamma)$ (by Exercise \ref{exo:RGmodelcompanion}). We have $M\prec N$ hence the type $\tp^N(C/AB)$ (computed in $N$) is finitely satisfiable in $\MM$. Any realisation $C$ of this type in $\MM$ satisfies the conclusion of \ref{INDTHM} over $E$.
\end{proof}

\begin{example}
    Note that in RG, $\indi a$ do not satisfy \ref{STAT} (over any set): let $a\equiv_C b$ and $B = C\cup\set{d}$ for $d\notin Cab$ with $R(a,d)$ and $\neg R(b,d)$.
\end{example}

\section{A Theorem of Adler}
\subsection{Indiscernible sequences}

\begin{definition}
    Let $(I,<)$ be an infinite linear order and $a = (a_i)_{i\in I}$ a sequence of tuples of the same size $\alpha$. Let $x=(x_i)_{i<\alpha}$ be a set of variables. The \textit{Ehrenfeucht-Mostowski type of $a$ over $C$}, denoted $\EM(a/C)$, is the set of $\LL_\alpha(C)$-formulas $\phi(\vec x_{1},\ldots,\vec x_n)$ with $\MM\models \phi(a_{i_1},\ldots,a_{i_n})$ for all $i_1<\ldots<i_n\in I$, for all $n<\omega$ and $\vec x_i\seq x$. For each $\vec x_i$ there is a finite subset $\beta_i\seq \alpha$ such that $\vec x_i = (x_j)_{j\in \beta_i}$ and when we say $\MM\models \phi(a_{i_1},\ldots,a_{i_n})$ we really mean $\MM\models \phi(a_{i_1}\upharpoonright\beta_1,\ldots,a_{i_n}\upharpoonright\beta_n)$.
\end{definition}
    Of course, by adding variables in the above we may assume that $\beta_i = \beta_j$ in the formula $\phi(\vec x_1,\ldots,\vec x_n)$.
\begin{remark}
    We see $\EM(a/C)$ as a (partial) type $\Sigma$ in variables 
    \[\underbrace{(x_i)_{i<\alpha},(x_i)_{i<\alpha},\ldots, (x_i)_{i<\alpha},\ldots}_\text{$\omega$-times}\]
    In particular one sees that the $\EM$-type does not depend on the order type of $(I,<)$.
\end{remark}

\begin{definition}
    Let $(I,<)$ be an infinite linear order and $a = (a_i)_{i\in I}$ a sequence of tuples. We say that $a$ is \textit{indiscernible over $C$} if for all $i_1<\ldots<i_n\in I$ and $j_1<\ldots<j_n\in I$ we have $\MM\models \phi(a_{i_1},\ldots,a_{i_n})$ if and only if $\MM\models \phi(a_{j_1},\ldots,a_{j_n})$.
\end{definition}

\begin{exercise}
    Prove that $(a_i)_{i\in I}$ is indiscernible over $C$ if and only if for all $n<\omega$, for all $i_1<\ldots<i_n\in I$ and $j_1<\ldots<j_n\in I$ we have 
    \[a_{i_1}\ldots a_{i_n}\equiv_C a_{j_1}\ldots a_{j_n}\]
\end{exercise}

\begin{exercise}
    Assume that $a=(a_i)_{i<\omega}$ is indiscernible over $C$, then $\phi(x_1,\ldots,x_n)\in \EM(a/C)$ if and only if $\models \phi(a_0,\ldots,a_{n-1})$.
\end{exercise}

\begin{exercise}[*]
    Prove that $\EM(a/C)$ is complete if and only if $a$ is indiscernible over $C$.
\end{exercise}

\begin{exercise}
    Let $(a_i)_{i<\alpha}$ be an indiscernible sequence over $\emptyset$. Then $\abs{\set{a_i\mid i<\alpha}}$ is either $\abs{\alpha}$ or $1$.
\end{exercise}

We will need some basic fact from set theory, namely the Ramsey theorem and the Erd\H{o}s-Rado theorem. In partition calculus, we denote by $[\kappa]^n$ the set of $n$-elements subsets of $\kappa$ and the symbol \[\kappa\to (\lambda)_\mu^n\] stands for the statement: for any function $f:[\kappa]^n\to \mu$, there is $A\seq \kappa$ with $\abs{A} = \lambda$ such that $f$ is constant on $[A]^n$. Stated differently: every partition of $[\kappa]^n$ into $\mu$ pieces has a homogeneous set of size $\lambda$ (meaning that there is an infinite subsets whose $n$-elements subsets are all in one component). Recall that the function $\beth$ is defined as follows: for $\alpha$ an ordinal and $\mu$ a cardinal: $\beth_0(\mu) = \mu$, $\beth_\alpha(\mu) = 2^{\beth_\beta(\mu)}$ if $\alpha = \beta+1$ and $\beth_\alpha(\mu) = \sup_{\beta<\alpha} \beth_\beta(\mu)$ if $\alpha $ is limit. We define also $\beth_\alpha:= \beth_\alpha(\aleph_0)$.

\begin{fact}
Classical partition calculus facts.
\begin{itemize}
    \item (Pigeonhole principle) $\omega\to (\omega)_k^1$
    \item (Ramsey Theorem) $\omega\to (\omega)^n_k$. In other words:
    if $A$ is an infinite set and $C_1,\ldots,C_k$ is a coloring of $[A]^n$, then there is an infinite subset of $A$ whose $n$-elements subsets is of the same color $C_i$.
    \item (Erd\H{o}s-Rado Theorem) $\beth_n^+ (\mu) \to (\mu^+)_\mu^{n+1}$.
\end{itemize}
\end{fact}

Of course if $\kappa'\geq \kappa$ and $\lambda '\leq \lambda$ then 
\[\kappa\to (\lambda)^n_\mu \implies \kappa'\to (\lambda')_\mu^n\]
so one looks for the smallest $\kappa$ and the biggest $\lambda$.

\begin{lemma}[Ramsey and Compactness]\label{lm:ramseycompactness}
    Let $C$ be a small set, $(I,<)$ and $(J,<)$ two ordered sets. For all sequence $a = (a_i)_{i\in I}$ there exists $b = (b_j)_{j\in J}$ such that $b$ is indiscernible over $C$ and $b$ satisfies $\EM(a/C)$, by which we mean that $\phi(b_{j_1},\ldots,b_{j_n})$ for all $j_1<\ldots<j_n\in J$ and $\phi\in \EM(a/C)$.
\end{lemma}
\begin{proof}
    Note that ``satisfying $\EM(a/C)$" do not make sense in general since $\EM(a/C)$ is a type in $\alpha\times \omega$ variables and $b$ is a tuple indexed by $\alpha\times J$. It does make sense to ask that $\models \phi(b_{j_1},\ldots,b_{j_n})$, for all $\phi(x_1,\ldots,x_n)\in \EM(a/C)$ and $j_1<\ldots<j_n\in J$. As $b$ will be indiscernible, we actually get that $\EM(a/C)\seq \EM(b/C)$. So $\EM(b/C)$ is a completion of $\EM(a/C)$. We need to show that the following partial type 
 in the set of variables $(x_j)_{j\in J}$ with $\abs{x_j} = \abs{a_i}$ is consistent:
    \begin{align*}
\{\phi(x_{j_1},\ldots,x_{j_n})\mid \phi\in &\EM(a/C), j_1<\ldots<j_n\in J \} \\ &\cup\{\phi(x_{i_1},\ldots,x_{i_n})\leftrightarrow \phi(x_{j_1},\ldots,x_{j_n}) \mid i_1<\ldots<i_n\in J,\\
& j_1<\ldots<j_n\in J, \phi\in \LL_{\alpha\times n}(C),n<\omega\}
    \end{align*}
    By compactness, it is enough to show consistency for all finite subsets of the above. As $\EM(a/C)$ is closed under conjunctions, it is enough to show that the following partial type is consistent:\[\set{\psi(x_1,\ldots,x_n)}\cup\set{\phi(x_{i_1},\ldots,x_{i_n})\leftrightarrow \phi(x_{j_1},\ldots,x_{j_n}) \mid i_1<\ldots<i_n\in J, j_1<\ldots<j_n\in J, \phi\in\Delta}\]
    for $\psi\in \EM(a/C)$ and a finite $\Delta$ over $C$. Let $A = \set{a_i\mid i\in I}$. Any element of $[A]^n$ is now considered as a tuple $\vec a = (a_{i_1},\ldots,a_{i_n})$ where $i_1<\ldots<i_n$. We define the following equivalence relation for $\vec a,\vec a'\in [A]^n$:
    \[\vec a\sim\vec a'\iff \models \phi(\vec a) \iff \models \phi(\vec a'), \text{ for all $\phi\in \Delta$}.\]
    Observe that there are at most $2^{\abs{\Delta}}$ equivalent classes, % if phi(a)<->phi(b) then either both satisfy it or both do not, this is for each phi in Delta
    hence by Ramsey's Theorem there exists an infinite subset $A'$ included in one single class. By enumerating $A'$ with the induced order, we get that the partial type above is finitely satisfiable.
\end{proof}

\begin{exercise}[*]
    Let $(a_i)_{i<\omega}$ be a $C$-indiscernible sequence and $\alpha>\omega$. Then there exists $(a_i)_{\omega\leq i<\alpha}$ such that $(a_i)_{i<\alpha}$ is $C$-indiscernible.  %take $(b_i)_{i<\alpha}$ satisfying the EM type of a. Then $a<\omega \equiv b<\omega$ so send an automorphism $\sigma$ and choose $a_i = \sigma(b_i)$
\end{exercise}

\begin{lemma}[Erd\H{o}s-Rado and compactness]\label{lm:erdosradocompactness}
    Let $\kappa> \abs{T}+\abs{C}$ and $\lambda = \beth_{(2^\kappa)^+}$. If $(a_i)_{i<\lambda}$ is a sequence of tuples with $\abs{a_i}\leq \kappa$. Then there is an $C$-indiscernible sequence $(b_i)_{i<\omega}$ such that for each $n$ there are $i_0<\ldots<i_n<\lambda$ with 
    \[b_0,\ldots,b_{n}\equiv_C a_{i_0},\ldots,a_{i_n}.\]
    In particular $b$ satisfies the $\EM$-type of $a$ over $C$.
\end{lemma}

\begin{proof}
    See Exercise \ref{exo:erdosrado}.
\end{proof}

\begin{theorem}\label{thm:indiscernibleerdorado}
    Assume that $a = (a_i)_{i<\omega}$ is $C$-indiscernible. Then there exists a (small) model $M\supseteq C$ such that $a$ is $M$-indiscernible. In particular $a$ is $\acl(C)$-indiscernible.
\end{theorem}
\begin{proof}
    Let $N$ be any small model containing $C$. Let $\kappa = \max\set{\abs{a_i}, \abs{N},\abs{T}}$ and $\lambda = \beth_{(2^\kappa)^+}$. Let $(a_i)_{i<\lambda}$ be an extension of $a$ of length $\lambda$. By Lemma \ref{lm:erdosradocompactness} there exists $(b_i)_{i<\omega}$ indiscernible over $N$ such that for all $n<\omega$ there exists $i_0<\ldots<i_n<\lambda$ with $b_0\ldots b_n\equiv_N a_{i_0}\ldots a_{i_n}$. In particular, as $(a_i)_{i<\lambda}$ is $C$-indiscernible, $(b_i)_{i<\omega}$ satisfies the $\EM(a/C)$, which is complete as $a$ is $C$-indiscernible. It follows that $b\equiv_C a_{<\omega}$. Let $\sigma$ be an automorphism over $C$ sending $b$ to $a$, and let $M = \sigma(N)$. Observe that $M$ is a small model containing $C$. Then $aM\equiv_C bN$ hence $a$ is $M$-indiscernible. For the rest, observe that for any model $M\supseteq C$ we have $\acl(C)\seq M$. 
\end{proof}

\begin{exercise}[Hard]\label{exo:erdosrado}
    We prove Lemma \ref{lm:erdosradocompactness}.
    \begin{enumerate}
        \item We prove that there exists the following data by induction. For $n<\omega$, a type $p_n(x_1,\ldots,x_n)\in S_n(C)$, $I_n\seq (2^{\kappa})^+$ with $\abs{I_n} = (2^{\kappa})^+$ and $(X_i^n)_{i\in I_n}$ with $X_i^n\seq \lambda$ such that
        \begin{itemize}
            \item $X_i^{n+1}\seq X_i^n$ for all $i\in I_{n+1}$
            \item $\abs{X_i^n}> \beth_{2^\kappa+\alpha}$, if $i$ is the $\alpha$'s element of $I_n$
            \item $(a_{j_1},\ldots,a_{j_n})\models p_n$ for all $j_1<\ldots<j_n\in X_i^n$ for all $i\in I_n$.
        \end{itemize}
        \begin{enumerate}
            \item Start with $I_0 = (2^\kappa)^+$, $X_i^0=\lambda$ and $p_0 = \emptyset$, check that the three conditions above are satisfied. 
            \end{enumerate}
            Assume that $I_n $, $X_i^n$ $(i\in I_n)$ and $p_n$ have been constructed. Let $(\xi_\alpha)_{\alpha<(a2^\kappa)^+}$ be an increasing enumeration of $I_n$ and set $I_n' = \set{\xi_{\alpha+n}\mid \alpha<(2^\kappa)^+}$. Fix $i = \xi_{\alpha+n}\in I_n'$. 
            \begin{enumerate}
                \item Prove that there are at most $2^\kappa$ types $\tp(a_{j_1},\ldots,a_{j_n}/C)$, for $j_1<\ldots<j_n\in X_i^n$.
                \item Deduce from Erd\H{o}s-Rado that there is some $X_i^{n+1}\seq X_i^n$ and $p_{n+1}^i\in S_{n+1}(C)$ such that $\abs{X_i^{n+1}}>\beth_{2^\kappa+\alpha}$ and $(a_{j_1},\ldots,a_{j_n})\models p_{n+1}^i$, for all $j_1<\ldots <j_n \in X_i^{n+1}$.
                (\textit{Hint.} Observe that $\beth_{2^\kappa+\alpha+n} = \beth_n(\beth_{2^\kappa+\alpha})$.)
                \item Deduce that there exists $I_{n+1}\seq I_n'$ such that $\abs{I_{n+1}} = (2^\kappa)^+$ and $p_{n+1}^i = p_{n+1}^j$ for all $i,j\in I_{n+1}$. (\textit{Hint.} Observe that $\abs{I_{n}'} = (2^\kappa)^+$.)
            \end{enumerate}
            \item For $x = (x_i)_{i<\omega}$ and $p(x) = \bigcup_{n<\omega}p_n$. Prove that any $b\models p$ satisfies the conclusion of Lemma \ref{lm:erdosradocompactness}.
        \end{enumerate}
\end{exercise}

\subsection{Morley sequences and Adler's theorem.}

\begin{definition}
    Let $\ind$ be in independence relation. We say that $(a_i)_{i<\alpha}$ is an \textit{$\ind$-Morley sequence over $C$} if $a_i\ind_C a_{<i}$ for all $i<\alpha$. For $p\in S(B)$ we say that $(a_i)_{i<\alpha}$ is a $\ind$-Morley sequence over $C$ \textit{in $p$} if $a_i\models p$ for all $i<\alpha$. If $B = C$ we call $(a_i)_{i<\alpha}$ a \textit{Morley sequence in $\tp(a_0/C)$}. 
\end{definition}
In general a Morley sequence is by default indexed by $\omega$.

\begin{exercise}
    Prove that if $\ind$ satisfies right \ref{MON}, \ref{EXT} and $a\ind_C B$ for $C\seq B$. Then there exists a Morley sequence over $C$ in $\tp(a/B)$.
\end{exercise}

\begin{exercise}
    Prove that if $\ind$ satisfies \ref{FEX} then for all $a,C$ there exists a Morley sequence in $\tp(a/C)$.
\end{exercise}

\begin{exercise}[*]
    Prove that if $\ind$ is invariant and satisfies (right) \ref{MON} and \ref{STAT}, then any Morley sequence in $\tp(a/C)$ is indiscernible over $C$. %(\textit{Solution.} By induction. Assume that  $a_{i_1}\ldots a_{i_{n-1}}\equiv_C a_{j_1}\ldots a_{j_{n-1}}$ for  $i_1<\ldots<i_n$ and $j_1<\ldots<j_n$, witnessed by some $\sigma(a_{i_k}) = a_{j_k}$ over $C$. By \ref{MON} we have $a_{i_n}\ind _C a_{i_1}\ldots a_{i_{n-1}}$ and $a_{j_n}\ind _C a_{j_1}\ldots a_{j_{n-1}}$. If follows that $\sigma(a_{i_n})\ind _C a_{j_1}\ldots a_{j_{n-1}}$ hence by \ref{STAT} we have $a_{j_1}\ldots a_{j_n}\equiv_C a_{j_1}\ldots a_{j_{n-1}}\sigma(a_{i_n})$ hence using $\sigma^{-1}$ we get $a_{i_1}\ldots a_{i_{n}}\equiv_C a_{j_1}\ldots a_{j_{n}}$)
\end{exercise}

\begin{definition}
    Let $\ind$ be in independence relation, we denote by $\indi\opp\quad$ the relation defined by \[A\indi\opp_C\quad B\iff B\ind_C A.\]
\end{definition}

\begin{exercise}
    If $\ind$ satisfies left-sided \ref{MON}/\ref{NOR}/\ref{BMON}/\ref{TRA} then $\indi\opp\ $ satisfies the right-sided version of those axioms.
\end{exercise}

\begin{lemma}\label{lemma:adlerthmoppmorleysequenceexistence}
    Assume that $\ind$ is invariant and satisfies (left and right) \ref{NOR}, right \ref{MON}, right \ref{BMON}, left \ref{TRA}. If there exists arbitrarily long $\ind$-Morley sequence in $\tp(a/B)$ over $C$ then there is a $B$-indiscernible $\indi\opp\ \ $-Morley sequence in $\tp(a/B)$ over $C$.
\end{lemma}
\begin{proof}
   Let $(a_i)_{i<\kappa}$ be a Morley sequence in $\tp(a/B)$ over $C$ for $\kappa$ big enough. Using Erd\H{o}s-Rado, there exists a sequence $b = (b_i)_{i<\omega}$ which is indiscernible over $B$ and such that for all $n<\omega$ there exists $i_1<\ldots<i_n$ such that $b_0\ldots b_{n-1} \equiv_B a_{i_1}\ldots a_{i_n}$. As $a_m\ind_C a_{<m}$ for all $m<\omega$, we have $b_n\ind_C b_{<n}$ by invariance and right \ref{MON}, hence $(b_i)_{i<\omega}$ is a $\ind$-Morley sequence over $C$ in $\tp(a/B)$ which is $B$-indiscernible. We now ``invert" the sequence. 
   \begin{claim}\label{claim:inversionofindiscernible}
       For any indiscernible sequence $(b_i)_{i<\omega}$ over $B$, there exists $(b_i')_{i<\omega}$ such that $b_0'\ldots b_n'\equiv_B b_n\ldots b_0$, for all $n<\omega$.
   \end{claim}
   \begin{proof}
       Let $x = (x_i)_{i<\omega}$ be variables with $\abs{x_i} = \abs{b_i}$. Let $p_n(x_0,\ldots,x_n) = \tp(b_n,\ldots,b_0)$. To prove the claim it is enough to find a realisation of $p(x) := \bigcup_n p(x_0,\ldots,x_n)$. First observe that $p_n(x_0,\ldots,x_n)\seq p_{n+1}(x_0,\ldots,x_{n+1})$: as $b_n\ldots b_0\equiv_B b_{n+1}\ldots b_1$, if $\phi(x_0,\ldots,x_n)\in p_n$ then $\models \phi(b_{n+1},\ldots,b_1)$ hence $\phi(x_0,\ldots,x_n)\in p_{n+1}$. It follows that $p$ is finitely consistent since $b_n\ldots b_0\models p_n$, for any $n$, so we conclude by compactness.
   \end{proof}
   Let $b' = (b_i')_{i<\omega}$ be the inverse of $b$ as in the claim. Then $b'$ is also indiscernible over $B$. As $b_0'\equiv_B b_0$, $b'$ is a sequence of realisations of $\tp(a/B)$. It remains to check that $b'$ is $\indi\opp\ $ Morley over $C$. By invariance, $b_0'\ind_C b_1'\ldots b_n'$ and as $b_0'\ldots b_n'\equiv_C b_1'\ldots b_{n+1}'$ we have by invariance $b_1'\ind_C b_2'\ldots b_{n+1}'$ and by right \ref{MON} we have $b_1'\ind_C b_2'\ldots b_{n}'$. By right \ref{NOR}, right \ref{BMON} and right \ref{MON} we get $b_0'\ind_{Cb_1'} b_2'\ldots b_n'$. Using left \ref{NOR} and left \ref{TRA} we have $b_0'b_1'\ind_C b_2'\ldots b_n'$. Iterating the last operation we get $b_n'\indi\opp_C b_{<n}'$ so $b'$ is an $\indi\opp\ $-Morley sequence.
\end{proof}

\begin{remark}
    One may ask the following question when confronted to the previous proof: do we really need Erd\H{o}s-Rado or Ramsey is enough? If one get the sequence $(b_i)_{i<\omega}$ via Ramsey and compactness instead of Erd\H{o}s-Rado, then $(b_i)_{i<\omega}$ satisfies the EM-type of $(a_i)_{i<\omega}$. The point would be to conclude that $(b_i)_i$ is also Morley over $C$. In other words the question is: if $a$ is a $\ind$-Morley sequence over $C$ and $b\models \EM(a/C)$, is $b$ a Morley sequence over $C$? Of course here the $\EM$-type is not necessarily complete (otherwise one conclude by invariance).
\end{remark}

\begin{remark}\label{rk:campenhausen} $\text{ }$\\
\noindent\begin{minipage}{0.8\textwidth}% adapt widths of minipages to your needs
Here is another proof of Claim \ref{claim:inversionofindiscernible} suggested by the Baron of Campenhausen. Start with a $B$-indiscernible sequence $b = (b_i)_{i<\omega}$. By Ramsey and compactness there exists a $B$-indiscernible sequence $b'' = (b_i'')_{i\in \Z}$ satisfying the EM-type of $b$ over $B$. Then for each $n\in \N$ we have 
    \[b_{-n}''\ldots b_0''\equiv_B b_0''\ldots b_n''\equiv_B b_0\ldots b_n\]
    hence by setting $b_n':= b_{-n}''$ for each $n\in \N$ we get $b_0'\ldots b_n'\equiv_B b_n\ldots b_0$.
\end{minipage}%
%\hfill
\begin{minipage}{0.2\textwidth}
\begin{center}
    \fbox{\includegraphics[scale=.24]{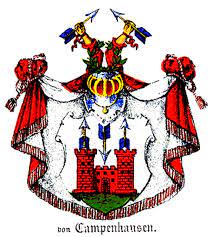}}
\end{center}
\end{minipage}
\end{remark}

%{\color{red} Important here, the proof works with left and right \ref{NOR}. In Adler's paper, he assumes \ref{EXT} which implies right \ref{NOR} so all is okay.}

\begin{fact} Some set-theoretic reminders about regular cardinals.
    \begin{enumerate}
        \item If $\kappa, \lambda$ are cardinals, we say that $\kappa$ is \emph{cofinal in $\lambda$} if there exists a subset $A\seq \lambda$ of cardinality $\kappa$ such that for all $i<\lambda$ there is $j\in A$ such that $i\leq j$. The smallest $\kappa$ which is cofinal in $\lambda$ is called the \textit{cofinality of $\lambda$}, denoted $\cf(\lambda)$. For example $\cf(n) = 1$ and $\cf(\aleph_0) = \aleph_0$. Of course, $\cf(\lambda)\leq \lambda$.
        \item A cardinal $\kappa$ is called \emph{regular} if $\cf(\kappa) = \kappa$.
        \item Every successor cardinal is regular. In particular, for any cardinal $\kappa$ there exists $\lambda\geq \kappa$ which is regular.
    \end{enumerate}
\end{fact}

\begin{lemma}\label{lemma:adlerthmoppmorleysequenceassumption}
    Assume that $\ind$ is invariant and satisfies left and right \ref{MON}, left \ref{NOR}, right \ref{BMON}, left \ref{TRA}, \ref{FIN} and \ref{LOC}. If there is a $B$-indiscernible $\indi\opp\ \ $-Morley sequence over $C\seq B$ in $\tp(a/B)$, then $a\ind^\opp_C B$. 
\end{lemma}
\begin{proof}
    Let $(a_i)_{i<\omega}$ is an $\indi\opp\ \ $-Morley sequence over $C$ in $\tp(a/B)$ which is indiscernible over $B$. By Ramsey and compactness, we may take an extension $(a_i)_{i<\kappa}$ for some regular cardinal $\kappa$ greater than or equal to $\kappa(B)$ as in \ref{LOC}.
     We know that $a_n\indi\opp_C \ \  a_{<n}$ for all $n<\omega$. If $\alpha<\kappa$, then using indiscernibility over $B$, invariance and left \ref{MON}, we have $a_\alpha\ind^\opp_C a_{i_1}\ldots a_{i_n}$, for any $i_1,\ldots,i_n<\alpha$. By \ref{FIN} of $\ind$, we conclude that $a_\alpha\ind^\opp_C a_{<\alpha}$ $(*)$.

     Using \ref{LOC}, there exists $D\seq C a_{<\kappa}$ with $\abs{D}<\kappa$ with $B\ind_D C (a_i)_{i<\kappa}$. As $\abs{D}<\kappa$ and $\kappa$ is regular, $D$ is not cofinal in $\kappa$ hence there is some $\alpha<\kappa$ such that $D\seq Ca_{<\alpha}$. By right \ref{BMON} and right \ref{MON} we conclude $B\ind_{Ca_{<\alpha}} a_\alpha$ $(**)$.

     By combining $(*)$ and $(**)$ we use left \ref{NOR} and left \ref{TRA} of $\ind$ to conclude $Ba_{<\alpha}\ind_{C} a_\alpha$ so $B\ind_{C} a_\alpha$ by left \ref{MON} and $B\ind_{C} a$ by invariance.
\end{proof}

\begin{theorem}\label{thm:adlersymmetry}
Let $\ind$ be an invariant relation satisfying left and right \ref{MON}, right \ref{BMON}, left \ref{NOR}, left \ref{TRA}, \ref{FIN}, \ref{LOC} and \ref{EXT}. Then $\ind$ satisfies \ref{SYM}.
\end{theorem}
\begin{proof}
    By Proposition \ref{prop:basicpropertieswiththeory}, $\ind$ also satisfies right \ref{NOR}. Assume that $a\ind_C B$. By left and right \ref{NOR} we may assume that $C\seq B$. Iterating the use of \ref{EXT}, there exists arbitrarily long $\ind$-Morley sequence in $\tp(a/B)$ over $C$ hence by Lemma \ref{lemma:adlerthmoppmorleysequenceexistence} there exists an $\indi\opp\ $-Morley sequence in $\tp(a/B)$ over $C$ which is $B$-indiscernible. By Lemma \ref{lemma:adlerthmoppmorleysequenceassumption} we have $a\indi\opp_C\ B$.
\end{proof}

\begin{definition}
    An \textit{Adler independence relation} (in short \textit{AIR}) is an invariant relation satisfying left and right \ref{MON}, right \ref{BMON}, left \ref{NOR}, left \ref{TRA}, \ref{FIN}, \ref{LOC} and \ref{EXT}. It is called a \textit{strict} AIR if it further satisfies \ref{AREF}.
\end{definition}

\begin{remark}
    By Adler's theorem of symmetry (Theorem \ref{thm:adlersymmetry}) any Adler independence relation is symmetric, hence there is no need in mentioning the left and right attributes in the previous definition. Using Proposition \ref{prop:basicpropertieswiththeory}, an invariant relation $\ind$ is an AIR if and only if it satisfies \ref{SYM}, \ref{FIN}, \ref{LOC}, \ref{NOR}, \ref{MON}, \ref{BMON}, \ref{TRA}, \ref{EXT}, \ref{EX}, \ref{FEX} and \ref{SCLO}.
\end{remark}

\begin{remark}[Local character depends on the type]
    Let $\ind$ be an invariant relation satisfying \ref{LOC}, i.e. for all $A$ there is $\kappa = \kappa(A)$ such that for all $B$ there is $C\seq B$ with $\abs{C}\leq \kappa$ and $A\ind_C B$. Then $\kappa$ actually does not depend on $A$ but on $\tp(A)$. If $A'\equiv A$ witnessed by $\sigma\in \Aut(\MM)$ with $\sigma(A) = A'$, then for any $B$, apply \ref{LOC} for $A$ with $\sigma^{-1}(B)$ and push back with $\sigma$ to get a subset $C\seq B$ of size $\kappa(A)$ such that $A'\ind_C B$. Exercise \ref{exo:LOCcharacterdependsonthesize} actually shows that under \ref{FIN}, the cardinal $\kappa(A)$ in \ref{LOC} depends only on $\abs{A}$.
\end{remark}

In particular, using \ref{LOC} as in the proof of Lemma \ref{lemma:adlerthmoppmorleysequenceassumption} yields the following convenient lemma.

\begin{lemma}\label{lm:LOCcharactergivesMS}
    Let $(b_i)_{i<\omega}$ be a $C$-indiscernible sequence and $\ind$ be an invariant relation satisfying right \ref{MON}, right \ref{BMON} and \ref{LOC}. Then there exists a model $M$ containing $C$ such that $(b_i)_{i<\omega}$ is an $M$-indiscernible $\ind$-Morley sequence over $M$.
\end{lemma}
\begin{proof}
    Let $\kappa$ be as in \ref{LOC} for $b_0$ and we may assume that $\kappa$ is regular (if $\kappa$ is not regular, consider $\kappa^+$). Note that $\kappa$ only depends on $\tp(b_0)$. Let $(b_i)_{i\leq\kappa}$ be an extension of $(b_i)_{i<\omega}$. 
    
    We construct a chain of models $(M_i)_{i<\kappa}$ such that $M_i$ contains $C b_{<i}$ and $(b_n)_{i\leq n\leq \kappa}$ is $M_i$-indiscernible, for all $i<\kappa$. We use Theorem \ref{thm:indiscernibleerdorado}: for $M_0$ take any model containing $C$ such that $(b_i)_{i<\kappa}$ is $M_0$-indiscernible. If $M_i$ has been constructed, hence $M_i$ contains $C b_{<i}$ and $(b_n)_{i\leq n< \kappa}$ is $M_i$-indiscernible. Then $(b_n)_{i+1\leq n< \kappa}$ is $M_i b_{i}$-indiscernible hence by Theorem \ref{thm:indiscernibleerdorado} there exists $M_{i+1}$ containing $M_ib_{i}$ such that $(b_n)_{i+1\leq n< \kappa}$ is $M_{i+1}$-indiscernible. If $i$ is a limit ordinal, and $(M_j)_{j<i}$ has been constructed, define $M_i = \bigcup_{j<i} M_j$.

    By \ref{LOC}, as $b_0\equiv b_\kappa$, there exists a subset $D$ of $\bigcup_{i<\kappa} M_i$ of size less than $\kappa$ such that $b_\kappa\ind _D \bigcup_{i<\kappa} M_i$. As $\kappa$ is regular, we have that $D\seq M_{i_0}$ for some $i_0<\kappa$ hence by \ref{BMON} we have $b_\kappa\ind_{M_{i_0}} \bigcup_{i<\kappa} M_i$ and by \ref{MON}, $b_\kappa\ind_{M_{i_0}} b_{<\kappa}$. By indiscernibility and invariance, we also have $b_i\ind _{M_{i_0}} (b_j)_{i_0<j<i}$. Using an automorphism over $C$ sending $(b_{i_0+j})_{j<\omega}$ to $(b_j)_{j<\omega}$, we get a model $M$ as needed. Note that $i_0+j<\kappa$ for all $j<\omega$ since otherwise $\kappa$ would not be regular.
\end{proof}

\begin{remark}
    Consider the following version of local character, that we call \textit{chain local character}: for all $A$ there exists $\kappa$ such that for all chain of models $(M_i)_{i<\kappa}$ there exists $i_0<\kappa$ such that $A\ind_{M_{i_0}} \bigcup_{i<\kappa} M_i$. Then, in the statement of the previous Lemma, one might replace \ref{BMON} and \ref{LOC} by chain local character and have the same conclusion.
\end{remark}

\begin{exercise}\label{exo:LOCcharacterdependsonthesize}
    We prove that if $\ind $ is invariant and satisfies, \ref{FIN}, right \ref{BMON} and \ref{LOC}, then for all $A$ there exists $\kappa$ depending only on $\abs{A}$ such that for all $B$ there exists $B_0\seq B$ of size $<\kappa$ such that $A\ind_{B_0} B$. Fix $A$ and $B$.
    \begin{enumerate}[$(a)$]
        \item Let $S$ be a set of representative of $\bigcup_{n<\omega}\MM^n$ quotiented by the equivalence relation $\equiv$. Prove that $\abs{S} = \abs{\bigcup_{n<\omega} S_n(\emptyset)}$.
        \item Let $\lambda = \sup\set{\kappa(s)\mid s\in S}$. Prove that for all finite tuple $a$ from $A$, there exists $B(a)\seq B$ of size $<\lambda$ such that $a\ind_{B(a)} B$.
        \item Take $B_0 = \bigcup_{a\seq_\text{finite} A} B(a)$. Check that $\abs{B_0}\leq \lambda+A$ and that $A\ind_{B_0} B$.
        \item Conclude.
    \end{enumerate}
\end{exercise}

\section{An easy criterion for NSOP$_4$ theories}

\subsection{Heirs and coheirs}

\begin{definition}[Heirs and coheirs]
    Let $C\seq B$, $p(x)$ a type over $C$ and $q$ a type over $B$ with $p\seq q$. 
    \begin{itemize}
        \item $q$ is an \textit{heir} of $p$ if for all $\phi(x,b)\in q$ there exists $c$ from $C$ such that $\phi(x,c)\in p$.
        \item $q$ is a \textit{coheir} of $p$ if $q$ is finitely satisfiable in $C$: for all $\phi(x,b)\in q$ there exists $c$ from $C$ such that $\models \phi(c,b)$.
    \end{itemize}
    We write $a\indi h _C b$ if $\tp(a/Cb)$ is an heir of $\tp(a/C)$ and $a\indi u _C b$ if $\tp(a/Cb)$ is a coheir of $\tp(a/C)$.
\end{definition}

\begin{exercise}[*]\label{exo:indiscernible_overinitialsegment_heir}
        Let $\alpha>\omega$ and $(a_i)_{i<\alpha}$ be a $C$-indiscernible sequence. Prove that $(a_i)_{\omega\leq i<\alpha}$ is a $Ca_{<\omega}$-indiscernible $\indi h$-Morley sequence over $Ca_{<\omega}$, i.e. $(a_i)_{\omega\leq i<\alpha}$ is $Ca_{<\omega}$-indiscernible and $a_i\indi h _{Ca<\omega} a_{<i}$ for all $ \omega \leq i<\alpha$.
\end{exercise}

\begin{exercise}\label{exo:heirdcl}
    Prove that if $a\indi h_C b$ then $a\indi h_{\dcl(C)} b$.
\end{exercise}

\begin{exercise}
    Prove that $a\indi h_C b$ if and only if $b\indi u_C a$, i.e. $(\indi h)^\opp = \indi u$.
\end{exercise}

\begin{exercise}
    If $\ind$ satisfying \ref{SYM} and $\indi 0$ is a relation with $\indi 0\to \ind$, then $(\indi 0)^\opp\to \ind$. In particular $\indi h\to \ind$ if and only if $\indi u\to \ind$.
\end{exercise}

\begin{exercise}
    Prove that $\indi h$ (and $\indi u$) are invariant under automorphisms.
\end{exercise}

\begin{proposition}\label{prop:propertiesofindiu}
The relation $\indi u$ is invariant and satisfies \ref{FIN}, left and right \ref{MON}, left and right \ref{BMON}, left \ref{TRA} and left \ref{NOR}.    
\end{proposition}
\begin{proof}
    We check left \ref{TRA} and left \ref{NOR}, the other properties are easy and left as an exercise. 
    Let $C\seq B\seq D$ and $A$ such that $D\indi u _B A$ and $B\indi u_C A$. Let $\phi(x,a)\in \tp(D/AC)$. As $D\indi u _B A$, there exists $b$ from $B$ such that $\models \phi(b,a)$, hence $\phi(x,a)\in \tp(B/AC)$. As $B\indi u _C A$ there exists $c$ from $C$ such that $\models \phi(c,a)$. It follows that $D\indi u_C A$. For \ref{NOR}, assume that $A\indi u _C B$ and assume that $\phi(ac, b)$ holds, for some $\LL(C)$-formula $\phi(x,y,z)$, i.e. $\phi(x,y,b)\in \tp(AC/BC)$. Then $\phi(xc,b)\in \tp(A/BC)$ hence as $A\indi u _C B$, there exists $c'$ from $C$ such that $\models \phi(c'c,b)$, hence $AC\indi u_C B$.
\end{proof}

Recall that a filter $\cF$ on a set $S$ is a subset of the powerset of $S$ which does not contain $\emptyset$, is closed under finite intersections and under supersets (if $a\in \cF$ and $a\seq b$ then $b\in \cF$). An ultrafilter $\cU$ on $S$ is a filter on $S$ such that for all subset $a$ of $S$ we have either $a$ or its complement is in $\cU$. A principal ultrafilter is given by all subsets containing a given element. Using Zorn's Lemma, every filter is contained in an ultrafilter.  

\begin{definition}
    Let $A, B$ be small sets and let $\cU$ be an ultrafilter on $A^n$. The \textit{average type of $\cU$ over $B$} is the type defined by:
    \[\Av(\cU/B):= \set{\phi(x,b)\in \LL(B)\mid \set{a\in A^n\mid \models \phi(a,b)}\in \cU}\]
    for $x = (x_1,\ldots,x_n)$.
\end{definition}

\begin{exercise}
    Check that the definition above defines a (consistent/complete) type over $B$.
\end{exercise}

%what if U is principal and B does not define a?

\begin{theorem}\label{thm:extensionforu-independence}
    $a\indi u _C B$ if and only if $\tp(a/BC) = \Av(\cU/BC)$ for some ultrafilter $\cU$ on $C^{\abs{a}}$. In particular, $\indi u$ satisfies \ref{EXT}.
\end{theorem}
\begin{proof}
    Assume first that $a\indi u _C B$. Let $n = \abs{a}$. Define $\cF$ to be the set of $C'\seq C^n$ such that $\phi(C,b)\seq C'$, for some $\phi(x,b)\in \tp(a/BC)$. As $a\indi u _C B$, $\phi(C,b)$ is non-empty for any $\phi(x,b)\in \tp(a/BC)$. As $\cF$ is clearly closed under finite intersections and closed under supersets, $\cF$ defines a filter on $C^n$. Let $\cU$ be an ultrafilter on $C^n$ extending $\cF$. Clearly, $\phi(x,b)\in \Av(\cU/BC)$ for any $\phi(x,b)\in \tp(a/BC)$ hence $\tp(a/BC)\seq \Av(\cU/BC)$. As $\tp(a/BC)$ is a complete type, we have $\tp(a/BC) = \Av(\cU/BC)$. Conversely, if $\cU$ is an ultrafilter on $C^n$ and $\tp(a/BC) = \Av(\cU/BC)$, then for all $\phi(x,b)\in \tp(a/BC)$, the set $\phi(C,b) = \set{c\in C^n\mid \models \phi(c,b)}\in \cU$ hence is not empty, so $a\indi u_C B$.

    To conclude, we first prove extension for finite tuples. Assume that $a\indi u _C B$ and let $\cU$ be an ultrafilter on $C^n$ such that $\tp(a/BC) = \Av(\cU/BC)$. Assume further that $BC\seq D$. Let $q(x,D)$ be the type $\Av(\cU/D)$. Clearly, for any $\phi(x,b)\in \tp(a/BC)$ we have $\phi(x,b)\in \Av(\cU/BC)$ if and only if $\phi(x,b)\in \Av(\cU/D)$ hence $\tp(a/BC)\seq q(x,D)$. For any $a'\models q(x,D)$, we have $a'\equiv_{BC} a$ and $a'\indi u_C D$. Using \ref{FIN}, we have that $\indi u$ satisfies \ref{EXT}.
\end{proof}

\begin{remark}
    The characterisation of finitely satisfiable types via average types explains the index ``$u$" in $\indi u$: it stands for \emph{ultrafilter}.
\end{remark}

\begin{corollary}\label{cor:indusatisfiesFEXovermodels}
    $\indi u$ satisfies \ref{EX} and \ref{FEX} \emph{over models}: for any $A,B,M$ there exists $A'\equiv_M A$ with $A'\indi u _M B$. 
\end{corollary}
\begin{proof}
    To get \ref{EX}, observe that if $\phi(x,m)\in \tp(A/M)$ then $\MM\models \exists x\phi(x,m)$ hence $M\models \exists x\phi(x,m)$. There is $m'$ from $M$ such that $\models \phi(m',m)$, hence $A \indi u _M M$. The rest is from Theorem \ref{thm:extensionforu-independence} and Proposition \ref{prop:propertiesofindiu}.
\end{proof}

\begin{theorem}\label{thm:modelheirindependence}
        $\indi h$ satisfies \ref{LOC}. In particular, if $a = (a_i)_{i<\omega}$ is $C$-indiscernible, then there exists a model $M$ containing $C$ such that $a = (a_i)_{i<\omega}$ is an $M$-indiscernible $\indi h$-Morley sequence over $M$, i.e. $a_n\indi h _M a_{<n}$ for all $n<\omega$. 
\end{theorem}

\begin{proof}
    By \ref{FIN} it is enough to prove that for all finite tuple $a$ and for all $B$ there exists $C\seq B$ with $\abs{C} \leq \abs{T}$ such that $a\indi h _C B$. Assume that $a$ and $B$ are given. We construct a sequence $(C_n)_{n<\omega}$ with $C_n\seq B$ such that $\abs{C_{n+1}} \leq \abs{T}+\abs{C_n}$ and for all $\phi(x,b)\in \tp(aC_n/B)$ there exists $c\in C_{n+1}$ such that $\phi(x,c)\in \tp(aC_n/C_{n+1})$. We proceed by induction. Let $C_0 =\emptyset$ and if $C_n$ has been constructed, then enumerates all formulas $\phi(x,b)\in \tp(aC_n/B)$ (i.e. such that there exists $d\in aC_n$ such that $\models \phi(d,b)$) and extend $C_n$ by adding one tuple $b$ for each such formula $\phi(x,b)$. Then $\abs{C_{n+1}}\leq \abs{T}+\abs{C_n}$. By construction we have $a\indi h _C B$ for $C = \bigcup_{i<\omega} C_i$ hence as $\abs{C} \leq \abs{T}$ we conclude \ref{LOC} for $\indi h$. As $\indi h$ satisfies right \ref{MON} and right \ref{BMON}, we conclude by Lemma \ref{lm:LOCcharactergivesMS}.
\end{proof}

\begin{exercise}[Ramsey]
    We give an alternate proof of the fact that if $a = (a_i)_{i<\omega}$ is $C$-indiscernible, then there exists a model $M$ containing $C$ such that $a = (a_i)_{i<\omega}$ is an $M$-indiscernible $\indi h$-Morley sequence over $M$. The advantage of this proof due to N. Ramsey (not the old Ramsey of ``Ramsey and compactness", a recent Ramsey) is that it does not use the Erd\"os-Rado theorem.
    \begin{enumerate}
        \item Consider the expansion $\MM^\Sk$ of $\MM$ by Skolem functions in the expansion $\LL^\Sk$ of the language $\LL$: for each $\LL$-formula $\phi(x,y)$ there is a $\abs{y}$-ary function $f_\phi$ such that $\MM^\Sk\models \phi(f_\phi(y),y)$. Prove that $\MM^\Sk$ is a monster model of its theory. 
        \item Prove that every $\dcl$-closed set in $\MM^\Sk$ is an elementary substructure of $\MM^\Sk$ and the reduct to $\LL$ is an elementary substructure of $\MM$ (\textit{Hint.} Use Exercise \ref{exo:domainofelementarysubstructure}.)
        \item By Ramsey (the old one) and compactness, there exists an $\LL^\Sk$-indiscernible sequence $(a_i')_{i<\omega+\omega}$ over $C$ satisfying the $\LL^\Sk$-EM type of $(a_i)_{i<\omega}$ over $C$. %The sequence $(a_i)_{i<\omega}$ is not necessarily indiscernible in $\MM^\Sk$ but its $\LL^\Sk$-EM-type contains its $\LL$-EM-type which is complete.
        \begin{enumerate}[$(a)$]
            \item Prove that there exists a model $N$ such that $(a_i')_{\omega<i<\omega+\omega}$ is $N$-indiscernible and $a_i'\indi h _{N} a_{<i}'$ in the sense of $\MM^\Sk$ (\textit{Hint.} Use Exercises \ref{exo:indiscernible_overinitialsegment_heir} and \ref{exo:heirdcl}.) %Consider $N = \dcl(C (a_i')_{i<\omega})$ in the expanded language $\LL^\Sk$.
            \item Prove that $(a_i')_{\omega<i<\omega+\omega}\equiv_C (a_i)_{i<\omega}$ in $\MM$.%As the sequence $(a_i')_{i<\omega}$ satisfies the $\LL$-EM-type of $(a_i)_{i<\omega}$ over $C$ in $\MM^\Sk$, this holds also in the reduct $\MM$. 
            \item Let $N_0$ be the reduct of $N$ to $\LL$. Prove that there exists $M\succ \MM$ such that $(a_i')_{\omega<i<\omega+\omega} N_0 \equiv_C (a_i)_{i<\omega} M$.%The  is an $N_0$-indiscernible $\indi h$-Morley sequence over $N_0$, $(a_i)_{i<\omega}$ is an $M$-indiscernible $\indi h$-Morley sequence over $M$, which concludes the proof.
            \item Conclude.
        \end{enumerate}
    \end{enumerate} 
\end{exercise}

\begin{exercise}[*]\label{exo:strongfinitecharacter}
    We define the following property
    \begin{itemize}
        \item (\setword{strong finite character}{SFIN}). If $a\nind_C b$, then there is a formula $\phi(x,b)\in \tp(a/Cb)$ such that for all $a'$, if $a'\models \phi(x,b)$ then $a'\nind_C b$.
    \end{itemize}
    Assume that $\ind$ satisfies left \ref{MON}, \ref{SFIN} and that $\indi\opp\ \ $ satisfies \ref{EX}. Prove that $\indi u \to \ind$.
%In particular, as $\indi u$ satisfies \ref{EXT} over models, so does $\ind$. 
%(\textit{Solution.} Assume $a\indi h _C b$ and $a\nind _C b$ then by \ref{SFIN} there is a formula $\phi(x,b)\in tp(a/Cb)$ such that if $a'\models \phi(x,b)$ then $a'\nind _C b$. As $\tp(a/Cb)$ is heir of $\tp(a/C)$ there is $c\in C$ such that $\models \phi(a,c)$, so $a\nind _C c$, so $b\nind _C C$ by right \ref{MON}. This contradicts \ref{EX}.)
\end{exercise}

\begin{exercise}[*]\label{exo:indihsatisfiesLOC}
    Prove that any relation satisfying \ref{MON}, \ref{SYM}, \ref{EX} and \ref{SFIN} also satisfies \ref{LOC}. Deduce from Example \ref{example:stindepinRGdonotsatLOC} that in RG, $\indi\st$ do not satisfy \ref{SFIN}.
\end{exercise}

{\color{blue}
\begin{exercise}
    Prove that in any theory, $\indi a$ satisfies \ref{SFIN}. Deduce that $\indi a$ always satisfies \ref{LOC}.
\end{exercise}
}

\begin{exercise}\label{exo:heirandcoheirinACF}
    In ACF prove that $\indi h\to \indi\alg\ $ and $\indi u\to \indi\alg\ $. Actually, one can prove that in ACF, $\indi h = \indi u = \indi\alg\ $.
\end{exercise}

\begin{exercise}
    Prove that if $\ind$ is invariant and satisfies \ref{EX}, \ref{SYM}, \ref{MON}, \ref{BMON}, \ref{TRA}, \ref{EXT} and \ref{SFIN}, then $\ind$ is an AIR.
\end{exercise}

\begin{exercise}[Open]
    Is there an AIR which does not satisfy \ref{SFIN}?
\end{exercise}

\subsection{NSOP$_4$ theories} 

\begin{definition}
    Given $n\geq 3$, we say that $T$ has the $n$-\textit{strong order property} ($n$-SOP) if there is an indiscernible sequence $(a_i)_{i<\omega}$ such that if $p(x,y) = \tp(a_0,a_1)$ then the type $q(x_1,\ldots,x_n)$ defined by \[p(x_1,x_2)\cup p(x_2,x_3)\cup\ldots\cup p(x_{n-1},x_n)\cup p(x_n,x_1)\]
    is inconsistent. We say that $T$ is NSOP$_n$ if it does not have the $n$-SOP.
\end{definition}

\begin{remark}\label{rk:NSOPnNSOPn+1}
    If $T$ is NSOP$_n$ then $T$ is NSOP$_{n+1}$. Indeed if $T$ is NSOP$_n$, then for any indiscernible sequence $(a_i)_{i<\omega}$ there is $c_1,\ldots, c_n$ such that $c_i c_{i+1}\equiv a_0a_1$ for $i< n$ and $c_n c_1\equiv a_0a_1$. Then $c_1c_2\equiv a_0a_1\equiv a_0a_2$ hence there exists $c_{\frac{3}{2}}$ such that $c_1c_{\frac{3}{2}} c_2 \equiv a_0a_1a_2$ and $c_1,c_{\frac{3}{2}}, c_2,\ldots,c_n$ witness NSOP$_{n+1}$.
\end{remark}

\begin{example}
    DLO has the SOP$_n$ for all $n\geq 3$. Indeed, for an increasing sequence $(a_i)_{i<\omega}$ were $a_i<a_j$ iff $i<j$, then $p(x,y)=\tp(a_0,a_1)$ is isolated by the formula $x<y$, and clearly if the type $q$ above was consistent, it would contradicts transitivity of $<$.
\end{example}

\begin{lemma}\label{lm:stat_invert}
    Assume that $\ind$ is invariant and satisfies \ref{SYM} and \ref{STAT} over models. If $a\ind_M b$ and $a\equiv_M b$ then $ab\equiv_M ba$.
\end{lemma}
\begin{proof}
    By an automorphism, let $a'$ be such that $ab\equiv_M ba'$. By invariance, $b\ind_M a'$ and by \ref{SYM} $a'\ind_M b$. Using \ref{STAT} we conclude $a'\equiv_{Mb} a$ hence $a'b\equiv_M ab$. As $ab\equiv_M ba'$, we have $ba\equiv_M ab$.
\end{proof}

\begin{theorem}\label{thm:criterionNSOP4}
Let $\ind$ be an invariant relation satisfying: \ref{SYM}, \ref{FEX}, \ref{STAT} over models, and the following weak transitivity over models:
    \[a\ind_{Md} b \text{ and } a\indi h _M d \text{ and } b\indi u _M d  \implies a\ind_M b\]
    Then $T$ is NSOP$_4$.
\end{theorem}
\begin{proof}
        Let $(a_i)_{i<\omega}$ be an indiscernible sequence and $p(x,y) = \tp(a_0,a_1)$. We show that 
    \[p(x_0,x_1)\cup p(x_1,x_2)\cup p(x_2,x_3)\cup p(x_3,x_0)\]
    is a consistent partial type. By Theorem \ref{thm:modelheirindependence}, $a_i \indi h _M a_{<i}$ for all $i<\omega$ and some small model $M$.

    By \ref{FEX}, there exists $a_0^*\equiv_{Ma_1} a_0$ such that $a_0^*\ind_{Ma_1} a_2$. By \ref{SYM}, we have $a_2\ind_{Ma_1} a_0^*$. As $a_2\indi h _M a_1$ and $a_0^*\indi u _M a_1$, we conclude $a_2\ind_{M} a_0^*$ using the weak transitivity assumption.

    We have $a_0^*\equiv_M a_2$ and $a_0^*\ind_M a_2$ hence by \ref{STAT} and \ref{SYM} (Lemma \ref{lm:stat_invert}), we have $a_0^*a_2\equiv_M a_2a_0^*$. Then, there exists $a_3^*$ such that $a_0^*a_2a_1\equiv_M a_2a_0^*a_3^*$. We claim that $(a_0^*,a_1,a_2,a_3^*)$ satisfies the type above.
    First, $a_0^*a_1\equiv a_0a_1$ hence $p(a_0^*,a_1)$. By indiscernability, $a_0a_1\equiv_M a_1a_2$ hence $p(a_1,a_2)$. By choice, $a_2a_3^*\equiv_M a_0^*a_1$ hence $p(a_2,a_3^*)$. Finally $a_3^*a_0\equiv_M a_1a_2$ hence $p(a_3^*,a_0)$.
\end{proof}
% \begin{proof}
%     Let $(a_i)_{i<\omega}$ be an indiscernible sequence and $p(x,y) = \tp(a_0,a_1)$. We show that 
%     \[p(x_0,x_1)\cup p(x_1,x_2)\cup p(x_2,x_3)\cup p(x_3,x_0)\]
%     is a consistent partial type. By Theorem \ref{thm:modelheirindependence}, $a_i \indi h _M a_{<i}$ for all $i<\omega$ and some small model $M$.
%     \begin{claim}
%         It is enough to show that there exists $a_0^*$ such that $a_0^*\equiv_{M a_1} a_0$ with $a_0^*\ind_M a_2$.
%     \end{claim}
%     \begin{proof}[Proof of the claim] If such $a_0^*$ exists, we have $a_0^*\equiv_M a_2$ and $a_0^*\ind_M a_2$ hence by Lemma \ref{lm:stat_invert}, we have $a_0^*a_2\equiv_M a_2a_0^*$. Then, there exists $a_3^*$ such that $a_0^*a_2a_1\equiv_M a_2a_0^*a_3^*$. We check that $(a_0^*,a_1,a_2,a_3^*)$ satisfies the type above. First, $a_0^*a_1\equiv a_0a_1$ hence $p(a_0^*,a_1)$. By indiscernability, $a_0a_1\equiv_M a_1a_2$ hence $p(a_1,a_2)$. By choice, $a_2a_3^*\equiv_M a_0^*a_1$ hence $p(a_2,a_3^*)$. Finally $a_3^*a_0\equiv_M a_1a_2$ hence $p(a_3^*,a_0)$.
% \end{proof}

% By \ref{FEX} there exists $a_0^*\equiv_{Ma_1} a_0$ such that $a_0^*\ind_{Ma_1} a_2$. By \ref{SYM} we have $a_2\ind_{Ma_1} a_0^*$. As $a_2\indi h _M a_1$ and $a_1\indi h _M a_0^*$, we conclude $a_2\ind_{M} a_0^*$ using the weak transitivity assumption.
% \end{proof}

\begin{corollary}
    ACF is NSOP$_4$.
\end{corollary}

\begin{proof}
    In ACF, consider the relation $\indi\alg\ $. By Proposition \ref{prop:ACFalgebraicsatisfiesbasics} and \ref{prop:ACFsatisfiesINVEXTSTAT}, the invariant relation $\indi\alg\ $ satisfies \ref{SYM}, \ref{TRA}, \ref{FEX} and \ref{STAT} over models. Further it is easy to check that $\indi h\to \indi\alg\ $ (see Exercise \ref{exo:heirandcoheirinACF}) hence the conditions of Theorem \ref{thm:criterionNSOP4} are satisfied by \ref{TRA} and \ref{MON}.
\end{proof}

\begin{corollary}\label{cor:RGNSOP4}
    RG is NSOP$_4$.
\end{corollary}
\begin{proof}
    We check that $\indi \st$ satisfies the hypotheses of Theorem \ref{thm:criterionNSOP4}. By Propositions \ref{prop:propositionstrongindependenceinRG} and \ref{propositionpropertiesindepRGwithtypes} it remains to prove the weak transitivity property. Assume that $A\indi h_M B$, $B\indi h _M D$ and $A\indi\st_{MB} D$. From $A\indi h_M B$, $B\indi h _M D$ we deduce that $A\cap B\seq M$ and $B\cap D\seq M$. Let $uv\in AMD$ be such that $\models R(u,v)$. As $A\indi\st_{MB} D$, in particular $uv\seq AMB$ or $uv\seq MBD$. However $AMB\cap AMD\seq AM\cup(B\cap D)\seq AM$ and $MBD\cap AMD\seq MD$ hence hence $uv\seq AM$ or $uv\seq MD$. We conclude that $A\indi\st_M B$.
\end{proof}

\begin{exercise}
Using Lemma \ref{lm:LOCcharactergivesMS} and the proof of Theorem \ref{thm:criterionNSOP4}, prove the following:
if $\ind$ is an invariant relation satisfying \ref{SYM}, \ref{FEX}, \ref{STAT} over models and $\indi 0$ is invariant and satisfies right \ref{MON}, right \ref{BMON} and \ref{LOC}, if further:
    \[a\ind_{Md} b \text{ and } a\indi 0 _M d \text{ and } d\indi 0 _M b  \implies a\ind_M b\]
    then $T$ is NSOP$_4$.
\end{exercise}

\begin{exercise}
    A formula $\phi(x,y)$ with $\abs{x} = \abs{y}$ has the \textit{$n$-strong order property} if there exists $(a_i)_{i<\omega}$ such that $\phi(a_i,a_j)$ for all $i<j$ and 
    \[\phi(x_1,x_2)\wedge \phi(x_2,x_3)\wedge \ldots \phi(x_n,x_1)\]
    is inconsistent in $T$. Prove that $T$ has SOP$_n$ if and only if there is a formula $\phi(x,y)$ which has SOP$_n$ modulo $T$.
\end{exercise}

\begin{exercise}[*]
    Let $\ind$ be an invariant relation satisfying: \ref{SYM}, \ref{MON}, \ref{TRA}, \ref{FEX}, \ref{SFIN} and \ref{STAT} over models then $T$ is NSOP$_4$. Deduce that ACF is NSOP$_4$.
\end{exercise}

\begin{exercise}[*]
    We define the following property
    \begin{itemize}
        \item (\setword{freedom}{FREE}). If $A\ind_C B$ and $C\cap AB\seq D\seq C$ then $A\ind_D B$.
    \end{itemize}
    \begin{enumerate}
        \item In any set $S$ the relation $A\cap B\seq C$ satisfies \ref{FREE}.
        \item Prove that in RG, the relation $\indi\st$ satisfies \ref{FREE}.
        \item We propose to prove the following criterion, due to Conant: \begin{center} if $T$ is a theory where there is an invariant relation $\ind$ on small sets satisfying \ref{SYM}, \ref{FEX}, \ref{STAT} over every set and \ref{FREE}, then $T$ is NSOP$_4$.
        \end{center}
        The proof is a variant of the one of Theorem \ref{thm:criterionNSOP4}.
        \begin{enumerate}
            \item Let $(a_i)_{i<\omega}$ be an indiscernible sequence and $p(x,y) = \tp(a_0,a_1)$. Let $C = a_0\cap a_1$. Prove that $a_i\cap a_j = C$ for all $i,j$ and that the sequence $(b_i)_{i<\omega}$ defined by $b_i = a_i\setminus C$ is indiscernible over $C$.
            \item Prove that there exists $b_0^*$ with $b_0^*\equiv_{Cb_1} b_0$ and $b_0^*\ind_C b_2$ (\textit{Hint.} Use \ref{FEX} and \ref{FREE}.)
            \item Prove that $b_0^*b_2\equiv_C b_2b_0^*$.
            \item Conclude as in Theorem \ref{thm:criterionNSOP4} that $p(x_0,x_1)\cup p(x_1,x_2)\cup p(x_2,x_3)\cup p(x_3,x_0)$ is consistent.
        \end{enumerate}
    \item Deduce that RG is NSOP$_4$.
    \item (\textit{Bonus}. The property \ref{STAT} over every set is a strong property, but using Theorem \ref{thm:indiscernibleerdorado} one might redo the whole proof above assuming only \ref{STAT} over models.)
    \end{enumerate}
\end{exercise}

%% file: chapter4.tex
\chapter{Forking and dividing}\label{chapter:4}

\section{Generalities on dividing and forking}

\subsection{Dividing}
\begin{definition}[Dividing]
    Let $k\in \N$, $b$ a tuple and $C$ a set. We say that a formula $\phi(x,b)$ \textit{$k$-divides over $C$} if there exists a sequence $(b_i)_{i<\omega}$ in $\tp(b/C)$ such that the set $\set{\phi(x,b_i)\mid i<\omega}$ is $k$-inconsistent, i.e. the conjunction $\bigwedge_{j=1}^k \phi(x,b_{i_j})$ is inconsistent for all $i_1<\ldots<i_k<\omega$. A formula divides if there is some $k$ such that it $k$-divides.

    A (partial) type $\Sigma(x)$ \textit{divides over $C$} if there exists $\phi(x,b)$ which divides over $C$ and $\Sigma(x)\models \phi(x,b)$. In particular if $\Sigma$ is closed under conjunction, $\Sigma$ divides over $C$ if it contains a formula that does.
\end{definition}

\begin{example}
 In DLO, the formula $\phi(x,b_1,b_2)$ defined by $b_1<x<b_2$ $2$-divides over $\emptyset$. In ACF, let $\phi(x,b)$ be any nontrivial polynomial equation $P(X,b)=0$ which solutions are not in $\Q^\alg$. Then $\phi(x,b)$ $2$-divides over $\emptyset$. %To see that if $(b_i)_{i<\omega}$ are algebraically independent over $\emptyset$ then if $a\models \phi(x,b_0)\wedge \phi(x,b_1)$ then $a\in \Q^\alg$, a contradiction
\end{example}

\begin{exercise}
    If $\phi(\MM,b)\seq \psi(\MM,d)$ and $\psi(x,d)$ divides over $C$ then $\phi(x,b)$ divides over $C$. 
\end{exercise}

\begin{exercise}
    If $T$ is strongly minimal and $\phi(x, b)\seq \MM$ is infinite, then $\phi(x,b)$ does not divide over $\emptyset$. 
\end{exercise}

\begin{lemma}\label{lm:dividing_indiscernible}
        $\pi(x, b)$ divides over $C$ if and only if there exists a $C$-indiscernible sequence $(b_i)_{i<\omega}$ such that $\tp(b_0/C) = \tp(b/C)$ and $\bigcup_{i<\omega} \pi(x,b_i)$ is inconsistent.
\end{lemma}
\begin{proof}
    Assume that $\pi(x,b)$ divides over $C$, so there is $\phi(x,b)$ with $\pi(x,b)\models \phi(x,b)$ ($\phi(x,b)$ is a conjunction of formula from $\pi$, by compactness) and such that $\phi(x,b)$ divides over $C$. Let $(b_i)_{i<\omega}$ be a sequence with $\tp(b_0/C) = \tp(b/C)$ and such that $\set{\phi(x,b_i)\mid i<\omega}$ is $k$-inconsistent, for some $k$. Then the formula $\theta(y_1,\ldots,y_k)$ defined by $\neg (\exists x \bigwedge_i \phi(x,y_i)$ is in the EM-type of $(b_i)_{i<\omega}$ over $C$ (even over $\emptyset$). By Ramsey and compactness, let $(b_i')_{i<\omega}$ be an indiscernible sequence satisfying $\EM((b_i)_{i<\omega}/C)$, then $\set{\phi(x,b_i)\mid i<\omega}$ is also $k$-inconsistent. In particular it is inconsistent and so is $\bigcup_{i<\omega}\pi(x,b_i')$.

    Conversely, assume that $(b_i)_{i<\omega}$ is an indiscernible sequence such that $\tp(b_0/C) = \tp(b/C)$ and that $\bigcup_{i<\omega} \pi(x,b_i)$ is inconsistent. By compactness, there is a finite conjunction $\phi(x,b)$ of formulas from $\pi(x,b)$ such that $\set{\phi(x,b_i)\mid i<\omega}$ is inconsistent. By compactness, there exists $k<\omega$ such that $\phi(x,b_{i_1})\wedge \ldots \wedge \phi(x,b_{i_k})$ is inconsistent, so $\models \theta(b_{i_1},\ldots,b_{i_k})$. As $(b_i)_{i<\omega}$ is $C$-indiscernible, if $\theta(b_{i_1},\ldots,b_{i_k})$ holds for some ordered tuple $i_1<\ldots<i_k$, then it holds for all such tuples, hence $\set{\phi(x,b_i)\mid i<\omega}$ is $k$-inconsistent. As $\tp(b_0/C) = \tp(b/C)$, we conclude that $\pi(x,b)$ divides over $C$. 
    \end{proof}

\begin{proposition}\label{prop:dividing_characterisation_indiscerniblesequencestayindisceniblerovercanda}
    The following are equivalent:
    \begin{enumerate}
        \item $\tp(a/Cb)$ does not divide over $C$
        \item for any $C$-indiscernible sequence $(b_i)_{i<\omega}$ with $b_0 = b$ there is $a'\equiv_{Cb} a$ such that $(b_i)_{i<\omega}$ is $Ca'$-indiscernible.
        \item for any $C$ indiscernible sequence $(b_i)_{i<\omega}$ there exists $(b_i')_{i<\omega}$ with $(b_i')_{i<\omega}\equiv_{Cb}(b_i)_{i<\omega}$ such that $(b_i')_{i<\omega}$ is $Ca$-indiscernible.
    \end{enumerate}
\end{proposition}

\begin{proof}
    $(1)\implies (2)$. Let $(b_i)_{i<\omega}$ be a $C$-indiscernible sequence with $b_0 = b$. By Lemma \ref{lm:dividing_indiscernible}, for $p(x,y) = \tp(ab/C)$, the partial type $\bigcup_{i<\omega} p(x,b_i)$ is consistent, let $a''$ be a realisation. By Ramsey and compactness, there exists a $Ca''$-indiscernible sequence $(b_i')_{i<\omega}$ realising $\EM((b_i)_{i<\omega}/Ca'')$. As $\models p(a'',b_i)$ for all $i<\omega$, $p(a'',x)\seq \EM((b_i)_{i<\omega}/Ca'')$, hence $b_i'\models p(a'',y)$ for all $j<\omega$. In particular, $\tp(a''b_0'/C) = p(x,y)$ i.e. $a''b_0'\equiv_C ab$. As $(b_i)_{i<\omega}$ is $C$-indiscernible, we have $(b_i')_{i<\omega}\equiv_C (b_i)_{i<\omega}$. Using an automorphism, there exists $a'$ such that $(b_i')_{i<\omega}a''\equiv_C (b_i)_{i<\omega}a'$, so $(b_i)_{i<\omega}$ is $Ca'$-indiscernible. As $a'b_0\equiv_C a''b_0'$ and $a''b_0'\equiv_C ab_0$ we have $a'\equiv_{Cb} a$.

    $(2)\implies (1)$. Let $p(x,b) = \tp(a/Cb)$ where $p(x,y)$ is a type over $C$. Let $(b_i)_{i<\omega}$ be any $C$-indiscernible sequence and $a'$ as in $(2)$. As $a'\equiv_{Cb} a$ we have $p(a',b)$. As $(b_i)_{i<\omega}$ is $Ca'$-indiscernible, we also have $p(a',b_i)$, for all $i<\omega$, so $\bigcup_{i<\omega} p(x,b_i)$ is consistent. By Lemma \ref{lm:dividing_indiscernible}, $\tp(a/Cb)$ does not divide over $C$.

    $(2)\iff (3)$. Clear via an appropriate automorphism.
\end{proof}

\begin{definition}
    We define the \textit{non-dividing independence relation} $\indi d$ by
    \[A \indi d_C B \iff  \text{$\tp(A/BC)$ does not divide over $C$}.\] Equivalently, $A\indi d_C B$ if for any enumeration $a$ of $A$ and $b$ of $B$ we have: for any indiscernible sequence $(b_i)_{i<\alpha}$ with $b_0 = b$ there is $a'$ such that $a'b_i \equiv_{C} ab$ for all $i<\omega$.
\end{definition}

Oddly enough, the non-dividing independence relation is also sometimes orally called the dividing independence relation.

\begin{theorem}\label{thm:propertieofdividing}
    The relation $\indi d$ is invariant and satisfies \ref{EX}, \ref{FIN}, left and right \ref{MON}, left and right \ref{NOR}, right \ref{BMON}, left \ref{TRA} and \ref{AREF}.
\end{theorem}

\begin{proof}
    The relation $\indi d$ is invariant because for any $\sigma\in \Aut(\MM)$, $\phi(x,b)$ divides over $C$ if and only if $\phi(x,\sigma(b))$ divides over $\sigma(C)$. The property \ref{EX} is trivial since $\tp(a/C)$ does not divides over $C$.
    The property \ref{FIN} holds for $\indi d$ since dividing is witnessed at the level of formulas, which only mention a finite number of elements from $A$. Assume that $A\indi d_C B$. If $A'\seq A$ and $B'\seq B$ then $\tp(A'/B'C)$ consist of formulas in $\tp(A/BC)$ hence no formula in $\tp(A'/B'C)$ divides over $C$ hence $A'\indi d_C B'$, which yields left and right \ref{MON}. Right \ref{NOR} holds by definition. For left \ref{NOR}, if $a\indi d_C b$ and $(b_i)_{i<\omega}$ is $C$-indiscernible with $b_0 = b$, then there is $a'$ such that $a'b_i\equiv_C ab$. Then $a'Cb_i\equiv_C aCb$ hence $aC\indi d _C b$. We prove \ref{BMON}: assume $a\indi d_C b$ and $C\seq D \seq b$. If $(b_i)_{i<\omega}$ is a $D$-indiscernible sequence with $b_0 = b$, then $(b_iD)_{i<\omega}$ is $C$-indiscernible hence there exists $a'$ with $a'b_iD\equiv_C abD$, so $a'b_i\equiv_D ab$ hence $a\indi d_D b$. We prove left \ref{TRA}: assume that $C\seq B\seq D$ and $D\indi d_B a$ and $B\indi d _C a$. Let $(a_i)_{i<\omega}$ be a $C$-indiscernible sequence with $a_0 = a$. As $B\indi d_C a$ there exists $B'\equiv_{aC} B$ such that $(a_i)_{i<\omega}$ is $B'$-indiscernible. Using an automorphism, there is $D'$ such that $D'B'\equiv_{aC} DB$. By invariance, $D'\indi d_{B'} a$ hence as $(a_i)_{i<\omega}$ is $B'$-indiscernible, there is $D''\equiv _{B'} D'$ such that $(a_i)_{i<\omega}$ is $D''$-indiscernible. Then $D''B'\equiv_{aC} D'B'\equiv_{aC} DB$ hence $D''\equiv_{aC} D$ so we conclude $D\indi d _C a$. For \ref{AREF}, if $a\notin \acl(C)$ then there exists a $C$-indiscernible sequence $(a_i)_{i<\omega}$ of distinct elements with $a_0 = a$. This sequence witnesses that the formula $x = a\in \tp(a/Ca)$ $2$-divides over $C$, hence $a\nindi d _B a$.
\end{proof}

In general, $\indi d$ is not an Adler independence relation, for instance, it need not be symmetric. It will be in the context of simple theories. However, we have the following connection between $\indi d$ and any Adler independence relation.

\begin{proposition}\label{prop:dividingindepstrongerthanAIR}
    $\indi d\to \ind$ for any Adler independence relation $\ind$.
\end{proposition}

\begin{proof}
    Assume that $\ind$ is an Adler independence relation and $a\indi d_C b$. Let $(b_i)_{i<\omega}$ be an $\ind$-Morley sequence over $C$ with $b_0 = b$. We may assume that $(b_i)_{i<\omega} $ is $C$-indiscernible using Lemma \ref{lemma:adlerthmoppmorleysequenceexistence} and $\ind=\indi\opp$. As $a\indi d_C b$ there exists $a'\equiv_{Cb} a$ such that $(b_i)_{i<\omega}$ is $Ca'$-indiscernible, by Proposition \ref{prop:dividing_characterisation_indiscerniblesequencestayindisceniblerovercanda}. The sequence $(b_i)_{i<\omega}$ is a $Ca'$-indiscernible $\ind$-Morley sequence over $C$ in $\tp(b/Ca')$ so by Lemma \ref{lemma:adlerthmoppmorleysequenceassumption} we have $a'\ind_C b$. By invariance, we have $a\ind_C b$.
\end{proof}

Another use of Neuman's Lemma is the following (see \cite{conant2022separation}):

\begin{proposition}
    $\indi d\rightarrow \indi a$
\end{proposition}
\begin{proof}
    Suppose $A\nindi a _C B$ and let $a_0\in (\acl(AC)\cap \acl(BC))\setminus \acl(C)$. Let $a = (a_0,\ldots,a_n)$ be a tuple enumerating the orbit of $a$ over $\Aut(\MM/BC)$. Let $\phi(x,b)$ be a formula with $\phi(x,y)\in \LL(C)$ which isolates $\tp(a/BC)$, so that $a$ enumerates $\phi(\MM)$. As $a_0\notin \acl(C)$ the orbit of every coordinate of $a$ under $\Aut(\MM/C)$ is infinite. By Neuman's Lemma (Fact \ref{fact:neumanslemma}) there exists $a^1 \equiv _C a$ such that $a^1\cap a = \emptyset$ (as set of coordinates). By iterating applications of Neuman's Lemma, there is a sequence $( a^i)_{i<\omega}$ of pairwise disjoint tuples such that $ a^i \equiv _C  a$. By shifting by an automorphism, let $ b_i$ be such that $ a^i b_i \equiv_C  a b$. Then $a^i$ enumerates the solutions of $\phi(x,b_i)$, so $\set{\phi(x, b_i)\mid i<\omega}$ is $2$-inconsistent, so $a\nindi d _C B$.
\end{proof}

\begin{remark}
    It was believed for a long time (at least since 2008!) that the implication $\indi d\to \indi a ^M$ were true in general. According to Proposition \ref{prop:monotonisation}, it is enough that $\indi d$ satisfies the right version of \ref{NOR}, \ref{MON}, \ref{BMON} and \ref{CLO}. By Theorem \ref{thm:propertieofdividing}, one only needs to check right \ref{CLO}. It was recently observed that $\indi d$ does not satisfy right \ref{CLO} (so $\indi d \neq (\indi d)^M$!), even worst, the implication $\indi d\to \indi a ^M$ does not hold in general. 
\end{remark}

\begin{exercise}
    Prove that $a\indi d_C b$ if and only if for any $C$-indiscernible sequence $(b_i)_{i<\alpha}$ with $b_0 = b$ there is $a'$ such that $a'b_i \equiv_{C} ab$ for all $i<\omega$.
\end{exercise}

\begin{exercise}\label{exo:divinDLOnotTRA}
    In DLO, assume that $b_1<c<a<b_2$. Check that $a\indi d_c b_1b_2$ and $a\indi d_\emptyset c$, deduce that $\indi d$ does not satisfy right \ref{TRA}.
\end{exercise}

\begin{exercise}\label{exo:vancampenhausenDLO}
    In DLO, prove that $A\indi d _C B$ if and only if $A\cap [b_1,b_2]\neq \emptyset$ implies $C\cap [b_1,b_2]\neq \emptyset$, for all $b_1<b_2\in B$. %thanks Thomas for this one
\end{exercise}

\begin{exercise}\label{exo:dividingsatstrongfinitecharacter}
    Prove that $\indi d$ satisfies \ref{SFIN}. Deduce that if $\indi d$ satisfies \ref{SYM} then $\indi d$ satisfies \ref{LOC}. %uses existence
\end{exercise}

\begin{exercise}
    Is there an independence relation $\ind\neq \indi d$ such that $\indi d = \indi M$ ?
\end{exercise}

\subsection{Digression: dividing in fields}

In this subsection, $T$ is a theory of fields in a language extending $\LLr$. Let $\K$ be a monster model of $T$.

\begin{lemma}\label{lm:heirfields}
Let $F, A,B$ be a small subfields of $T$ with $F\seq A\cap B$. If $A\indi h_F B$ or $A\indi u _F B$ then $A\indi\ld _F B$.
\end{lemma}
\begin{proof}
As $\indi h = (\indi u)^\opp$ and $\indi \ld$ is symmetric, it is enough to prove that $\indi h\to \indi \ld$. Let $a_1,\ldots,a_n\in A$ be linearly dependent over $B$. Then for some $b_1,\ldots,b_n$ we have $\sum_i a_ib_i = 0$. As $\tp(A/B)$ is an heir, there exists $c_1,\ldots,c_n\in F$ such that $\sum_i a_ic_i = 0$, hence $a_1,\ldots,a_n$ are linearly independent over $F$.
\end{proof}

\begin{lemma}\label{lm:fieldsformuladivides}
    Let $b = (b_1,\ldots,b_n)$ be a tuple linearly independent over a model $F$ of $T$. Then the formula $\phi(x,y)$ defined by 
    \[(x_1,\ldots,x_n)\neq (0,\ldots,0)\wedge x_1b_1+\ldots+ x_nb_n = 0\]
    $n$-divides over $F$.
\end{lemma}

\begin{proof}
    Using Corollary \ref{cor:indusatisfiesFEXovermodels} there exist a $\indi u$-Morley sequence $(b^i)_{i<\omega}$ in $\tp(b/F)$ with $b^0 = b$. Let $\Sigma(x)$ be the partial type consisting of all equations $\phi(x,b^j) = (x_1,\ldots,x_n)\neq (0,\ldots,0)\wedge (\sum_i x_ib_i^j=0)$, for all $j<\omega$. Using Lemma \ref{lm:heirfields} we have that $b^j\indi\ld _F (b^i)_{i<j}$ hence the determinant of each $n$ of those equations $(\sum_i x_ib_i^j=0)$ is nonzero. It follows that $(0,\ldots,0)$ is the only realisation of $n$ of those equations. As the tuple $(0,\ldots,0)$ does not satisfy $\phi(x,b_i)$ by definition, the partial type $\Sigma(x)$ is $n$-inconsistent.
\end{proof}

\begin{proposition}
Let $F$ be a small model of $T$ and $A,B$ be $\acl$-closed small subsets of $\K$ containing $F$. If $A\indi d _F B$ then $A\indi \ld_F B$. 
\end{proposition}
\begin{proof}
Assume that $A\indi d _F B$. If $A\nindi\ld_F B$, then there exists a tuple $b = (b_1,\ldots,b_n)$ from $B$ such that $\sum_ia_ib_i$ for some $(a_1,\ldots,a_n)\neq (0,\ldots,0)$ from $A$ and $b_1,\ldots,b_n$ linearly independent over $F$. Hence the formula $(x_1,\ldots,x_n)\neq (0,\ldots,0)\wedge (\sum_i x_ib_i = 0)$ is in $\tp(A/B)$, which contradicts Lemma \ref{lm:fieldsformuladivides}.
\end{proof}

\begin{remark}
    Chatzidakis also proved that in any theory of fields if $A\indi d_F B$ (or $A\indi u_F B$) then $\K$ is a separable extension of the field compositum $A(B)$ and $\acl(AB)\cap A^\alg (B^\alg) = A(B)$.
\end{remark}

\subsection{Forking and forcing the extension axiom}

\begin{definition}
    A formula $\psi(x,b)$ \textit{forks over $C$} if $\psi(x,b)\models \bigvee_{i=1}^n\phi(x,b_i)$ and each $\phi(x,b_i)$ divides over $C$. A (partial) type $\pi(x)$ forks over $C$ if it implies a formula that forks over $C$.

    We define the \textit{non-forking independence relation} $\indi f$ by 
    \[A\indi f_C B \iff \tp(A/BC)\text{ does not fork over $C$}\]
\end{definition}

We clearly have $\indi f\to \indi d$.

\begin{example}[A formula which forks and does not divides]\label{example:orientedcircle}
    We describe the theory of an oriented circle, which we will denote $\Circle$. \\
\noindent\begin{minipage}{0.3\textwidth}
\begin{center}
\begin{tikzpicture}
   \draw [blue,thick,domain=-50:110] plot ({cos(\x)}, {sin(\x)});
   \draw [black,thin,domain=0:360] plot ({cos(\x)}, {sin(\x)});
   \path (0:0) ++(-50:1) coordinate (A);
   \fill[black] (A) circle[radius=1pt] ++(-50:1em) node {x};
   \path (0:0) ++(10:1) coordinate (B);
   \fill[black] (B) circle[radius=1pt] ++(10:1em) node {y};
   \path (0:0) ++(110:1) coordinate (C);
   \fill[black] (C) circle[radius=1pt] ++(110:1em) node {z};
   \draw [black,->,thin,domain=30:55] plot ({1.7*cos(\x)}, {1.7*sin(\x)});
\end{tikzpicture}
\end{center}
\end{minipage}\begin{minipage}{0.7\textwidth}% adapt widths of minipages to your needs
    The domain is an infinite set and the language consists of a single ternary relation $C(x,y,z)$ such that 
    \begin{itemize}
        \item $C(x,y,z)\iff C(y,z,x)\iff C(z,x,y)$
        \item for all $x$ the relation on $(y,z)$ defined by $C(x,y,z)$ is a strict linear order which is dense without endpoints on the domain
    \end{itemize}
    One easily shows that the theory $\Circle$ has quantifier elimination in the language $\set{C}$.
    An easy and concrete model of this theory is given by defining the relation 
    \[C(x,y,z) := (x<y<z)\vee (y<z<x)\vee (z<x<y)\] in a dense linear order without endpoint. 
\end{minipage}\\
\noindent\begin{minipage}{0.7\textwidth}% adapt widths of minipages to your needs
    As in DLO, for any $a,c$ the formula $\phi(y; a,c)$ given by $C(a,y,c)$ divides over $\emptyset$. To see this easily, consider pairs $(a_i,c_i)_{i<\omega}$ such that $C(a_0,a_i,c_i)$ and each $a_ic_i$ defines disjoint arcs via $\phi$. It is easy to see that the type of a pair of distinct elements over $\emptyset$ is unique hence $a_ic_i\equiv ac$ for all $i<\omega$. The family obtained $(\phi(x,a_i,b_i))_{i<\omega}$ is $2$-inconsistent. 
    
    The formula $x= x$ does not divide over $\emptyset$. Now observe that for any $a,b,c$ such that $C(a,b,c)$ we have 
    \[(x = x)\models C(a,x,b)\vee C(b,x,c) \vee C(c,x,a)\vee x = a\vee x = b\vee x = c\]
    hence as the formulas in the disjunct divides over $\emptyset$ the formula $x = x$ forks over $\emptyset$. Note that two distinct points $a,b$ are enough to get that $x = x$ forks over $\emptyset$.
\end{minipage}%
%\hfill
\begin{minipage}{0.3\textwidth}
\begin{center}
\begin{tikzpicture}
   \draw [blue,thick,domain=-50:90] plot ({cos(\x)}, {sin(\x)});
   \draw [blue,thick,domain=110:170] plot ({cos(\x)}, {sin(\x)});
   \draw [blue,thick,domain=190:230] plot ({cos(\x)}, {sin(\x)});
   \draw [blue,thick,domain=240:260] plot ({cos(\x)}, {sin(\x)});
   \draw [blue,thick,domain=270:280] plot ({cos(\x)}, {sin(\x)});
   \draw [black,thin,domain=0:360] plot ({cos(\x)}, {sin(\x)});
   \path (0:0) ++(-50:1) coordinate (A);
   \fill[black] (A) circle[radius=1pt] ++(-50:1em) node {$a_0$};
   \path (0:0) ++(90:1) coordinate (C);
   \fill[black] (C) circle[radius=1pt] ++(90:1em) node {$c_0$};
   \path (0:0) ++(110:1) coordinate (B);
   \fill[black] (B) circle[radius=1pt] ++(110:1em) node {$a_1$};
   \path (0:0) ++(170:1) coordinate (D);
   \fill[black] (D) circle[radius=1pt] ++(170:1em) node {$c_1$};
   \path (0:0) ++(190:1) coordinate (E);
   \fill[black] (E) circle[radius=1pt] ++(190:1em) node {$a_2$};
   \path (0:0) ++(230:1) coordinate (F);
   \fill[black] (F) circle[radius=1pt] ++(230:1em) node {$c_2$};
   \path (0:0) ++(240:1) coordinate (G);
   \fill[black] (G) circle[radius=1pt] ++(230:1em) node {$ $};
   \path (0:0) ++(260:1) coordinate (H);
   \fill[black] (H) circle[radius=1pt] ++(260:1em) node {$ $};
   \path (0:0) ++(270:1) coordinate (I);
   \fill[black] (I) circle[radius=1pt] ++(270:1em) node {$ $};
   \path (0:0) ++(280:1) coordinate (J);
   \fill[black] (J) circle[radius=.9pt] ++(280:1em) node {$ $};
   \path (0:0) ++(287:1) coordinate (K);
   \fill[black] (K) circle[radius=.8pt] ++(287:1em) node {$ $};
   \path (0:0) ++(292:1) coordinate (L);
   \fill[black] (L) circle[radius=.6pt] ++(292:1em) node {$ $};
   \path (0:0) ++(295:1) coordinate (L);
   \fill[black] (L) circle[radius=.6pt] ++(295:1em) node {$ $};
   \draw [black,->,thin,domain=30:55] plot ({1.7*cos(\x)}, {1.7*sin(\x)});
\end{tikzpicture}
$\text{~}$\\
\begin{tikzpicture}
    \path (0:0) ++(170:1) coordinate (X);
   \fill[black] (X) circle[radius=1pt] ++(170:1em) node {x};
    \path (0:0) ++(-20:1) coordinate (A);
   \fill[black] (A) circle[radius=1pt] ++(-20:1em) node {a};
   \path (0:0) ++(110:1) coordinate (B);
   \fill[black] (B) circle[radius=1pt] ++(110:1em) node {b};
   \path (0:0) ++(250:1) coordinate (C);
   \fill[black] (C) circle[radius=1pt] ++(250:1em) node {c};
   \draw [blue,thick,domain=-20:110] plot ({cos(\x)}, {sin(\x)});
   \draw [red,thick,domain=110:250] plot ({cos(\x)}, {sin(\x)});
   \draw [green,thick,domain=250:340] plot ({cos(\x)}, {sin(\x)});
    \draw [black,thin,domain=0:360] plot ({cos(\x)}, {sin(\x)});
   \draw [black,->,thin,domain=30:55] plot ({1.7*cos(\x)}, {1.7*sin(\x)});
\end{tikzpicture}
\end{center}
\end{minipage}
\end{example}

\begin{definition}
    Given any relation $\ind$, we define the relation $\indi *$:
    \[A\indi * _C B \iff \text{for all $D\supseteq B$, there exists $A'\equiv_{BC} A$ with $A'\ind_C D$}\]
\end{definition}

We always have $\indi * \to \ind$. By definition, if $\indi 0 \to \ind$ and $\indi 0$ satisfies \ref{EXT}, then $\indi 0\to \indi *$.

\begin{proposition}\label{prop:forcingextpreservation}
    If $\ind$ is invariant and satisfies left and right \ref{MON} then $\indi *$ is invariant and satisfies left and right \ref{MON}, right \ref{NOR} and \ref{EXT}. If $\ind$ satisfies one of the following property: right \ref{BMON}, left \ref{TRA}, left \ref{NOR}, \ref{AREF} then so does $\indi *$.
\end{proposition}
\begin{proof}
    Assume that $\indi *$ is invariant and satisfies left and right \ref{MON}. Assume that $ABC\equiv A'B'C'$ and let $D'$ be such that $B'\seq D'$. There exists $D$ such that $B\seq D$ and an automorphism $\sigma$ such that $\sigma(ABCD) = A'B'C'D'$. If $A\indi *_C B$ there exists $A''\equiv_{BC} A$ such that $A''\indi *_C D$. By invariance $\sigma(A'') \indi * _{C'} D'$ and as $A''\equiv_{BC} A$ we have $\sigma(A'')\equiv_{B'C'} A'$, hence $A'\indi *_{C'} B'$.

    If $A\indi *_C B$, $A_0\seq A$ and $B_0\seq B$. Let $D$ be such that $B_0\seq D$, then $B\seq DB_0$ so there exists $A'\equiv_{BC} A$ witnessed by $\sigma$ such that $A'\ind_C DB_0$. By right \ref{MON} we have $A'\ind_C D$. Also, let $A_0' = \sigma(A)$, then $A_0'\seq A'$ is such that $A_0'\equiv_{BC} A_0$. In particular $A_0'\equiv_{B_0C} A$ and by left \ref{MON} we have $A_0'\ind_C D$. We conclude that for all $B_0\seq D$ there exists $A_0'\equiv_{B_0C} A_0$ such that $A_0\ind_C D$ so $A_0\indi *_C B_0$.
     
     We prove \ref{EXT}. Assume that $A\indi * _C B$ and $B\seq D$. Let $a = (a_i)_{i<\alpha}$ be an enumeration of $A$.
     \begin{claim}
         There exists $p\in S_\alpha(CD)$ extending $\tp(a/BC)$ such that for all $\kappa$ and for all $\kappa$-saturated model $M$ containing $CD$ there exists $a''\models p$ with $a''\ind_C M$.
     \end{claim}
     \begin{proof}[Proof of the claim]
         Assume not. Then for all $p\in S_\alpha(CD)$ extending $\tp(a/BC)$ there exists $\kappa(p)$ and a $\kappa(p)$-saturated model $M(p)$ containing $CD$ such that for all $a''\models p$ we have $a''\nind_C M(p)$. As $S_\alpha(CD)$ is small, let $\kappa = \max \set{\kappa(p)\mid p\in S_{\alpha}(CD)}$ and let $M$ be a $\kappa$-saturated model. By the assumption $a\indi *_C B$, there exists $a'\equiv_{BC} a$ such that $a'\ind_C M$. Let $q = \tp(a'/CD)$. Clearly $q$ extends $\tp(a/BC)$. As $M$ is $\kappa$-saturated, it is $\kappa$-universal, hence as $\kappa \geq \kappa(q)$ there is an automorphism $\sigma$ over $CD$ such that $\sigma(M(q))\seq M$. As $a'\ind_C M$ we have $a'\ind_C \sigma(M(q))$ by \ref{MON} hence by applying $\sigma^{-1}$ and by invariance, we have $\sigma^{-1}(a')\ind_C M(q)$. As $\sigma^{-1}(a')\models q$ we get a contradiction.
     \end{proof}
         Let $p$ be as in the claim and let $a'$ be a realisation of $p$. Let $E$ be a superset of $D$. Let $M$ be a $\abs{E}$-saturated model containing $CD$. Then by the claim there exists $a''\equiv_{CD} a'$ such that $a''\ind_C M$. Using $\abs{E}$-saturation of $M$ and invariance we may assume that $E\seq M$ (as in the proof of the claim) hence by \ref{MON} we have $a''\ind_C E$. We conclude that $a'\indi * _C D$ hence as $a'\equiv_{BC} a$, $\indi *$ satisfies \ref{EXT}.

     The property right \ref{NOR} follows from \ref{EXT}. We now assume that $\ind$ is invariant and satisfies left and right \ref{MON}.

    Assume that $\ind$ satisfies right \ref{BMON}. Assume that $A\indi * _C B$ with $C\seq B_0\seq B$. Then for any $D$ with $B\seq D$ there exists $A'\equiv_{BC} A$ such that $A'\ind_C D$, hence $A'\ind_{B_0} D$ so $A\indi *_{B_0} B$.

    Assume that $\ind$ satisfies left \ref{TRA}. We use here the clearly equivalent alternative definition of $\indi *$:
    $A\indi * _C B$ if for all $B\seq D$ there exists $D'\equiv_{AC} D$ such that $A\ind_C D'$. Assume that $D\indi *_B A$ and $B\indi * _C A$ for $C\seq B\seq D$. Let $\hat A$ be a superset of $A$. As $B\indi * _C A$ and $D\indi *_B A$, there exists $\hat A '\equiv_{A} \hat A$ with $B\ind_C \hat A'$ and there exists $\hat A ''\equiv_{D} \hat A$ such that $D\ind _B \hat A''$. Then $\hat A''\equiv_B \hat A'$ hence $B\ind_C \hat A''$ by invariance. By left \ref{TRA} we have $D\ind_C \hat A''$ hence we conclude $D\indi *_C A$.

    Assume that $\ind$ satisfies left \ref{NOR}. Assume that $A\indi *_C B$. Then for all $B\seq D$ there is $A'\equiv_{BC} A$ such that $A'\ind_C D$. As $A'C\equiv_BC AC$ and $AC\ind_C D$ we conclude $AC \indi *_C B$ 

    If $\ind $ satisfies \ref{AREF} then so does $\indi *$ since $\indi *\to \ind$.
\end{proof}

\begin{lemma}\label{lm:forkexistence_equivalent}
Suppose that $\ind$ is an invariant relation. The following are equivalent:
\begin{enumerate}
    \item $\indi *$ satisfies \ref{EX};
    \item $\ind$ satisfies \ref{FEX};
    \item there exists an invariant relation $\indi 0$ satisfying \ref{FEX} with $\indi 0\rightarrow \ind$.
\end{enumerate}
\end{lemma}
\begin{proof}
$(1) \implies (2)$. Let $A,B,C$ are given. By $(1)$, we have $A\indi*_C C$. In particular for $D = B$, there exists $A'\equiv_C A$ such that $A'\ind _C B$ so $\ind$ satisfies \ref{FEX}. 

$(2)\implies (3)$ is clear, take $\indi 0 = \ind$. 

$(3)\implies (1)$. Let $A,C$ be given, we want to show that $A\indi * _C C$. Let $D\supseteq C$, and by \ref{FEX} for $\indi 0$ there exists $A'\equiv_C A$ with $A'\indi 0 _C D$. As $\indi 0\rightarrow \ind$, $A'\ind _C D$ hence $A\indi * _C C$.
\end{proof}

\begin{exercise}
    Prove that if $\ind$ is invariant and satisfies \ref{MON} and \ref{SFIN} then $\indi *$ satisfies \ref{SFIN}.
\end{exercise}

\begin{exercise}[*]
    Prove that if $\ind$ is invariant and satisfies \ref{MON} and \ref{FIN} then $\indi *$ satisfies \ref{FIN}.
\end{exercise}

\begin{theorem}\label{thm:dividingstarredequalsforking}
    $\indi f = \indi d ^*$. 
\end{theorem}
\begin{proof}
    Observe that $a\indi f_C b$ if and only if for all $B\seq D$ the (partial) type $\pi_D(x)$ defined by 
    \[\tp(a/Cb)\cup \set{\neg \phi(x,d)\mid d\seq D,  \text{$\phi(x,d)$ divides over $C$}}\]
    By compactness, $\pi_D$ is inconsistent if and only if there exist a finite number of formula $\phi_i(x,d)$ which divides over $C$ and such that $\tp(a/Cb)\models \bigvee_i \phi_i(x,d)$, the latter is equivalent to $\tp(a/Cb)$ forking over $C$. The theorem follows.
\end{proof}

\begin{corollary}\label{cor:forkingisanairiffLOC}
    $\indi f$ always satisfies \ref{FIN}, left and right \ref{MON}, left \ref{NOR}, right \ref{BMON}, left \ref{TRA}, \ref{FIN} and \ref{EXT}.
    In particular, $\indi f$ is an AIR if and only if $\indi f$ satisfies \ref{LOC}.
\end{corollary}
\begin{proof}
    By putting together Proposition \ref{prop:forcingextpreservation}, Theorem \ref{thm:propertieofdividing} and Theorem \ref{thm:dividingstarredequalsforking}, it remains to check that $\ind$ satisfies \ref{FIN}, which is trivial by the definition of $\indi f$. The property \ref{LOC} is the only one missing from the definition of an AIR.
\end{proof}

\begin{remark}
    If $\indi f$ is an AIR, then $\indi d = \indi f$ by Proposition \ref{prop:dividingindepstrongerthanAIR} and Theorem \ref{thm:dividingstarredequalsforking}.
\end{remark}

\begin{example}
We come back to Example \ref{example:orientedcircle}, which cumulates several pathological behaviours around forking and dividing and $\indi *$.
\begin{enumerate}
    \item $\indi f\neq \indi d$ in general: $a\indi d _\emptyset \emptyset$ whereas $a\nindi f_\emptyset \emptyset$ (even at the level of formula, there is a forking formula that does not divides).
    \item $\indi f$ does not satisfy \ref{EX}: $a\nindi f_\emptyset \emptyset$, equivalently $\indi d$ does not satisfy \ref{FEX} (by Lemma \ref{lm:forkexistence_equivalent})
    \item $\indi d$ does not satisfy \ref{EXT}\footnote{Note that \ref{FEX} and \ref{EXT} would be equivalent if $\indi d$ satisfies right \ref{TRA} by Proposition \ref{prop:basicpropertieswiththeory} (c,d), but by Example \ref{exo:divinDLOnotTRA}, $\indi d$ does not satisfy right \ref{TRA} in general.}: for any distinct $a,b,c$, we have $a\indi d_\emptyset \emptyset$ but there is no $a'\equiv_{\emptyset} a$ such that $a'\indi d_\emptyset  bc$, because $\tp(a'/bc)$ contains either the formula $C(b,x,c)$ or $C(c,x,b)$ or $x = b$ or $x = c$ all of which divides over $\emptyset$.
    \item Forcing the extension axiom $\indi *$ does not preserve \ref{EX} in general: $\indi d$ always satisfy \ref{EX} and $\indi f$ does not.
\end{enumerate}     
\end{example}

\begin{corollary}\label{cor:forcingBMONandEXT}
    If $\ind$ is invariant and satisfies left and right \ref{MON} then $\indi M ^*$ is invariant and satisfies left and right \ref{MON}, right \ref{NOR}, right \ref{CLO}, right \ref{BMON} and \ref{EXT}. Further $\indi M ^*\to \indi M\to \ind$.
\end{corollary}
\begin{proof}
    By Proposition \ref{prop:monotonisation}, $\indi M$ satisfies right \ref{MON} and right \ref{BMON}. By Proposition \ref{prop:forcingextpreservation}, $\indi M ^*$ satisfies right \ref{MON}, right \ref{BMON} and \ref{EXT}. By Proposition \ref{prop:basicpropertieswiththeory}, $\indi M ^*$ further satisfies right \ref{NOR} and right \ref{CLO}. In particular Proposition \ref{prop:monotonisation} applies and $\indi M ^*\to \indi M$.
\end{proof}

\color{blue}
Recall that $\indi m$ is the ``naive" monotonisation from Definition \ref{def:naivemonotonisation}.

\begin{corollary}\label{prop:forcingnaiveBMONandEXT}
    If $\ind$ is invariant and satisfies left and right \ref{MON} then $\indi m ^*$ is invariant and satisfies left and right \ref{MON}, right \ref{NOR}, right \ref{CLO}, right \ref{BMON} and \ref{EXT}. Further $\indi m ^*\to \indi m\to \ind$.
\end{corollary}
\begin{proof}
    By Proposition \ref{prop:naivemonotonisation}, $\indi m$ satisfies left and right \ref{MON} and right \ref{BMON}. By Proposition \ref{prop:forcingextpreservation}, $\indi m ^*$ satisfies right \ref{MON}, right \ref{BMON} and \ref{EXT}. It is standard that \ref{EXT} implies right \ref{NOR} and right \ref{CLO}. As $\indi m ^*\to \ind$ and $\indi m ^*$ satisfies the right-sided versions of \ref{NOR}, \ref{MON} and \ref{BMON}, Proposition \ref{prop:naivemonotonisation} applies and $\indi m ^*\to \indi m$.
\end{proof}

\begin{corollary}\label{cor:twomonsamextension}
    Assume that $\ind$ satisfies two-sided \ref{NOR} and \ref{MON} then $\indi M^* = \indi m ^*$.
\end{corollary}

\begin{proof}
    We have $\indi M^*\to \indi m^*$ by Propositions \ref{prop:naivemonotonisation} and \ref{prop:forcingextpreservation} because $\indi M^*$ satisfies right \ref{BMON} and \ref{EXT}
    and $\indi M^*\to \ind$ . By Corollary \ref{cor:forcingBMONandEXT} $\indi m ^*$ satisfies left and right \ref{NOR}, \ref{MON}, right \ref{CLO}, right \ref{BMON} and \ref{EXT}. As $\indi m ^*\to \ind$ and $\indi m ^*$ satisfies right \ref{CLO}, right \ref{BMON}, we get $\indi m^*\to \indi M$. As  $\indi m ^*$ satisfies \ref{EXT}, we conclude $\indi m^* \to \indi M^*$.
\end{proof}

\color{black}

\begin{exercise}
     Prove that $\indi f$ satisfies \ref{SFIN}. Deduce that $\indi f$ satisfies \ref{LOC} if and only if $\indi f$ satisfies \ref{SYM} and \ref{EX}.
\end{exercise}

\begin{exercise}
    Is there an independence relation $\ind$ different from $\indi u$ such that $\indi * = \indi u$? % maybe some sort of 'filter' independence relation?
\end{exercise}

\subsection{Forking and the independence theorem}

\begin{theorem}\label{thm:forcingBMON+EXTforking}
    Let $\indi 0$ be an invariant relation satisfying right \ref{MON}, right \ref{BMON} and \ref{LOC}.

  Let $\ind$ be an invariant relation satisfying left and right \ref{MON} and the following property ($\indi 0$-amalgamation over models):
  \begin{center}
      if $c_1\equiv_M c_2$ and $c_1\ind_M a$, $c_2\ind_M b$ and $a\indi 0 _M b$
      then there exists $c$ with $c\ind_E ab$ and $c\equiv_{Ma} c_1$, $c \equiv_{Mb} c_2$.
  \end{center}
  Then $\indi{M}^{*} \rightarrow \indi f$. 
\end{theorem}
\begin{proof}
By Corollary \ref{cor:forcingBMONandEXT}, the relation ${\indi{M}} ^{*}$ satisfies satisfies left and right \ref{MON}, right \ref{NOR}, right \ref{CLO}, right \ref{BMON} and \ref{EXT}.

We show that $\indi M ^* \rightarrow \indi d$, the result follows from $\indi f = {\indi d}^*$ (Theorem \ref{thm:dividingstarredequalsforking}). 

Assume that $a\indi M ^* _C b$, for some $a,b,C$. Let $(b_i)_{i<\omega}$ be a $C$-indiscernible sequence with $b = b_0$. By Lemma \ref{lm:LOCcharactergivesMS}, there exists a model $M\supseteq C$ such that $(b_i)_{i<\omega}$ is an $M$-indiscernible $\indi 0$-Morley sequence over $M$, i.e. $b_i\indi 0_M b_{<i}$ for all $i<\omega$.

By \ref{EXT} there exists $a'$ such that $a'\equiv_{Cb} a$ and $a' \indi M ^*_{C}   b M$. It follows from \ref{BMON} and right \ref{MON} that
    $$a' \ind_{M} b.$$

For each $i\geq 0$ there exists an automorphism $\sigma_i$ over $M$ sending $b = b_0$ to $b_i$, so setting $a_i' = \sigma_i(a')$ we have:
$a_i'b_i \equiv_{M} a'b$ hence by invariance $a_i'\ind_{M} b_i$. Note that $a'b\equiv_C ab$.

\begin{claim}
    There exists $a''$ such that $a''b_i \equiv_{M} a'b$ for all $ i<\omega$.
\end{claim} 

\begin{proof}[Proof of the claim]
    By induction and compactness, it is sufficient to show that for all $i<\omega$, there exists $a_i''$ such that for all $k\leq i$ we have $a''_i b_k \equiv_{M} a' b$ and $a_i''\ind_{M} b_{\leq i}$. For the case $i = 0$ take $a''_0 = a'$. Assume that $a_i''$ has been constructed, we have 
$$ a_{i+1}' \ind_{M}b_{i+1}\text{ and } b_{ i+1} \indi{0}_{M} b_{\leq i}  \text{ and }  a_i''\ind_{M}b_{\leq i}.$$

As $a_{i+1}' \equiv_{M} a_i''$, by $\indi{0}$-amalgamation over models, there exists $a_{i+1}''$ such that 
\begin{enumerate}
    \item $a_{i+1}''b_{i+1} \equiv_{M} a_{i+1}'b_{i+1}$
    \item $a_{i+1}'' b_{\leq i}\equiv_{M} a_i''b_{\leq i}$
    \item $a_{i+1}'' \ind_{M} b_{\leq i+1}$.
\end{enumerate}
By induction and compactness, there exists $a''$ such that $a''b_i \equiv_{M} a b$ for all $i<\omega$, which proves the claim.
\end{proof}

Let $a''$ be as in the claim, then as $a'b\equiv_C ab$ we have $a''b_i \equiv_C ab$ for all $i<\omega$, hence $a\indi{d}_C b$.
\end{proof}

\begin{remark}
  Observe that the relations $\ind$ and $\indi 0$ in the previous result may not be symmetric. For instance, it could be that $\indi 0 = \indi h$ (see below). Then, the parameters $a$ and $b$ in the statement of $\indi 0$-amalgamation do not play a symmetrical role a priori. However the role of $(c_1,a)$ and $(c_2,b)$ are symmetric hence if a relation satisfies $\indi h$-amalgamation, it means that $tp(c_1/Ma)$ and $tp(c_2/Mb)$ can be amalgamated \textit{whenever $a\indi 0 _M b$ or $b\indi 0 _M a$} (i.e. $a\indi 0^\opp _M b$). It follows that $\indi 0$-amalgamation and $\indi 0^\opp$-amalgamation are the same. Observe that the property $\ind$-amalgamation has a contravariant behaviour: if $\indi 1\rightarrow \indi 2$, then if a relation satisfies $\indi 2$-amalgamation, it also satisfies $\indi 1$-amalgamation. In particular, if $\indi 0\to \ind$ then if $\ind$ satisfies \ref{INDTHM}, it satisfies $\indi 0$-amalgamation. 
\end{remark}

\begin{corollary}\label{cor:forcingBMON+EXTforking}
    Some consequences of Theorem \ref{thm:forcingBMON+EXTforking}.
    \begin{enumerate}
    \item Let $\ind$ be an invariant relation which satisfies left and right \ref{MON} and $\indi h$-amalgamation over models.
  Then $\indi{M}^{*} \rightarrow \indi f$. 
    \item Let $\ind$ be an invariant relation, which satisfies left and right \ref{MON}, right \ref{BMON}, \ref{LOC} and \ref{INDTHM} over models. Then $\indi * \rightarrow \indi f$. 
    \item Let $\ind$ be an invariant relation satisfying left and right \ref{MON}, \ref{EXT} and \ref{STAT} over models, then $\indi M ^*\to \indi f$.
    \end{enumerate}
\end{corollary}
\begin{proof}
    $(1)$ From Proposition \ref{prop:propertiesofindiu} and Theorem \ref{thm:modelheirindependence}, the relation $\indi h$ satisfies \ref{MON}, right \ref{BMON} and \ref{LOC}, hence we apply Theorem \ref{thm:forcingBMON+EXTforking}.

    $(2)$ We apply Theorem \ref{thm:forcingBMON+EXTforking}. As $\ind$ satisfies right \ref{MON}, \ref{BMON} and \ref{LOC} we consider $\indi 0 = \ind$. Observe that $\ind$-amalgamation is the same as \ref{INDTHM} hence by Theorem \ref{thm:forcingBMON+EXTforking} we get $\indi M ^* \to \indi f$. As $\ind$ satisfies \ref{BMON} we have $\indi M = \ind$, so we conclude. 

    $(3)$ As $\ind$ satisfies right \ref{MON}, \ref{EXT} and  \ref{STAT} over models, it satisfies $\indi 0$-amalgamation for \textit{any} independence relation $\indi 0$. Pick any $\indi 0$ satisfying right \ref{MON}, right \ref{BMON} and \ref{LOC} (for instance $A\indi 0 _C B$ iff $A\cap B\seq C$) and apply Theorem \ref{thm:forcingBMON+EXTforking}.
\end{proof}
\color{blue}
\begin{remark}
    By Corollary \ref{cor:twomonsamextension}, if $\ind$ further satisfies left and right \ref{NOR}, both Theorem \ref{thm:forcingBMON+EXTforking} and Corollary \ref{cor:forcingBMON+EXTforking} hold for $\indi m ^*$ in place of $\indi M ^*$.
\end{remark}
\color{black}

\begin{exercise}
    If $\ind$ satisfies \ref{SYM}, \ref{EX}, \ref{MON},  \ref{BMON}, \ref{EXT}, \ref{INDTHM} and \ref{SFIN} then $\ind \to \indi f$.
\end{exercise}

\begin{exercise}
    In DLO, define $A\indi > _C B$ if and only if $a> b$ for all $a\in A$, $b\in BC$. Check that $\indi >$ satisfies left and right \ref{MON}, right \ref{BMON}, \ref{EXT} and \ref{STAT} (over every set). Deduce that $\indi >\to \indi f$. Is it faster than proving that the formula $x> b$ does not fork over $C$ for any $C$?
\end{exercise}

\section{Simple theories}
\subsection{Simple theories and dividing}

We denote by $\omega^{<\omega}$ the set of all finite sequences of finite ordinals, considered as a tree with infinite branching at each node. Given $s,t\in \omega^{<\omega}$, we write $s\leq t$ if $s$ is a prefix of $t$. We denote by $\omega^\omega$ the set of all sequences of finite ordinals of length $\omega$.
\begin{definition}[Shelah]\label{def:simplicity})
    Let $k\geq 2$. A formula $\phi(x,y)$ has the \textit{$k$-tree property} if there exists $(b_\nu)_{s\in \omega^{<\omega}}$ such that
    \begin{enumerate}[$(a)$]
        \item $\set{\phi(x,b_s)\mid s\seq \nu}$ is consistent for each $\nu\in \omega^\omega$;
        \item $\set{\phi(x,b_{s^\frown i})\mid i<\omega}$ is $k$-inconsistent, for all $s\in \omega^{<\omega}$.
    \end{enumerate}
    A formula has the \textit{tree property} (TP) if it has the $k$-tree property for some $k<\omega$.

    A theory $T$ is \textit{simple} (or NTP\footnote{Nobody uses the determination NTP. The only interest in this notion is the following. There exists two notions: the \textit{tree property of the first kind} (TP$_1$) and the \textit{tree property of the second kind} (TP$_2$) such that a formula has the TP if and only if this formula has the TP$_1$ or the TP$_2$ (we also say \textit{is} TP$_1$ or \textit{is} TP$_2$). Then a theory is NTP$_i$ if no formula in $T$ is TP$_i$. In this sense, a theory is NTP (simple) if an only if it is NTP$_1$ and NTP$_2$.}) if no formula has the tree property.
\end{definition}

\begin{example}
    In DLO, the formula $\phi(x,yz)$ given by $y<x<z$ has the $2$-tree property: start with $(b_ic_i)_{i<\omega}$ with $c_0<b_0<c_1<b_1<...$. Inductively if $b_s c_s $ has been constructed, we find $\set{b_{s^\frown j}c_{s^\frown j}\mid j<\omega}$ such that 
    \[b_s<b_{s^\frown 0}<c_{i^\frown 0}<b_{i^\frown 1}<c_{i^\frown 1}<\ldots<c_s.\]
    Then $(b_sc_s)_{s\in \omega^{<\omega}}$ is a witness of $2$-TP for $\phi(x,yz)$.
\end{example}

\begin{remark}
\begin{itemize}
    \item If a formula $\phi(x,y,c)$ has the tree property, witnessed by a tree of parameters $(b_{s})_{s\in \omega^{<\omega}}$, then the formula $\phi(x; yz)$ has the tree property with the tree $(b_sc)_{s\in \omega^{<\omega}}$. Hence for simplicity, it suffices to check that no formula without parameters has the tree property.
    \item Any theory definable (even interpretable) in a simple theory is again simple. This is because simplicity is defined at the level of formulas. In particular, any reduct of a simple theory is again simple. 
\end{itemize}
\end{remark}

% \begin{example}
%     The theory of Skolem arithmetic $(\N, \cdot)$ is not simple. Note that the divisibility relation is definable from the multiplication: $x\mid y$ if and only if $\exists z (xz = y)$. So $(\N,\mid)$ is a reduct of $(\N,\cdot)$. Let $\cP = \set{p_i\mid i<\omega}$ be the set of primes and define the following tree: $b_s:= \prod_{i\in dom(s)} p_{s(i)}$ for each $s\in \omega^{<\omega}$. Then the formula $\set{x\mid b_s\  \mid  s\in \omega^{<\omega}}$ has the $2$-tree property hence the theory of $(\N,\mid )$ is not simple. As $x\mid y$ is also a formula in $(\N,\cdot)$, the theory of $(\N,\cdot)$ is not simple.
% \end{example}

\begin{remark}\label{remark:wideningofthetreeTP}
    Assume that $\phi(x,y)$ has the $k$-tree property, witnessed by $b = (b_{s})_{s\in \omega^{<\omega}}$. Then the tuple $b$ satisfies the (partial) type $\Sigma^{\omega,\omega}((y_s)_{s\in \omega^{<\omega}})$ given by the union of (every branch is consistent)
    \[\bigcup_{\nu\in \omega^\omega} \set{\exists x \phi(x,y_{s_1})\wedge \ldots \wedge \phi(x,y_{s_n})\mid s_1,\ldots,s_n\seq \nu, n<\omega}\]
    and (every level is $k$-inconsistent)
    \[\bigcup_{s\in \omega^{<\omega}} \set{\neg \exists x \bigwedge_{j = 1}^k \phi(x, y_{s^\frown i_j}) \mid i_1<\ldots<i_k<\omega}\]
    Conversely, if this type is consistent then the formula $\phi(x,y)$ has the $k$-tree property. Note the importance of having $k$-inconsistency instead of inconsistency on each level otherwise the tree property would not be expressable as a type. From $\Sigma^{\omega,\omega}$, one derives easily a type $\Sigma^{\kappa,\mu}$ which witnesses the tree property for a tree shaped like $\kappa^{<\mu}$, for some limit ordinal $\mu$. To do so, simply by add more variables and change $\omega^{<\omega},\omega^\omega$ by $\kappa^{<\mu},\kappa^\mu$. Then by compactness, the formula $\phi(x,y)$ has the tree property if an only if $\Sigma^{\kappa,\mu}$ is consistent, for any infinite $\kappa$ and limit ordinal $\mu$.
\end{remark}

\begin{definition}
    A \textit{$\phi$-k dividing sequence over $C$, of length $\mu$} is a sequence $(\phi(x,b_i))_{i<\mu}$ such that $\phi(x,b_i)$ divides over $Cb_{<i}$ and $\bigwedge_{i<\mu}\phi(x,b_i)$ is consistent.
\end{definition}

\begin{lemma}\label{lm:TPequivtophi-kdividingsequence}
    $\phi(x,y)$ has the $k$-tree property if and only if there exists arbitrary long $\phi$-k dividing sequence over $\emptyset$.
\end{lemma}
\begin{proof}
    Assume that $\phi(x,y)$ has the $k$-tree property and $\mu$ is given. We may assume that $\mu$ is a limit ordinal. %take a limit ordinal bigger and restrict
    By Remark \ref{remark:wideningofthetreeTP} there exists an infinite tree of parameters $(b_s)_{s\in \kappa^{<\mu}}$ for some $\kappa> 2^{\max\set{\abs{T},\mu}}$. We choose recursively a branch $\nu\in \kappa^\mu$ such that the sequence of $\phi(x,b_{s})$ for initial $s\seq \nu$, forms a $\phi$-k dividing sequence. Let $n = \abs{y}$. As $\kappa>2^{\max \set{\abs{T}}}$ and $\abs{S_{n}(\emptyset)}\leq 2^{\max \set{\abs{T}}}$ there are infinitely many $b_{i}$'s (for $i<\kappa$, in the lowest level of the tree) which have the same type over $\emptyset$, so choose one of the indexes of those $i_0$. Clearly $\phi(x,b_{i_0})$ is a $k$-dividing formula over $\emptyset$. By induction, assume that $s = i_0^\frown \ldots ^\frown i_j$ is such that $\phi(x,b_s)$ divides over $b_{i_0}, \ldots, b_{i_0^\frown \ldots ^\frown i_{j-1}}$ and consider elements $(b_{s^\frown i})_{i<\kappa}$. As $\kappa>2^{\max\set{\abs{T},\mu}}> \abs{S_n(b_{i_0}, \ldots, b_{s}})$, infinitely many of $(b_{s^\frown i})_{i<\kappa}$ have the same type over $b_{i_0}, \ldots, b_{s}$ hence choose the index $i_{j+1}$ of one of those an set $s' = s^\frown i_{j+1}$. Then $\phi(x,b_{s'})$ divides over $b_{i_0},\ldots,b_s$. The limit case is done similarly. By induction, we find a branch $\nu\in \kappa^\mu$ such that the sequence $(\phi(x,b_{\nu\upharpoonright \set{0,\ldots,i}})_{i<\mu}$ is a $\phi$-k dividing sequence.

    Conversely, assume that there is an infinite a $\phi(x,y)$-k dividing sequence $(\phi(x,b_i))_{i<\omega}$. For each $i<\omega$, let $(b_i^n)_{n<\omega}$ be a sequence in $\tp(b_i/b_{<i})$ with $b_i^0 = b_i$ witnessing that $\phi(x,b_i)$ divides over $b_{<i}$. We define a tree of parameters $(c_s)_{s\in \omega^{<\omega}}$ as follows. Start by defining $c_0 = b_0^0, c_1 = b_0^1, c_n = b_0^n$ for each $n<\omega$, which will be the first level of the tree. We describe level 2: define $c_{0n} := b_1^n$ and for each $n<\omega$. There is an automorphism $\sigma_n$ sending $b_0^0$ to $c_n = b_0^n$ and define $c_{ni}:= \sigma_n(b_1^i)$, so that the level above $c_n$ is $\sigma_n((b_1^i)_{i<\omega})$. Once a tree $(c_{s})_{s\seq \omega^{<m}}$ of height $m$ has been constructed, such that $c_{0\ldots 0} = b_m^0$ and for any $s$, $c = (c_{s(0)},\ldots,c_{s(m)})\equiv (b_0,\ldots,b_n)$, witnessed by an automorphism $\sigma_s(b_0,\ldots,b_m) = c$. Then define $c_{s^\frown n} := \sigma_s(b_{m+1}^n)$, so that the level above $c_s$ is the image of the level above $b_m$. This ensures that set of formula along the branches is consistent, and that the levels are $k$-inconsistent. 
\end{proof}

\begin{theorem}\label{thm:simplecharacterisationLOC}
    $T$ is simple if and only if $\indi d$ satisfies \ref{LOC}.
\end{theorem}
\begin{proof}
    By Lemma \ref{lm:TPequivtophi-kdividingsequence}, it is enough to prove that $\indi d$ satisfies \ref{LOC} if and only if there exists no infinite $\phi$-k dividing sequence.

    Assume that $(\phi(x,b_i))_{i<\kappa}$ is a $\phi$-k dividing sequence, for some regular cardinal $\kappa$ and $a\models \bigwedge_i \phi_i(x,b_i)$. Let $C\seq B:= \set{b_i\mid i<\kappa}$ with $\abs{C}<\kappa$. Then there exists $i<\kappa$ such that $C\seq b_{<i}$ so that $\phi(x,b_i)$ divides over $b_{<i}$ and hence $\phi(x,b_i)$ divides over $C$. It follows that for all $C\seq B$ such that $\abs{C}<\kappa$ we have $a\nindi d _{C} B$ hence \ref{LOC} fails.

    Conversely, assume that $\indi d$ fails \ref{LOC}. Hence there exists a finite tuple $a$, a cardinal $\kappa$ and a set $B$ such that $a\nindi d_{C} B$ for all $C\seq B$ with $\abs{C}\leq \kappa$ \footnote{In all generality, \ref{LOC} is equivalent to the same statement assuming $\leq$ instead of $<$.}. As $a\indi d _B B$ we have $\abs{B}> \kappa$. We construct a sequence of formula $(\phi_i(x,b_i))_{i<\kappa^+}$ as follows. Start with any $b_0\seq B$ such that $a\nindi d_\emptyset b_0$ witnessed by a $k_0$-dividing formula $\phi_0(x,b_0)$, for some $k_0<\omega$. For any $i<\kappa^+$, if $b_{<i}$ has been constructed, then $\abs{b_{<i}}\leq \kappa$ hence there exists $b_i$ such that $a\nindi d _{b_{<i}} b_i$ and a formula $\phi_i(x,b_i)$ which $k_i$-divides over $b_{<i}$. As $\kappa^+>\abs{T}$, there is an infinite (of size $\kappa^+$) subsequence with the same $\phi$, and again an infinite subsequence with the same $k_i$, which is an infinite $\phi$-$k$ dividing sequence.
\end{proof}

%exo on the rank?

\subsection{The Kim-Pillay theorem}
Let $T$ be a theory and $\MM$ a monster model of $T$.

\begin{theorem}\label{thm:KPgeneral}
    Assume that there exists an invariant relation $\ind$ satisfying left and right \ref{MON}, right \ref{BMON}, \ref{LOC}, \ref{EXT} and \ref{INDTHM} over models, then:
    \begin{itemize}
        \item $\ind\to \indi f$
        \item $T$ is simple;
        \item $\indi f = \indi d$;
        \item $\indi f$ is symmetric.
    \end{itemize} 
    Assume further that $\ind$ satisfies \ref{FIN}, left \ref{NOR} and left \ref{TRA}, then $\ind = \indi f = \indi d$ and $\ind$ is symmetric.
\end{theorem}

\begin{proof}
    Under the first assumptions, we have by Corollary \ref{cor:forcingBMON+EXTforking} $(2)$ that $\indi *\to \indi f$. As $\ind$ also satisfies \ref{EXT}, we have $\indi * = \ind$ hence $\ind \to \indi f$. As $\ind$ satisfies \ref{LOC}, so does $\indi f$ hence by Corollary \ref{cor:forkingisanairiffLOC}, $\indi f$ is an AIR and $\indi f = \indi d$ and $\indi f$ satisfies \ref{SYM}. Theorem \ref{thm:simplecharacterisationLOC} $T$ is simple since $\indi d$ satisfies \ref{LOC}. If $\ind$ satisfies the other properties, then $\ind$ is also an AIR hence satisfies \ref{SYM} and $\ind = \indi d$ follows from Proposition \ref{prop:dividingindepstrongerthanAIR}.
\end{proof}

\begin{corollary}[The Kim-Pillay theorem]\label{cor:KIMpillayoriginal}
    Assume that there exists an invariant relation $\ind$ satisfying \ref{SYM}, \ref{NOR}, \ref{MON}, \ref{BMON}, \ref{TRA}, \ref{FIN}, \ref{LOC}, \ref{FEX} and \ref{INDTHM} over models. Then $T$ is simple and $\ind = \indi f$. 
\end{corollary}

\begin{remark}
    The statement of Corollary \ref{cor:KIMpillayoriginal} is not the full statement of what people generally call the Kim-Pillay theorem. The full version is a characterisation of simplicity: $T$ is simple \textit{if and only if} there is an independence relation $\ind$ satisfying all those axioms. To recover the full statement, it remains to prove that if $T$ is simple, then $\indi f = \indi d$ and that $\indi f$ satisfies the independence theorem. The Kim-Pillay characterisation of simple theories is a deep result as it yields an equivalence between two \textit{a priory} very different definitions in nature. Let us rephrase the Kim-Pillay result: a theory is simple if and only if there exists an AIR which satisfies \ref{INDTHM}, and if so, it is unique (and equals $\indi f$). The local, or syntactic Definition \ref{def:simplicity} and the existence of an AIR satisfying \ref{INDTHM} are properties of very different nature. One nontrivial corollary of this equivalence is that the existence of an AIR satisfying \ref{INDTHM} is closed under reduct, which is, at first sight, not clear at all.
\end{remark}

\begin{remark}
    The Kim-Pillay theorem is a crucial result in model theory and is very useful from a practical point of view. Assume that you stumble upon some theory in the bend of a road and that you would like to know whether this theory is simple. You then have a choice between checking that no formula $\phi(x,y)$ has the tree property, for this you would need a very strong understanding of definable sets, probably a strong quantifier elimination result. The other solution is to find a ternary relation $\ind$ and check that it satisfies the axioms of the Kim-Pillay theorem. The main obstacle in general is the independence theorem but it is still much easier to prove in general than checking TP. If you succeed in using the Kim-Pillay theorem, you would know much more about this independence relation, it coincides with $\indi f$ then you would have a new definition of $\ind$, in terms of formulas. 
\end{remark}

\begin{exercise}
    Assume that there exists an invariant relation $\ind$ which satisfies left and right \ref{MON}, right \ref{BMON}, \ref{LOC}, \ref{EXT} and \ref{INDTHM} over models, then $T$ is simple. 
\end{exercise}

\begin{exercise}
    Assume that there exists an invariant relation $\ind$ satisfying \ref{SYM}, \ref{EX}, \ref{MON}, \ref{BMON}, \ref{SFIN} and $\indi u$-amalgamation over models. If $\indi *$ is symmetric , then:
    \begin{itemize}
        \item $\indi * \to \indi f$
        \item $T$ is simple;
        \item $\indi f = \indi d$;
        \item $\indi f$ is symmetric.
    \end{itemize} 
    Assume further that $\ind$ satisfies \ref{NOR} and \ref{TRA}, then $\indi * = \indi f = \indi d$. %Loca character comes from \indi h\to \ind and as \indi * is symmetric we should have \indi u\to \indi *, and hence is \indi * is also symmetric we have $\indi h \to \indi *$ hence it satisfies LOC and etc.
\end{exercise}

\subsection{Back to the examples}

\begin{corollary}
    ACF and RG are simple theories and the forking independence relations are given respectively by $\indi \alg\ \ $ and $\indi a$.
\end{corollary}

\begin{proof}
    Using the Kim-Pillay theorem (Corollary \ref{cor:KIMpillayoriginal}) with $\ind = \indi\alg\ \ $ for ACF (Propositions \ref{prop:ACFalgebraicsatisfiesbasics} and \ref{prop:ACFsatisfiesINVEXTSTAT}) and $\ind = \indi a$ for RG (Propositions \ref{prop:algebraicindepRGbasicproperties} and \ref{propositionpropertiesindepRGwithtypes}).
\end{proof}

\begin{remark}
    Observe that concerning RG, we used two different independence relations to prove that it is NSOP$_4$ ($\indi \st$, see Corollary \ref{cor:RGNSOP4}) or that it is simple. 
\end{remark}

\color{blue}{

\section{A second criterion for NSOP$_4$ theories}

\subsection{Theories that do not admit a stationary independence relation} The criterion provided in Theorem \ref{thm:criterionNSOP4} relies on the existence of an independence relation $\ind$ satisfying \ref{STAT} over models, which is a rather strong hypothesis. The following lemma is useful to prove that some theories do not admit such independence relation.

\begin{proposition}
    Let $\ind$ be an invariant independence relation satisfying: \ref{AREF}, \ref{SYM}, \ref{FEX} and \ref{STAT} over models. Then for all $a\notin \acl(\emptyset)$ there exists $b\neq a$ such that $ab\equiv ba$. In particular, there is no definable relation $R(x,y)$ which is both
    \begin{itemize}
        \item total $(\forall xy \ R(x,y)\vee R(y,x))$ and 
        \item antisymmetric $(\forall xy\ R(x,y)\wedge x\neq y\to \neg R(y,x))$.
    \end{itemize}
\end{proposition}
\begin{proof}
    Let $a$ be any singleton or tuple and let $M$ be a model such that $a\notin M$. By \ref{FEX}, there exists $b\equiv_M a$ such that $b\ind_M a$. By \ref{AREF}, we have $a\neq b$. By \ref{SYM} and \ref{STAT} over models, we conclude $ab\equiv_M ba$, see Lemma \ref{lm:stat_invert}. In particular, $ab\equiv ba$ holds. If such a relation $R$ were definable, then we would have both $R(a,b)$ and $R(b,a)$, a contradiction.
\end{proof}
\begin{remark}
    In the previous Proposition, it is enough for the relation $R$ to be type-definable, or even invariant.
\end{remark}

\begin{example}
    It follows that any theory in which a total order is definable (DLO, Presburger arithmetic) does not admit an invariant independence relation satisfying \ref{AREF}, \ref{SYM}, \ref{FEX} and \ref{STAT} over models.
\end{example}

\begin{example}[A simple theory with no stationary independence relation]
    For a given prime $p$, let $\F_p$ denotes the prime field of characteristic $p$. We denote by $\F_p^\times$ the multiplicative group $\F_p\setminus\set{0}$ and 
    \[\sq(\F_p^\times) = \set{x^2\mid x\in \F_p^\times}.\]
    It is classical that $\sq(\F_p)$ is a subgroup of $\F_p^\times$ of index $2$ (apply the isomorphism theorem with the multiplicative homomorphism $x\mapsto x^2$). Assume further that $-1\notin \sq(\F_p)$, which is classically equivalent to $p\notin 4\Z+1$. Then we have
    \[\F_p^\times = \sq(\F_p^\times)\cup -\sq(\F_p^\times).\quad \quad (\star)\]
    Now let $F$ be an ultraproduct of finite fields $\F_p$ such that $p\notin 4\Z+1$. As $(\star)$ is a first-order property of fields, we have $F^\times =  \sq(F^\times)\cup -\sq(F^\times)$ and hence the relation $x-y\in \sq(F^\times)$ is total and antisymmetric so there is no invariant relation satisfying \ref{AREF}, \ref{SYM}, \ref{FEX} and \ref{STAT} over models. We will see later that the theory of $F$ is simple.
\end{example}

}

\subsection{Independent amalgamation and NSOP$_4$} The independence theorem can be seen as a weak form of stationarity over models, 

\begin{theorem}[Second criterion for NSOP$_4$]\label{thm:criterionNSOP4second}
Let $\ind$ be an invariant relation satisfying \ref{SYM}, \ref{FEX} and the following two properties:
\begin{itemize}
    \item (Weak transitivity over models) if $a\ind_{Md} b$, $ a\indi h _M d$ and $b\indi u _M d$ then $a\ind_M b$.
    \item (Weak independence theorem over models) if $c_1\equiv_M c_2$ and $c_1\indi h_M a$, $c_2\indi u_M b$ and $a\ind_M b$
      then there exists $c$ with $c\equiv_{Ma} c_1$, $c \equiv_{Mb} c_2$.
\end{itemize} 
Then $T$ is NSOP$_4$.
\end{theorem}

\begin{proof}
    The proof starts as in the one of Theorem \ref{thm:criterionNSOP4}. Let $(a_i)_{i<\omega}$ be an indiscernible sequence and for $p(x,y) = \tp(a_0,a_1)$, we need to prove that 
    \[p(x_0,x_1)\cup p(x_1,x_2)\cup p(x_2,x_3)\cup p(x_3,x_0)\]
    is consistent. By Theorem \ref{thm:modelheirindependence}, there is a model $M$ such that $a_i\indi h_M a_{<i}$ for all $i$. By \ref{FEX}, there exists $a_2^*\equiv_{Ma_1} a_2$ such that $a_2^*\ind_{Ma_1} a_0$. Fom $a_2\indi h _M a_1$ and $a_1\indi h_M a_0$ we get $a_2^*\indi h _M a_1$ by invariance and $a_0\indi u_M a_1$ by definition. Using weak transitivity, we conclude $a_2^*\ind_M a_0$. As $a_0\equiv_M a_1$ there exists $c_1$ such that $c_1a_0\equiv_M a_0a_1$ and by invariance, $c_1\indi u_M a_0$. Similarly, as $a_2^*\equiv_M a_1$ there exists $c_2$ such that $a_1a_2^*\equiv_M a_2^*c_2$ and by invariance $c_2\indi h_M a_2^*$. We have $c_1\equiv_M c_2$, $c_1\indi u _M a_0$, $c_2\indi h _M a_2^*$ hence by the weak independence theorem over models, we conclude that there exists $a_3^*$ such that $a_3^*a_0\equiv_M c_1a_0\equiv_M a_0a_1$ and $a_3^*a_2^*\equiv_M c_2a_2^*\equiv_M a_2^*a_1\equiv_M a_2a_1$. As $a_0a_1\equiv a_1a_2$, we conclude that 
    \[a_0a_1\equiv a_1a_2^*\equiv a_2^* a_3^*\equiv a_3^*a_0\]
    hence the type above is consistent.
    \end{proof}

\begin{remark}
    Note the analogy between ``weak independence theorem over models" and the ``$\indi 0$-amalgamation over models" of Theorem \ref{thm:forcingBMON+EXTforking}:
    \begin{center}
    if $c_1\equiv_M c_2$ and $c_1\ind_M a$, $c_2\ind_M b$ and $a\indi 0 _M b$
    then there exists $c$ with $c\ind_E ab$ and $c\equiv_{Ma} c_1$, $c \equiv_{Mb} c_2$.
\end{center}
    In the weak independence relation, there is no need for $c\ind_M ab$ in the conclusion. Note also that $c_1$ and $c_2$ can be permuted hence so are $\indi h$ and $\indi u$.
\end{remark}

\begin{corollary}
    Assume that there is an invariant independence relation $\ind$ which satisfies \ref{SYM}, \ref{MON}, \ref{BMON}, \ref{TRA}, \ref{LOC} and \ref{INDTHM} over models, then $T$ is NSOP$_4$.
\end{corollary}
\begin{proof}
This is obtained via Lemma \ref{lm:LOCcharactergivesMS} and the proof of Theorem \ref{thm:criterionNSOP4second}.
\end{proof}

\begin{remark}
    It was quite frustrating that our first criterion for NSOP$_4$ could not be used for proving that simple theories are NSOP$_4$. This is possible with this second criterion, assuming that the converse of Corollary \ref{cor:KIMpillayoriginal} holds.
\end{remark}

\begin{exercise}
Using Lemma \ref{lm:LOCcharactergivesMS}, prove that the statement of Theorem \ref{thm:criterionNSOP4second} stays true by replacing every occurrence of $\indi u,\indi h$ by $\indi 0, \indi 0 ^\opp$, for $\indi 0$ and invariant relation satisfying right \ref{MON}, right \ref{BMON} and \ref{LOC}.
\end{exercise}

\begin{exercise}
    Is the converse of Theorem \ref{thm:criterionNSOP4second} true?
\end{exercise}

%% file: main.bbl
\begin{thebibliography}{10}

\bibitem{adlerthesis}
Hans Adler.
\newblock Explanation of independence. {D}issertation zur {E}rlangung des
  {D}oktorgrades der {F}akultaet fuer {M}athematik und {P}hysik der
  {A}lbert-{L}udwigs-{U}niversitaet {F}reiburg im {B}reisgau, 2005.

\bibitem{A09}
Hans Adler.
\newblock A geometric introduction to forking and thorn-forking.
\newblock {\em J. Math. Log.}, 9(1):1--20, 2009.

\bibitem{C11}
Enrique Casanovas.
\newblock {\em Simple theories and hyperimaginaries}, volume~39 of {\em Lecture
  Notes in Logic}.
\newblock Association for Symbolic Logic, Chicago, IL; Cambridge University
  Press, Cambridge, 2011.

\bibitem{CR16}
Artem Chernikov and Nicholas Ramsey.
\newblock On model-theoretic tree properties.
\newblock {\em J. Math. Log.}, 16(2):1650009, 41, 2016.

\bibitem{gabefreeamalgamation}
Gabriel Conant.
\newblock An axiomatic approach to free amalgamation.
\newblock {\em J. Symb. Log.}, 82(2):648--671, 2017.

\bibitem{conant2022separation}
Gabriel Conant and James Hanson.
\newblock Separation for isometric group actions and hyperimaginary
  independence, 2022.

\bibitem{conant2023surprising}
Gabriel Conant and Alex Kruckman.
\newblock Three surprising instances of dividing, 2023.

\bibitem{halevi2021saturated}
Yatir Halevi and Itay Kaplan.
\newblock Saturated models for the working model theorist, 2021.

\bibitem{H08}
Wilfrid Hodges.
\newblock {\em Model theory}.
\newblock Encyclopedia of Mathematics and its Applications. Cambridge
  University Press, 1 edition, 2008.

\bibitem{KR17}
Itay Kaplan and Nicholas Ramsey.
\newblock On {K}im-independence.
\newblock {\em Journal Of The European Mathematical Society},
  22(5):1423–1474, 2020.

\bibitem{KP97}
Byunghan Kim and Anand Pillay.
\newblock Simple theories.
\newblock {\em Annals of Pure and Applied Logic}, 88(2):149 -- 164, 1997.
\newblock Joint AILA-KGS Model Theory Meeting.

\bibitem{TZ12}
Katrin Tent and Martin Ziegler.
\newblock {\em A course in model theory}, volume~40 of {\em Lecture Notes in
  Logic}.
\newblock Association for Symbolic Logic, La Jolla, CA; Cambridge University
  Press, Cambridge, 2012.

\end{thebibliography}
